\documentclass[12pt, a4paper]{amsart}
\usepackage[margin=1.2in]{geometry}
\usepackage[normalem]{ulem}
\usepackage[Verbose]{parallel}
\usepackage{amsmath}
\usepackage{amsthm}
\usepackage{bbm}
\usepackage{color}
\usepackage{amsfonts}
\usepackage{hhline}
\usepackage{amssymb}
\usepackage{amscd}
\usepackage{amsmath}

\numberwithin{equation}{section}
\newtheorem{theorem}{Theorem}[section]

\newtheorem{cor}[theorem]{Corollary}
\newtheorem{lemma}[theorem]{Lemma}
\newtheorem{prop}[theorem]{Proposition}

\newtheorem{defi}{Definition}[section]

\theoremstyle{definition}
\newtheorem{example}{Example}[section]
\newtheorem{rem}{Remark}[section]

\renewcommand\ll{\mathfrak l}

\newcommand\uno{\mathbbm 1}
\newcommand\half{\tfrac{1}{2}}

\newcommand\ov{\overline}

\renewcommand\({\left(}

\newcommand\be{\beta}

\newcommand\ka{\widehat \k}

\newcommand\g{\mathfrak g}
\newcommand\ga{\widehat{\mathfrak g}}
\newcommand\h{\mathfrak h}
\newcommand\q{\mathfrak q}
\newcommand\ha{\widehat{\mathfrak h}}

\newcommand\bb{\mathfrak b}
\newcommand\D{\Delta}
\renewcommand\l{\lambda}
\newcommand\Dp{\Delta^+}

\newcommand\Da{\widehat\Delta}
\newcommand\Pia{{\widehat\Pi}}
\newcommand\Dap{\widehat\Delta^+}
\newcommand\Wa{\widehat{W}}
\renewcommand\d{\delta}
\renewcommand\r{\mathfrak r}

\renewcommand\a{\alpha}
\renewcommand\aa{\mathfrak a}
\renewcommand\b{\bb}

\renewcommand\k{\mathfrak k}

\renewcommand\i{{\mathfrak  i}}

\newcommand\nat{\mathbb N}
\newcommand\ganz{\mathbb Z}

\newcommand\s{\sigma}
\renewcommand\L{\Lambda}
\renewcommand\aa{\mathfrak a}

\renewcommand\u{{\mathfrak u}}

\newcommand\y{{\bf y}}

\newcommand\e{\epsilon}
\newcommand\C{\mathbb C}
\newcommand\R{\mathbb R}

\newcommand\va{|0\rangle}

\renewcommand\ha{\widehat{\mathfrak h}}

\newcommand{\fa}{\mathfrak{a}}

\newcommand{\fg}{\mathfrak{g}}

\newcommand{\fk}{\mathfrak{k}}

\newcommand{\ft}{\mathfrak{t}}
\newcommand\p{\mathfrak p}

\newcommand{\CC}{\mathbb{C}}
\newcommand{\fp}{\mathfrak{p}}

\newcommand{\Wab}{\mathcal W_\s^{ab}}

\begin{document}
\title{Conformal embeddings and simple current extensions}
\author[Kac, M\"oseneder, Papi, Xu]{Victor~G. Kac}
\author[]{Pierluigi M\"oseneder Frajria}
\author[]{Paolo  Papi}
\author[]{Feng Xu}
\begin{abstract} In this paper we investigate the structure of  intermediate vertex  algebras associated with  a maximal conformal embedding
of a reductive Lie algebra in a semisimple Lie algebra of classical type.
\end{abstract}
\maketitle
\tableofcontents
\section{Introduction}
Let $\g$ be a semisimple finite-dimensional complex Lie algebra and $\k$  a reductive subalgebra of $\g$. The embedding $\k\hookrightarrow \g$ is  called conformal    if the central charges of the Sugawara construction of the Virasoro algebra for
the affinizations $\ga,\ka$ are equal.
Conformal embeddings  were popular in physics literature in the mid 80's, due to their relevance for string compactifications. In particular, maximal conformal embeddings were classified
in \cite{SW},  \cite{AGO}. In the  vertex algebra framework the definition can be rephrased as follows: the simple affine vertex  algebras $V_{\bf k}(\k)$ and $V_1(\g)$ have the same
Sugawara conformal vector for some multiindex $\bf k$ of levels.\par
Recall that  an irreducible  module $M$  over a vertex algebra $V$ is called a {\it simple current} if for any irreducible $V$-module $M_2$ there exists a unique irreducible $V$-module $M_3$ such that the space of intertwiners $\left[\begin{matrix}M_3\\M\ M_2\end{matrix}\right]$ is one-dimensional and $\left[\begin{matrix}S\\M\ M_2\end{matrix}\right]=\{0\}$ for any irreducible $V$-module $S$ not isomorphic to $M_3$.
   A simple current extension of $V$ is a simple vertex operator algebra $W$ containing  $V$ such that  $V$ and  $W$ have the same Sugawara conformal vector and there is grading $W=\sum_{a\in D} W^a$ of $W$ by  an abelian group  $D$ such that $W^0=V$ and $W^a$ is a simple current for $V$ for any $a\in D$.
\par
Our main result is the following theorem, which appears as Theorem \ref{p} in the body of the paper.
\begin{theorem}\label{pp} Let $\k\hookrightarrow \g$ be a maximal conformal embedding  and $\g$ a simple classical Lie algebra. Assume that $W$ is a simple vertex subalgebra of $V_1(\g)$ such that
$V_{\mathbf{j}}(\k)\subset W$.
Then either $W=V_1(\g)$ or  $W$ is a simple current extension  of  $V^{\bf j}(\k)$. Moreover, all these extensions are explicitly described.
\end{theorem}
As detailed in Remark \ref{exc}, the theorem holds when $\k$ is a regular subalgebra of a simple Lie algebra $\g$ of exceptional type.
We have not been able to settle the case when  $\k$ is not regular. However Remark \ref{conje}  leads us  to think that  the theorem is  true in this case too:\vskip5pt
\noindent
{\bf Conjecture 1.1}: {\it Theorem \ref{pp} holds for any conformal embedding $\k\hookrightarrow  \g$}.

\vskip10pt A motivation for this paper comes from the following
conjecture, which is a modified vertex algebra  version  of
conjecture 1.2 in \cite{F}: \vskip5pt\noindent {\bf Conjecture
1.2}Ê{\sl\ (generalized Wall's conjecture for vertex algebras)}:
{\it Suppose that $B$ is a vertex operator algebra and $A \subset B$
is a vertex operator subalgebra. Assume that $B$ decomposes into a
direct sum of finitely many irreducible representations of $A$. Then
the number of minimal (resp. maximal) vertex subalgebras between $A$
and $B$ is less than $C\cdot (\dim\,End_A(B))^{3/2}$ , where $C$ is
a constant and $End_A(B)$ is the space of linear maps from $B$ to
itself which commute with the action of $A$.}

Consider  the special case when $A$ is the fixed point subalgebra of
a simple vertex operator algebra $D$ under a faithful action of a
finite group $G,$  and $B$ is the fixed point subalgebra of  $D$
under a faithful action of a finite subgroup $H$ of $G$. By the
Galois correspondence (cf. \cite{HMT}, \cite{KacR}), the
intermediate vertex subalgebras between $A$ and $B$ are in one to
one correspondence with subgroups of $G$ which contain $H$, and
minimal/maximal vertex subalgebras correspond to maximal/minimal
subgroups of $G$ which contain $H$. So the minimal version of the
above conjecture in this case  states that the number of maximal
subgroups of $G$ which contain $H$ is less than  $C\cdot
n^{\frac{3}{2}}$ where $n$ is the number of double cosets of $H$ in
$G$. When $H$ is trivial or a normal subgroup of $G$, this is a
theorem of Liebeck, Pyber and Shalev \cite{LPS}. It is believed that
the constant $C$ can be chosen to be $1$.

 At the time of this writing, the
original 1961 conjecture of G.E. Wall (corresponding to bounding the
number of maximal subgroups by $|G|$)  seems to have been disproven.
But the above conjecture is still open even for group cases when the
subgroup $H$ of $G$ is not normal.

The maximal version of the above conjecture in the group case can be
easily verified. For example when $H$ is trivial the conjecture
follows from the fact that the number of minimal subgroups of $G$ is
less than $|G|$, and this can be directly checked as follows: note
that any minimal subgroup of $G$ is a cyclic group of prime order,
and so any two such minimal subgroups intersect only at identity. It
follows that the number of minimal subgroups of $G$ is less than
$|G|$.

In Xu's paper \cite{F}, the above
conjecture was tested for a special class of $A\subset B$ coming from
conformal embeddings. In particular, in Theorem 3.14 of
\cite{F} all simple intermediate vertex algebras  in these cases are listed. The proof
in \cite{F} is a mixture of
techniques from the theory of subfactors and  uses unbounded smeared vertex
operators.  The main theorem  of this paper gives a vast generalization of
Theorem  3.14 of \cite{F} using very different methods.\vskip5pt
 Theorem \ref{pp}  can be nicely illustrated by the case of the maximal conformal embedding $\g\hookrightarrow so(\g)$. This  example does not have the many technical complications of the general case,
but it is enlightening in showing the ideas underlying  the proof. We therefore refer the reader to  Section \ref{duee} for an outline of proof
in this special case. \par
A key tool in our treatment  is given by the decomposition formulas found in \cite{CKMP}. In that paper, starting from a  conformal embedding
$\k\hookrightarrow so(\p)$ associated to an infinitesimal symmetric space $\g=\k\oplus\p$, we found explicit formulas expressing the decomposition of   the basic module of $\ga$ in $\ka$-irreducibles. Also, we provided  a combinatorial interpretation of these decompositions in terms of the abelian subspaces of $\p$ that are stable with respect to the action of a Borel subalgebra of $\k$. \par
In the present paper,
we first generalize the results of \cite{CKMP}Ê to any conformal embedding in a classical simple Lie algebra.
It is very convenient, in view of this generalization, to reformulate our previous result in the framework of affine vertex algebras rather than that of basic modules of affine Lie algebras. This  is done in Theorem \ref{decoeabeliani}. The  generalization to the classical  case is given in Theorem \ref{decoclassical}. With this generalized result available, we proceed to prove
Theorem \ref{pp} along the lines of the adjoint case.\par
A natural question left open by Theorem \ref{pp} is to analyze the structure of the simple current extensions, or, in other terms, to determine the abelian groups parametrizing the extensions. This is done in the last section, where we relate these groups to certain subgroups of the center of the connected simply connected Lie group corresponding to $\k$ (cf. Proposition
\ref{centro}). These subgroups turn out to be
characterized by suitable integrality conditions: see Proposition \ref{integral} and Corollary \ref{cori}.\par
The paper is organized as follows. In Section \ref{duee}Ê we discuss the strategy of proof of Theorem \ref{pp}  in the adjoint case.  Section \ref{tre} is divided in  several subsections. In the first three subsections  we recall a few basic facts on vertex algebras  and perform some calculations in the fermionic vertex algebra needed in what follows. In Section \ref{confemb} we discuss the notion of conformal embedding from several points of view. In \ref{fdp} we highlight the relationships of conformal embeddings with  the finite decomposition property and we also apply   results  by Kac and Wakimoto to classify  certain  instances of finite decomposition  at admissible rational non-integer levels.
Section \ref{rva} contains the material we need about representations of vertex algebras, with  special emphasis on  fusion coefficients and simple currents.
In Section \ref{symmpairs} we collect all the results from the theory of symmetric pairs which will be used in the derivation of the decomposition formulas. In Section \ref{6} we state and prove the decomposition theorems, and in Section \ref{7} we apply them to prove Theorem \ref{pp}.
Section \ref{8}Ê illustrates the final sentence of Theorem \ref{pp}, yielding an explicit description of the groups parametrizing the  simple currents extensions.
\subsubsection*{Notational conventions} Regarding  Dynkin diagrams, we use the conventions for names and numberings of \cite{Kac}, Chapter~4, Tables~Fin, Aff 1-2.Ê\par We denote by $\nat$ the set of nonnegative integers.
\section{The adjoint case}\label{duee}
In this section we sketch the proof of Theorem \ref{pp} in the case of the conformal embedding of a simple Lie algebra $\g$ in $so(\g)$. Precise definitions and references will be given in the next sections.\par
 Let $V_{k}(\g)$ denote the simple affine vertex algebra of level $k\in\nat$ (cf. \ref{ava}).
It is known that  the  irreducible $V_{k}(\g)$-modules with highest  weight  $k \omega$ where $\omega$
is a minuscule fundamental weight of $\g$ are simple currents for $V_{k}(\g)$. We call them {\it special simple currents}. In this section we give an outline of proof of the following special case of Theorem \ref{pp}.
\begin{theorem} \label{T} Let $\g$ be a simple Lie algebra and let $h^\vee$ be its  dual Coxeter number. The only simple vertex operator algebras $V$ such that    $V_{h^\vee}(\g)\subseteq V \subsetneq V_{1}(so(\g))$ are simple currents
extensions of  $V_{h^\vee}(\g)$ whose factors are special simple currents.
\end{theorem}

We identify $V_{1}(so(\g))$ with the basic module $L(\tilde \L_0)$, where $\tilde \L_0$ is the $0$-th fundamental weight for $\widehat{so(\g)}$.
 Next we realize $L(\tilde \L_0)$ explicitly as the even part of  the fermionic vertex algebra  $F(\bar \g)$ of $(\bar\g,(\cdot,\cdot))$ where  $\bar\g$ is $\g$ seen as an odd space and
$(\cdot, \cdot)$  is the Killing form on $\g$ (cf. \ref{ferm}).  If $x\in\g$,  we let $\bar x$ be the corresponding element of $F(\bar \g)$.


If $X\in so(\g)$, set $\Theta(X)=\half\sum_{i} :\overline{X(x_i) }\bar x^i:$ where $\{x_i\}$ is a basis of $\g$ and $\{x^i\}$ its dual basis.
 In the identification $V_{1}(so(\g))=L(\tilde\L_0)\subset F(\bar \g)$,  an element $X$ of $so(\g)$  is mapped to $\Theta(X)$.



 Let $\h$ be a Cartan subalgebra of $\g$.
 Choose a non-zero vector
  $x_\a$ in the root space $\g_\a$ corresponding to the positive root $\a$ and consider $x_{-\a}$ such that $(x_\a,x_{-\a})=1$.

 Let $\L_0$ be the $0$-th fundamental weight for $\widehat{\g}$.
According to \cite{CKMP},  as a representation of the derived algebra of $\widehat \g$,   we have
$$L(\tilde \L_0)=\bigoplus_{I\in Ab_0}L(h^\vee\L_0+\langle I \rangle),$$
where
 $Ab_0$ is the set of even dimensional abelian ideals in a (suitable) Borel subalgebra $\b$ of $\g$ and $\langle I \rangle =\sum_{\g_\a\subset I}\a$. Moreover,
 the highest vector of $L(h^\vee\L_0+\langle I \rangle)$ is given by $$v_I=:\prod_{\g_\a\subset I}\bar x_\a:.$$
 Likewise,  we introduce
 $$v_{-I}=:\prod_{\g_\a\subset I}\bar x_{-\a}:.$$

We  identify $V_{h^\vee}(\g)$ with $L(h^\vee\L_0)$, so, if $V$ is a vertex algebra intermediate between $L(h^\vee\L_0)$ and $L(\tilde \L_0)$, then
$$V=\sum_{I\in Ab_V}L(h^\vee\L_0+\langle I \rangle)
$$
with $Ab_V$ a subset of $Ab_0$.


 Write $E_{\a,\a}\in gl(\g)$ for the linear transformation, diagonal in the given basis, such that $E_{\a,\a}(x_\a)=x_\a$,  $E_{\a,\a}(x_\be)=0$ for $\be\ne\a$, and $E_{\a,\a}(h)=0$ for $h\in\h$.
A special case of formula \eqref{lambdaproductss}  gives that,
$$
(v_I)_{(2n-2)} v_{-I}=\sum_{\g_\a\subset I}\Theta(E_{\a,\a}-E_{-\a,-\a}).
$$

If $I\in Ab_V$ then $v_I\in V$. It is easy to check (see also  \cite[Remark 4.10]{K}) that if $I\in Ab_0$, then there is a unique ideal $I^*\in Ab_0$ such that $\langle I^* \rangle =-w_0(\langle I \rangle)$. Here $w_0$ is the longest element of
 the Weyl group of $\g$. The simplicity of  $V$ implies that, if $L(h^\vee\L_0+\langle I \rangle)\subset V$,  then $L(h^\vee\L_0+\langle I^* \rangle)\subset V$ (see Lemma \ref{nonz} below).
Since $v_{-I}\in L(h^\vee\L_0+\langle I^* \rangle)$, we have that $v_{-I}\in V$. It follows that $\sum_{\g_\a\subset I}\Theta(E_{\a,\a}-E_{-\a,-\a})\in V$, so
 that $V$ contains both $\Theta( ad(\g))$ and $\sum_{\g_\a\subset I}\Theta(E_{\a,\a}-E_{-\a,-\a})$. If $\sum_{\g_\a\subset I}(E_{\a,\a}-E_{-\a,-\a}) \not\in ad(\g)$, then, since the conformal embedding of $ad(\g)$ in $so(\g)$ is maximal
, we have that
  $so(\g)\subset V$, but then $V=V_{1}(\g)$. Otherwise, we have $\sum_{\g_\a\subset I}(E_{\a,\a}-E_{-\a,-\a})=ad(h)$ with $h\in\h$. Since $ad(h)=\sum_{\a\in\D}\a(h)E_{\a,\a}$, there is a simple root $\a_i$ such that $\a_i\in I$ and this can be shown to   imply that $\a_i$ has coefficient 1 in the highest root and $I$ is the nilradical of the maximal parabolic corresponding to $\a_i$.  Thus $\langle I\rangle =h^\vee \omega_i$, and $\omega_i$ is a minuscule  fundamental weight of $\g$. But then $L(h^\vee(\L_0+\omega_i))$ is a special simple current.

\section{Vertex algebras}\label{tre}
For basic notions on the theory of vertex algebras we refer the reader to \cite{KacV} or  \cite[\S 1]{KacD}. We will be mainly consistent with notation used in the latter reference.
In particular, we shall denote by $a\mapsto Y(a,z)=\sum_{n\in\ganz} a_{(n)}Êz^{-n-1}$ the state-field correspondence, by $T$ the translation operator and by $p(a)$ the parity of $a\in V$. For
$a,b\in V$ we set $p(a,b)=(-1)^{p(a)p(b)}$.

A  conformal (or Virasoro) element in a vertex algebra $V$ is a vector
$\omega$ in $V$,  such that the  associated field $Y(\omega,z)=\sum_{n\in\ganz} \omega_{(n)}z^{-n-1}$
has the following three properties:
\begin{enumerate}
\item $[\omega_\lambda\omega]=(T+2\lambda)\omega+c\frac{\lambda^3}{12}$ for some $c\in\C$;
\item $\omega_{(1)}$ is diagonalizable;
\item
$\omega_{(0)}=T$.
\end{enumerate}
The number $c$ is called the central charge of $\omega$. Write $V=\oplus_{n\in\C}V_n$ for the eigenspace decomposition of $V$ with respect to $\omega_{(1)}$. If $v\in V_n$, then $n$ is called the conformal weight of $v$ and it is denoted by $\D_v$.
As shown in  Proposition 1.15 of \cite{KacD}
\begin{equation}\label{conformalweight}
\D_{a_{(n)}b}=\D_a +\D_b -n-1,\ \D_{T(a)}=\D_a+1.
\end{equation}

Recall that  $V_1/TV_0$ acquires a natural structure of Lie algebra, with bracket defined by $[a+TV_0,b+TV_0]=a_{(0)}b+TV_0$. In particular, if  $V_0=\C|0\rangle$, then $V_1$ is a Lie algebra with bracket $[a,b]=a_{(0)}b$.\par

If $A,B$ are subspaces of $V$, we define, following \cite[Remark 7.6]{BK}
\begin{equation}\label{.}
A\cdot B= Span\{a_{(n)}b\mid a\in A, b \in B, n\in\ganz\}.\end{equation}

\begin{lemma}\label{nonvanishing} (a) The product \eqref{.} and the vector space sum define on the set of all $T$-stable subspaces of $V$ a structure of unital associative commutative ring
$\mathcal A_V$.\par
(b) If $V$ is a simple  vertex algebra, then the ring $\mathcal A_V$ is a domain.\par
(c)
If $V$ is a simple vertex algebra, and $a,b\in V$ are such that $b\ne 0$ and $Y(a,z)b=0$, then $a=0$.\end{lemma}
\begin{proof} The subspace $\C|0\rangle$ is an identity element of  $\mathcal A_V$ by (1.23) and (1.24) of \cite{KacD}. The associativity of the product \eqref{.} follows from the Borcherds identity
(see \cite[(1.28)]{KacD}) and commutativity follows from the skewsymmetry  (see \cite[(1.29)]{KacD}), proving (a).\par
Claim (b) follows from (c). In order to prove (c), it is enough
 to show that, if $a,b\in V,\,b\ne 0$ and $Y(a,z)b=0$, then $a=0$.
Consider the set $Cent(b)=\{a\in V\mid Y(a,z)b=0\}$. We claim that $Cent(b)$ is an ideal of $V$. It is $T$-stable, since $Y(Ta,z)=\partial_zY(a,z)$. Moreover by the Borcherds identity (cf. (1.28) of \cite{KacD}), letting $m>>0$, we find that if $x\in V$, $a\in Cent(b)$, then, for any $k,n\in\ganz$,
$$
\sum_{j\in \nat}\binom{m}{j} (x_{(n+ j)}a)_{(m+k- j)}b=0
$$
Choose $N$ so that $x_{(h)}a=0$ for $h>N$. An obvious induction on $t\geq 0 $ shows that $(x_{(N-t)}a)_{(r)}b=0$ for all $t\geq 0 $ and all $r$. Hence $Cent(b)$ is a left ideal, and since it is $T$-stable it is a 2-sided ideal (by skewsymmetry). Since $V$ is simple, we have either $Cent(b)=\{0\}$ or $Cent(b)=V$. In the latter case, from the skewsymmetry relation, we get that $Y(b,z)=0$, which contradicts the assumption that $b\ne 0$.
\end{proof}


\subsection{The fermionic vertex algebra}\label{ferm}If $A$ is a vector space we write $\bar A$ for
the totally odd vector  superspace such that $\bar A^0=\{0\}$ and $\bar A^1=A$. If $a\in A$, write $\bar a$ to denote $a$ seen as an element of $\bar A$. Assume that $A$ is equipped with a nondegenerate symmetric  bilinear form $(\cdot,\cdot)$. Then  one can construct the
Clifford  Lie conformal algebra as
$$
R^{Cl}(\bar A)=(\C[T]\otimes\bar A)\oplus \C K,
$$
with the  $\l$-bracket
\begin{equation}\label{lambdabracketfermi}
[\bar a_\l \bar b]=(a,b)K\quad [\bar a_\l K]=[K_\l K]=0,\ a,b\in A,
\end{equation}
and $T$ acting on $\C[T]\otimes A$ by left multiplication and trivially on $K$.
Let $V(\bar A)$ be the universal enveloping vertex algebra of this Lie conformal algebra.
The vertex algebra
$$
F(\bar A)=V(\bar A)/:(K-1)V(\bar A):
$$
is called the {\it fermionic vertex algebra}.

The fermionic vertex algebra has a Virasoro element:
\begin{equation}\label{virf}\omega_{\bar A}= \frac{1}{2}\sum: T (\bar x_i )\bar x^i : ,\end{equation}
where
$\{x_i \}$ and $\{x^i\}$  is a pair of dual bases of $A$.

\begin{rem}\label{Cliffordmodule}
The vertex algebra $F(\bar A)$ can be constructed explicitly as follows: consider  $L(\bar A)=\sum_{r\in\frac{1}{2}+\ganz}t^r\otimes A$ and set $L(\bar A)^+=\sum_{r>0}t^r\otimes A$. Extend the form $(\cdot,\cdot)$ to $L(\bar A)$ by setting
$$
(t^r\otimes a_1,t^s\otimes a_2)=\d_{r,-s}(a_1,a_2).
$$
Let $\text{Cliff}(L(\bar A))$ be the corresponding Clifford algebra. Then one can identify $F(\bar A)$ and the Clifford module $\text{Cliff}(L(\bar A))/(\text{Cliff}(L(\bar A))L(\bar A)^+)$. If $x\in A$, then the corresponding field is $Y(x,z)=\sum_{n\in \ganz}x_{(n)}z^{-n-1}$ where $x_{(n)}$ is the operator given by the action of $t^{n+\frac{1}{2}}\otimes x$. (See \cite[3.6]{KacV}).\end{rem}



\subsection{Some calculations in the fermionic vertex algebra}
The goal of this section is to prove formula \eqref{genfactorials} in Proposition \ref{f}, which will be needed in the sequel.
\begin{lemma}\label{1lambamany}
If $x,y\in A$ are such that $(x,y)=1$, then
\begin{enumerate}
\item
\begin{align*}
[T^n(\bar y)_{\l} &:T^{n_1}(\bar x)T^{n_2}(\bar x)\dots T^{n_k}(\bar x):]\\&=\sum_{i=1}^k(-1)^{i+1}(-\l)^n\l^{n_i}:T^{n_1}(\bar x)T^{n_2}(\bar x)\dots \widehat{T^{n_i}(\bar x)}\dots T^{n_k}(\bar x):.
\end{align*}
\item
\begin{align*}
[:T^{n_1}(\bar x)&T^{n_2}(\bar x)\dots T^{n_k}(\bar x):_\l T^n(\bar y)]\\
&=\sum_{i=1}^k(-1)^{k-i}(\l+T)^n(-\l-T)^{n_i}:T^{n_1}(\bar x)T^{n_2}(\bar x)\dots \widehat{T^{n_i}(\bar x)}\dots T^{n_k}(\bar x):.
\end{align*}
\end{enumerate}
\end{lemma}
\begin{proof}
The first formula is a direct application of Wick's formula. The second formula is obtained from the first by applying  skew-symmetry of the $\l$-bracket.
\end{proof}

We let  $\D_{[a_\l b]}$ be the conformal weight of the leading coefficient of the polynomial $[a_\l b]$. If $\mathbf{n}\in\mathbb{N}^k$, set $T^{\mathbf{n}}(\bar x)=T^{n_1}(\bar x)\cdots T^{n_k}(\bar x)$. In what follows we will use many times \eqref{conformalweight}.
\begin{lemma}\label{confw}If $\mathbf{n}\in\mathbb{N}^h$, $\mathbf{m}\in\mathbb{N}^k$ with $h<k$, and $[:T^{\mathbf{n}}(\bar x):_\l :T^{\mathbf{m}}(\bar y):]\ne 0$, then
$$
\Delta_{[:T^{\mathbf{n}}(\bar x):_\l :T^{\mathbf{m}}(\bar y):]}\ge min(m_i)+\half.
$$
\end{lemma}
\begin{proof}
The proof is by induction on $h$. If $h=1$, then  the result follows from Lemma \ref{1lambamany}. If $h>1$, set $\mathbf{n}_i=(n_1,\dots, \widehat{n_i},\dots, n_h)\in\nat^{h-1}$. By Wick's formula
\begin{align*}
[:T^{\mathbf{n}}&(\bar x):_\l :T^{\mathbf{m}}(\bar y):]=:[:T^{\mathbf{n}}(\bar x):_\l :T^{m_1}(\bar y):]T^{\mathbf{m}_1}(\bar y):\\&+(-1)^h:T^{m_1}(\bar y)[:T^{\mathbf{n}}(\bar x):_\l :T^{\mathbf{m}_1}(\bar y):]:\\
&-(-1)^h\sum_{i=1}^h(-1)^{i+1}\int_0^\l(\l-\mu)^{m_1}(-\l+\mu)^{n_i}[:T^{\mathbf{n}_i}(\bar x):_\mu :T^{\mathbf{m}_1}(\bar y):]d\mu.
\end{align*}
The conformal weight of the first summand is clearly bigger than $m_2+\half$. The conformal weight of the second summand is clearly bigger than $m_1+\half$. Finally, by the induction hypothesis,
$$\D_{[:T^{\mathbf{n}_i}(\bar x):_\mu :T^{\mathbf{m}_1}(\bar y):]}\ge min(m_2,\dots,m_k)+\half.$$
The result follows.
\end{proof}
If $\mathbf{n},\mathbf{m}\in\mathbb{N}^k$, set $A(\mathbf{n},\mathbf{m})=((-1)^{n_j}(m_i+n_j)!)_{1\le i,j\le k}$ and $C(\mathbf{n},\mathbf{m})=detA(\mathbf{n},\mathbf{m})$. For $k=0$, we set $C(\emptyset,\emptyset)=0$. Let $\nat^k_{reg}$ be the set of all $\mathbf{n}$ such that $n_i\ne n_j$ for $i\ne j$. Introduce the divided powers $\l^{(s)}=\frac{\l^s}{s!}$.
\begin{lemma}\label{lambdageneral}Assume $\mathbf{n},\mathbf{m}\in\nat_{reg}^k$, $k\ge1$. Let $j_0$ (resp. $i_0$) be the index such that $n_{j_0}=min(n_j)$ (resp. $m_{i_0}=min(m_i)$). Set $N=\Delta_{:T^{\mathbf{n}}(\bar x):}+\Delta_{:T^{\mathbf{m}}(\bar y):}-1$.

If $(x,y)=1$, then
\begin{align*}
&[:T^{\mathbf{n}}(\bar x):_{\l} :T^{\mathbf{m}}(\bar y):]=\\&(-1)^{[k/2]}
C(\mathbf{n},\mathbf{m})\l^{(N)}|0\rangle\\&+(-1)^{[k/2]}(-1)^{i_0+j_0}C(\mathbf{n}_{i_0},\mathbf{m}_{j_0})\l^{(N-n_{j_0}-m_{i_0}-1)}:T^{n_{j_0}}(\bar x)T^{m_{i_0}}(\bar y):\\&+\text{lower order terms}.
\end{align*}
\end{lemma}
\begin{proof}
The proof is by induction on $k$. If $k=1$, then
$$
[:T^n(\bar x):_\l :T^m(\bar y):]=(-1)^n\l^{n+m}=(-1)^n(m+n)!\l^{(n+m)},
$$
as desired.

If $k\ge 1$, we can assume that $i_0=j_0=1$.  By Wick's formula and Lemma \ref{1lambamany},
\begin{align*}
[:T^{\mathbf{n}}&(\bar x):_\l :T^{\mathbf{m}}(\bar y):]=:[:T^{\mathbf{n}}(\bar x):_\l :T^{m_1}(\bar y):]T^{\mathbf{m}_1}(\bar y):\\&+(-1)^k:T^{m_1}(\bar y)[:T^{\mathbf{n}}(\bar x):_\l :T^{\mathbf{m}_1}(\bar y):]:\\
&-(-1)^k\sum_{i=1}^k(-1)^{i+1}\int_0^\l(\l-\mu)^{m_1}(-\l+\mu)^{n_i}[:T^{\mathbf{n}_i}(\bar x):_\mu :T^{\mathbf{m}_1}(\bar y):]d\mu.
\end{align*}
By Lemma \ref{confw}, the conformal weights occurring in the first summand are greater or equal than $n_1+min(m_i\mid i\ge 2)+1$. Since $n_1+min(m_i\mid i\ge 2)+1>m_1+n_1+1$, we see that
\begin{align}
[:T^{\mathbf{n}}(\bar x)&:_\l :T^{\mathbf{m}}(\bar y):]=(-1)^k:T^{m_1}(\bar y)[:T^{\mathbf{n}}(\bar x):_\l :T^{\mathbf{m}_1}(\bar y):]:\label{WRight}\\
&-(-1)^k\sum_{i=1}^k(-1)^{i+1}\int_0^\l(\l-\mu)^{m_1}(-\l+\mu)^{n_i}[:T^{\mathbf{n}_i}(\bar x):_\mu :T^{\mathbf{m}_1}(\bar y):]d\mu\notag\\&+\text{lower order terms.}\notag\end{align}
Applying Wick's formula and Lemma \ref{1lambamany}, we obtain that
\begin{align*}
[:T^{\mathbf{n}}&(\bar x):_\l :T^{\mathbf{m}_1}(\bar y):]=:e^{T\partial_\l }T^{n_1}(\bar x)[:T^{\mathbf{n}_1}(\bar x):_\l :T^{\mathbf{m}_1}(\bar y):]:
\\&+(-1)^{k-1}:e^{T\partial_\l} T^{\mathbf{n}_1}(\bar x)[:T^{n_1}(\bar x):_\l :T^{\mathbf{m}_1}(\bar y):]:\\
&+(-1)^{k-1}\sum_{i=1}^{k-1}(-1)^{i+1}\int_0^\l(-\l+\mu)^{n_1}(\l-\mu)^{m_{i+1}}[:T^{\mathbf{n}_1}(\bar x):_\mu :T^{(\mathbf{m}_1)_i}(\bar y):]d\mu.
\end{align*}
The conformal weight of the second summand is bigger than $min(n_i\mid i>1)+\half$ and, by Lemma \ref{confw}, the conformal weights of the terms occurring in the third sum are bigger than $min(n_i\mid i>1)+\half$ as well. Since $min(n_i\mid i>1)+\half>n_1+\half$, substituting in \eqref{WRight}, we have
\begin{align*}
[:T^{\mathbf{n}}&(\bar x):_\l :T^{\mathbf{m}}(\bar y):]=(-1)^k:T^{m_1}(\bar y)e^{T\partial_\l }T^{n_1}(\bar x)[:T^{\mathbf{n}_1}(\bar x):_\l :T^{\mathbf{m}_1}(\bar y):]:\\
&-(-1)^k\sum_{i=1}^k(-1)^{i+1}\int_0^\l(\l-\mu)^{m_1}(-\l+\mu)^{n_i}[:T^{\mathbf{n}_i}(\bar x):_\mu :T^{\mathbf{m}_1}(\bar y):]d\mu\\&+\text{lower order terms.}
\end{align*}
By the induction hypothesis, we find that
\begin{align*}
[&:T^{\mathbf{n}}(\bar x):_\l :T^{\mathbf{m}}(\bar y):]=
(-1)^{k+[(k-1)/2]}C(\mathbf{n}_1,\mathbf{m}_1)\l^{(N-m_1-n_1-1)}:T^{n_1}(\bar x)T^{m_1}(\bar y):\\
&-(-1)^{k+[(k-1)/2]}\sum_{i=1}^k(-1)^{i+1} C(\mathbf{n}_i,\mathbf{m}_1)\int_0^\l(-1)^{n_i}(\l-\mu)^{m_1+n_i}\mu^{(N-n_i-m_1-2)}|0\rangle d\mu\\&+\text{lower order terms.}
\end{align*}
Now observe that
$$\int_0^\l(-1)^{n_i}(\l-\mu)^{m_1+n_i}\mu^{(N-n_i-m_1-2)}d\mu=(-1)^{n_i}(m_1+n_i)!\l^{(N)}
$$
and that $-(-1)^{k+[(k-1)/2]}=(-1)^{[k/2]}$, hence the proof is complete.
\end{proof}

Let $\mathbf{n!}=(0,1,\dots,n)\in \nat^{n+1}$.
\begin{cor}\label{formulafactorial} Let $N=(n+1)^2-1$. If $n\ge1$ and $(x,y)=1$ then
$$
[:T^{\mathbf{n!}}(\bar x):_\l :T^{\mathbf{n!}}(\bar y):]=(\prod_{i=0}^n i!)^2(\l^{(N)}|0\rangle+(n+1)\l^{(N-1)}:\bar x\bar y:)+\text{lower order terms.}
$$
\end{cor}
\begin{proof}We note that $\Delta_{:T^{\mathbf{n!}}(\bar x):}=(n+1)^2/2$.
Moreover
$$
C(\mathbf{n!},\mathbf{n!})=det((-1)^j (i+j)!)_{0\le i,j\le n}=(-1)^{[k/2]}det( (i+j)!)_{0\le i,j\le n}.
$$
The latter determinant  is classically evaluated as (see e.g. \cite{Kra})
$$
det( (i+j)!)_{0\le i,j\le n}=\prod_{i=1}^n(i!)^2.
$$
Note that
$$C(\mathbf{n!}_1,\mathbf{n!}_1)=det((-1)^j (i+j)!)_{1\le i,j\le n}=(-1)^{[k/2]}det( (i+j)!)_{1\le i,j\le n}
$$
Using Proposition 1 in Section 2.1 of \cite{Kra}, one can show that
\begin{equation}\label{Kra}
det( (i+j)!)_{1\le i,j\le n}=(n+1)\prod_{i=1}^n(i!)^2.
\end{equation}
The statement now follows from Lemma \ref{lambdageneral}.
\end{proof}

Let $A'$ be a subspace of $A$ such that $A=A'\oplus (A')^\perp$. Then the map $a\otimes b\mapsto :ab:$ establishes an isomorphism between $F(\bar A')\otimes F((\bar A')^\perp)$ and $F(\bar A)$. Hence, if $a',b'\in F(\bar A')$ and $a'',b''\in F((\bar A')^\perp)$,
$$
:a'a'':_{(n)}:b'b'':=p(a'',b')\sum_{r+s=n-1}:(a'_{(r)}b')(a''_{(s)}b''):.
$$

In particular, if $N'=\D_{a'}+\D_{b'}-1$, $N''=\D_{a''}+\D_{b''}-1$, and $N=N'+N''+1$, then
\begin{align}
[a'a''_\l b'b'']&=p(b',a''):(a'_{(N')}b')(a''_{(N'') }b''):\l^{(N)}\label{decouple}\\&+p(a',a'')(:(a'_{(N'-1)}b')(a''_{(N'')} b''):+:(a'_{(N')}b')(a''_{(N''-1)}b''):)\l^{(N-1)}\notag\\&+ \text{lower order terms.}\notag
\end{align}

We are now ready to prove the main result of this section.
Set $C_n=\prod_{i=0}^n(i!)^2$.
 \begin{prop}\label{f} Assume that $x_1,\dots,x_k,y_1,\dots,y_k\in A$ are such that $(x_i,x_j)=(y_i,y_j)$ $=0$ for all $i,j$ and $(x_i,y_j)=\d_{ij}$. Fix $n_1,\dots,n_k\in\nat$ and set
 $$a=:T^{\mathbf{n_1!}}(\bar x_1)\dots T^{\mathbf{n_k!}}(\bar x_k):
$$
and
$$
b=:T^{\mathbf{n_1!}}(\bar y_1)\dots T^{\mathbf{n_k!}}(\bar y_k):.
$$

Then
\begin{align}
[a_\l b]=&\prod_{i<j}(-1)^{(n_i+1)(n_j+1)}\prod_{i=1}^k C_{n_i}\left(\l^{(N)}|0\rangle+\l^{(N-1)}\sum_{i=1}^k (n_i+1):\bar x_i\bar y_i:\right)\label{genfactorials}\\
&+\text{lower order terms,}\notag
\end{align}
where $N=\sum_{i=1}^k (n_i+1)^2-1$. In particular
\begin{align}\label{lambdaproductss}
[:\bar x_1\dots \bar x_k:_\lambda : \bar y_1\dots\bar y_k:]&=(-1)^{\lfloor k/2\rfloor}(|0\rangle\l^{(k-1)}+\lambda^{(k-2)}(\sum_i :\bar x_i\bar y_i:)) \\&+ \text{lower order terms}\notag.
\end{align}
\end{prop}
\begin{proof}
By formula (1.40) of \cite{KacD}, we have that
$$
a=::T^{\mathbf{n_1!}}(\bar x_1):\dots :T^{\mathbf{n_k!}}(\bar x_k)::,\qquad
b=::T^{\mathbf{n_1!}}(\bar y_1):\dots :T^{\mathbf{n_k!}}(\bar y_k)::.$$
We can therefore apply repeatedly  \eqref{decouple} combined with Corollary \ref{formulafactorial} to obtain the result.
\end{proof}
\subsection{The affine vertex algebra}\label{ava}
Let $\g$ be a simple or abelian finite dimensional complex Lie algebra.

Let $(\cdot,\cdot)$  be
the normalized nondegenerate invariant symmetric bilinear form on $\g$ (i.e., the square length of a long root of $\g$ is $2$); if $\g$ is abelian, any non-degenerate symmetric bilinear form will do. One defines the  Lie
conformal algebra $Cur(\g)$ as
$$
Cur(\g)=(\C[T]\otimes \g)+\C K
$$
with $\l$-bracket defined for $a,b\in1\otimes \g$  by
\begin{align*}
[a_\l b]&=[a,b]+\l(a,b)K,\ a,b\in\g,\\
[a_\l K]&=[{K}_\l K]=0.
\end{align*}
Let $V(\g)$ be its universal enveloping vertex algebra.
Let $h^\vee$ be the dual Coxeter number of $\g$ ($h^\vee=0$ if $\g$ is abelian), and choose $k\in\C$ with $k\ne -h^\vee$. The vertex algebra
$$
V^{{k}}(\g)=V(\g)/:(K-k\va)V(\g):
$$
is called the {\it level ${k}$} universal affine vertex algebra. It has a unique simple quotient $V_{k}(\g)$ called the {\it level $k$} simple affine vertex algebra. If $\g$ is reductive, let  $\g=\g_0\oplus \g_1 \oplus \dots \oplus \g_s$ with $\g_0$ abelian and $\g_i$ simple ideals for $i=1,\dots,s$. If $\mathbf{k}\in\C^{s+1}$, we define
$$
V^{\mathbf{k}}(\g)=V^{k_0}(\g_0)\otimes V^{k_1}(\g_1)\otimes\dots\otimes V^{k_s}(\g_s)
$$
and
$$
V_{\mathbf{k}}(\g)=V_{k_0}(\g_0)\otimes V_{k_1}(\g_1)\otimes\dots\otimes V_{k_s}(\g_s).
$$
Note that both $V^{\mathbf{k}}(\g)$ and $V_{\mathbf{k}}(\g)$ are vertex algebras with Virasoro element $\omega_\g$ given by the Sugawara construction:
$$\omega_\g=\sum_{j=0}^s\frac{1}{2(k_j+h_j^\vee)}\sum_{i=1}^{\dim \g_j} :x^j_i x_j^i :.
$$
Here $\{x_i^j\},\{x_j^i\}$ are dual bases of $\g_j$ and $h^\vee_j$ is its dual Coxeter number.\par
Observe that $V^{\mathbf{k}}(\g)_0=\C|0\rangle$, so  $V^{\mathbf{k}}(\g)_1$ is a Lie algebra. The obvious embedding of $\g$ in $Cur(\g)$ defines an embedding of $\g$ in $V^{\mathbf{k}}(\g)$ and it is easy to check that the image of this embedding is precisely $V^{\mathbf{k}}(\g)_1$. If $\pi: V^{\mathbf{k}}(\g)\to V_{\mathbf{k}}(\g)$ is the canonical homomorphism, then it is not hard to check  that $\pi$ is injective when restricted to $V^{\mathbf{k}}(\g)_1$, thus we can identify $V_{\mathbf{k}}(\g)_1$ and $\g$.


\subsection{Conformal embeddings}\label{confemb}
\begin{defi}
Let $V$ and $W$ be vertex algebras equipped with Virasoro elements $\omega_V$, $\omega_W$ and assume that $W$ is a vertex subalgebra of $V$. We say that $W$ is conformally embedded if $\omega_V=\omega_W$.
\end{defi}
The following construction provides several examples of conformal embeddings. Let $\fa =\r \oplus \fp$ be the eigenspace
decomposition of an involution $\s$ of a semisimple Lie algebra
$\fa$. In this case it is usually said that $(\fa,\r)$ is a symmetric pair. Let $\fa=\oplus \fa_s$ be the decomposition of $\fa$ into $\s$-indecomposable ideals. Let $(\cdot,\cdot)$ be the Killing form of $\fa$. Clearly $(\cdot,\cdot)$ is nondegenerate when restricted to $\p$. Write
    $\r=\sum_{S=0}^M\r_S$ with $\r_0$ abelian and $\r_S$  simple ideal for $S>0$. Let $(\cdot,\cdot)_0=\half(\cdot,\cdot)_{|\r_0\times\r_0}$ and, if $S>0$, let $(\cdot,\cdot)_S$ be the normalized   invariant form on $\r_S$ .  Let $ad_\p$ denote the adjoint action of $\r$ on $\p$.     Set $h_i^\vee$ to be the dual Coxeter number  of $\fa_i$ (since $\s_{|\fa_i}$ is indecomposable, $\fa_i$ is either simple or the sum of two isomorphic simple ideals, and in the latter case $h_i^\vee$ is the dual Coxeter number of one of these simple ideals). If $S>0$, let $h^\vee_S$ be the dual Coxeter number of $\r_S$.
    Set, for $S>0$,
    $$
    j_S=n_Sh_{i_S}^\vee-h^\vee_S
   $$
where $n_S=\frac{1}{h_{i_S}^\vee(\a,\a)}$,  $\a$ is any long root of $\r_S$ and $\fa_{i_S}$ is the unique indecomposable ideal containing $\r_S$. Set also $j_0=1$.
 Then,  (see \cite[Section 2]{CKMP})
     \begin{equation}\label{jS}
     \half tr(ad_\p(h)ad_\p(h'))=j_S(h,h')_S
     \end{equation}
for any $h,h'\in\r_S$.
     Set
     \begin{equation}\label{mathbfj}\mathbf{j}=(j_0,\dots, j_M).
     \end{equation}

Let  $\{x_i\}$ a basis of $\p$, and $\{x^i\}$ its dual basis. It was shown by Kac and Peterson (when $\s$ is inner)
\cite{KacPeterson} and by Goddard, Nahm and Olive
\cite{GNO}  that the map
\begin{equation}
\Theta(X)=\half\sum_{i} :\overline{[X,x_i]}\bar  x^i:
\end{equation}
can be extended to an embedding of $V_{\mathbf{j}}(\r)$ in $F(\bar \p)$.

Moreover the image of the Virasoro element of $V_{\mathbf{j}}(\r)$ in $F(\bar\p)$ coincides with $\omega_{\bar\p}$ (cf. \eqref{virf}). Thus $V_{\mathbf{j}}(\r)$  embeds conformally in $F(\bar \p)$.

The other way around also holds: let
\begin{equation}\label{reductive}
\k=\k_0\oplus \sum_{i=1}^s\k_i
\end{equation} be a reductive Lie algebra with, as usual, $\k_0$ abelian and $\k_i$ simple ideals for $i>0$. The symmetric space theorem \cite{GNO} implies  the following statement.
\begin{prop}\label{symmthm} If $\k$ is such that
\begin{enumerate}
\item[i.] there is vertex algebra embedding of $V_{\mathbf{k}}(\k)$  in $F(\bar A)$ for some vector space $A$ and some ${\mathbf{k}}\in\C^*\times\nat^s$;
\item[ii.]  the Virasoro elements of $V_{\mathbf{k}}(\k)$  and  $F(\bar A)$ coincide,
\end{enumerate}
 then there is a symmetric pair  $(\aa,\r)$ such that $\r=\k$ and $\p= A$; moreover,   $\mathbf{k}=\mathbf{j}$.
 \end{prop}

A special case of the above discussion occurs when we consider a vector space $A$ with a nondegenerate bilinear symmetric form $(\cdot,\cdot)$, and embed $A$ in $\tilde A=A\oplus\C$ with the form extended by $(a,1)=0$ for $a\in A$ and $(1,1)=1$. Let $L$  be the linear map on $\tilde A$ defined by $L(a+c)=-a+c$. Then $\s=Ad(L)$ is an involution of $so(\tilde A,(\cdot,\cdot))$ and the corresponding eigenspace decomposition is
$$
so(\tilde A,(\cdot,\cdot))=so(A,(\cdot,\cdot))\oplus A.
$$
The adjoint action of $so(A,(\cdot,\cdot))$ on $A$ coincides with the standard action of $so(A,(\cdot,\cdot))$ on $A$. In this case the above construction gives that
$V_1(so(A,(\cdot,\cdot))$ embeds conformally in $F(\bar A)$ and the embedding is given by the extension of the map $\Theta$ defined by
\begin{equation}\label{theta}
\Theta(X)=\half\sum_{i} :\overline{X(x_i)}\bar  x^i:.
\end{equation}

 Note also that $F(\bar A)_0=\C|0\rangle$, so $F(\bar A)_1$ is a Lie algebra.
One checks easily that the map $\Theta$
 between $so(A, (\cdot ,\cdot))=V_{1}(so(A,(\cdot,\cdot)))_1$ and $F(\bar A)_1$ defined above is onto. Thus $\Theta$ gives an isomorphism between $F(\bar A)_1$ and $so(A, (\cdot ,\cdot))$.

 Set
 \begin{equation}\label{F}F(\bar A)^0=\oplus_{n\in\nat}F(\bar A)_n.\end{equation} Since $V_{1}(so(A,(\cdot, \cdot))_n\ne \{0\}$ only if $n\in\nat$, it is clear that $\Theta(V_{1}(so(A,(\cdot, \cdot)))\subset F(\bar A)^0$. Since $F(\bar A)_1\subset \Theta(V_{1}(so(A,(\cdot, \cdot)))$, we can identify $F(\bar A)^0$ and $V_{1}(so(A,(\cdot, \cdot))$.

\begin{defi}\label{c} If $\g$ is a simple Lie algebra and $\k$ is a reductive subalgebra as in \eqref{reductive}, we say that $\k$ is conformally embedded in $\g$, if there is $\mathbf{k}\in \C^*\times \nat^M$ such that $V_{\mathbf{k}}(\k)$
is conformally embedded in $V_{1}(\g)$.
\end{defi}

\begin{rem}\label{confinson}If $(\aa,\r)$ is a symmetric pair, it is clear that $\Theta$ maps $V_{\mathbf{j}}(\r)$ in $F(\bar\p)^0=V_{1}(so(\p))$. Thus $\r$ is conformally embedded in $\g=so(\p)$ and the symmetric space theorem  implies that any conformal embedding in $so(n,\C)$ arises in this way.
\end{rem}

\begin{rem}
If $\k\hookrightarrow  \g$ is a conformal embedding and $\mathfrak v$ is  a reductive subalgebra of $\g$ containing $\k$, then $\mathfrak v$ is conformal in $\g$. This follows from the fact that, since $V_{1}(\g)$ is a unitarizable $V_{1}(\g)$-module, then the equality of the Virasoro elements is equivalent to the equality of the corresponding central charges. Thus, if $\k$ is a maximal conformal subalgebra of $\g$, then it  is actually a maximal reductive subalgebra.
The list of such maximal embeddings is given in \cite{AGO}.
\end{rem}
\subsection{Finite decomposition property}\label{fdp} The notion of conformal embedding first arose in literature in  a slightly different way.
Consider a  pair $(\g ,\fk)$, where $\g$ is a finite-dimensional
simple Lie algebra over $\CC$ and $\fk$ is a reductive
subalgebra of $\g$, such that the restriction of the Killing form of
$\g$ to $\fk$ is non-degenerate.
Denote by $\ga,\ka$ the corresponding  affinizations \cite{Kac}. Recall that a level $k$ highest weight representation of $\ga$ can be naturally viewed as
a representation of $V^{k}(\g)$ (cf. \cite{KacV}).
It is well-known that any integrable highest weight
$\widehat{\g}$-module, when restricted to $\widehat{\fk}$, decomposes
into a direct sum of irreducible $\widehat{\fk}$-modules \cite{Kac},
but almost always this decomposition is infinite.  As we have already remarked, the first cases of a finite decomposition were found in \cite{KacPeterson}.
This led to define the notion of  a conformal
  pair in terms of the following finite decomposition property: $\fk$ was called a \emph{conformal subalgebra} of $\fg$
if there exists a non-trivial  integrable highest weight module $V$
over the affine Kac--Moody algebra $\widehat{\g}$, such that the restriction to $\widehat{\fk}$ of
each weight space of the center $\mathfrak z(\k)$ of $\fk$ in $V$ decomposes into a
finite direct sum of irreducible $\widehat{\fk}$-modules. \par
It was readily found that the decomposition in question is finite
if and only if  the
central charges of the Sugawara construction of the Virasoro
algebra for $\widehat{\g}$ and $\widehat{\fk}$ are equal \cite{GNO}.
Recall that the function
$$c_\g(k)=\frac{k\dim\g}{k+h^\vee}$$
expresses the central charge of the Sugawara construction at level $k$ for a simple Lie algebra $\g$ with dual Coxeter number $h^\vee$. In \cite[Section 2]{AGO} it is shown that
$c_{\g,\k}(k)=c_\g(k)-\sum_Sc_{\k_S}(j_S k)$ is strictly increasing as  a function of $k$.
By \cite[Proposition 12.12.c)]{Kac} we have $c_{\g,\k}(k)\geq 0$ if $k$ is a positive integer. By the coset construction, we have the finite decomposition property
of an irreducible highest weight $\widehat{\g}$-module $V$ of level $k$ only if the coset central charge is zero:
\begin{equation}\label{centralzero}
c_{\g,\k} (k)=0.
\end{equation}
 Hence the
decomposition in
question has a chance to be finite only if the level of the
$\widehat{\g}$-module $V$ is equal to~$1$, and if it is finite for
one of the $\widehat{\fg}$-modules of level~$1$, it is also finite
for all others.

A more conceptual argument, kindly suggested by the referee (see also \cite{A}), is the following: if the embedding is conformal then $\omega_\g-\omega_\k$ is in the maximal ideal of  $V^k(\g)$, hence there must be a singular vector  in $V^k(\g)$ of conformal weight $2$. This implies that there is  $\l$, with $\l$ either zero or a sum of at most two roots of $\g$, such that
$$
\frac{(2\bar\rho+\l,\l)-2(k+h^\vee)}{k+h^\vee}=2.
$$
Here $\bar\rho$ is a Weyl vector for $\g$ and $(\cdot,\cdot)$ is the normalized invariant form.
Since  $(\bar\rho,\a)\le h^\vee-1$  for any root $\a$ and $\Vert \l\Vert^2\le 8$, we see that the above equality implies that $
\frac{2h^\vee+4-2k}{k+h^\vee}\ge 2$, so $k\le 1$.

\vskip5pt
On the other hand, if we remove the integrability condition,  equality of central charges may happen at levels less than  $1$. Some examples of this
phenomenon were found in \cite{A}; it was also noted that sometimes one also gets the finite decomposition property. We would like to put the examples discovered by Adamovi\'c and Per\v{s}e \cite{A} in the framework of Kac-Wakimoto theory of admissible representations \cite{KacWW}.\par
Let $\ha=\h\oplus\C K\oplus\C d$ a Cartan subalgebra of $\ga$ ($\h$ is a Cartan subalgebra of $\g$). Let $\rho$ be a Weyl vector for $\ga$ (cf. \cite{Kac}).
Recall that an irreducible highest weight module $L(\l),\,\l\in\ha^*$ over an affine algebra $\ga$ is said to be {\it admissible} if
\begin{enumerate}
\item $(\l+\rho)(\a)\notin -\nat$, for each positive coroot $\a$;
\item the rational linear span of positive simple coroots equals the rational linear span of the coroots which are integral valued on $\l+\rho$.
\end{enumerate}
Admissible modules are classified in \cite{KacWW}; their importance lies in the fact that the character of these representations has modular invariance properties even though they are not necessarily integrable.
We are concerned about admissible representation by virtue of the following observation, which relies on the theory developed in \cite{KacWW}.\par
 Denote by  $g_\g(\l)$ the number, called the growth of $L(\l)$, defined by  formula (3.20b) of \cite{KacWW} for the weight $\l$ of $\ga$. It expresses the asymptotic behavior
of the character of $L(\l)$:  in the limit  ${t\to 0^+}$ of the real parameter $t$
$$ tr_{L(\l)} e^{2\pi t(d+z)} \sim b_\l\cdot e^{\frac{\pi}{12 t}g_\g(\l)}\quad(z\in \h)$$
where $b_\l=b_\l(z)$ is a nonzero function of $z$.
\par
\begin{theorem}\label{cond} Let $\g$ be  a simple Lie algebra and $\k=\sum_{S=0}^r\k_S$ a reductive equal rank maximal
subalgebra of $\g$.  Assume that $L(k\L_0),\,k\in\mathbb Q\setminus \ganz,$ is admissible and that any irreducible subquotient of $L(k\L_0)$ is  admissible as a $\ka$-module.  Then  the $\widehat{\g}$-module $L(k\Lambda_0)$
decomposes finitely w.r.t.  $\widehat{\k}$ if and only if \eqref{centralzero} and
\begin{equation}\label{eecc}
g_\g(k\L_0)=\sum_{S=0}^r g_{\k_S}(j_S k\L_{0_S})
\end{equation}
hold.
 \end{theorem}
\begin{proof}Consider $\mu\in \h_0^*$ and  define $U(\L,\mu)$ as in
\cite[12.12]{Kac}. Then $L(k\L_0)=\bigoplus\limits_i U(k\L_0,\l_i)\otimes L^{\ka}(\l_i)$ as  representation of $Vir\oplus\ka$, where $Vir$ is the coset Virasoro algebra.

If $L(k\L_0)$ decomposes finitely  then $U(k\L_0,\l_i)$ are finite dimensional hence the central charge of the coset Virasoro, which is $c_{\g,\k} (k)$, is zero, hence \eqref{centralzero} holds.

Next we prove \eqref{eecc}. Since $\k$ is an equal rank subalgebra of $\g$ and $L^\k(\l_i)$ occurs in $L(k\L_0)$, we have that $\l_i-\l_j$ is a sum of roots of $\widehat \k$. Since the growth of an admissible
 $\widehat{\k}$-module $L^\k(\l)$ depends only on the set of roots $\a$ of $\widehat \k$ such that $\l(\a^\vee)\in\ganz$, we see that the growth of $L^\k(\l_i)$ does not depend on $i$.
Denote it by $g'$. Since $L^\k(\sum_S j_S k\L_{0_S})$ clearly occurs in $L(k\L_0)$, we have that $g'=\sum_{S=0}^r g_{\k_S}(j_S k\L_{0_S})$.
 Setting  $g=g_\g(k\L_0)$ we have,
as $t\to 0^+$, 
$$
b_{k\L_0}(z)e^{\pi g/12t}(1+o(1))=(\sum_ib_{\l_i}(z))e^{\pi g'/12t}(1+o(1)).
$$
 If $g'<g$, we obtain by dividing by $e^{\pi g/12t}$,
as $t\to 0^+$, that $b_{k\L_0}(z)=0$ for all $z$.  If $g'>g$, we obtain by dividing by $e^{\pi g'/12t}$,
as $t\to 0^+$, that $(\sum_ib_{\l_i}(z))=0$, but then again $b_{k\L_0}(z)=0$ for all $z$, a contradiction.

Assume now that \eqref{centralzero} and \eqref{eecc} hold. If $\l$ is admissible, then, by Proposition 5 of \cite{KacWWW}, which holds under the additional assumption that the eigenspaces of $(\omega_\g)_0$ on $L(\l)$ are finite-dimensional, the $Vir$-modules $U(\l,\l_i)$ have finite Jordan-Holder series. The additional hypothesis is satisfied in the special case $\l=k\L_0$. Moreover each irreducible subquotient has central charge given by $c_{\g,\k} (k)$ and growth given by $g_\g(k\L_0)-\sum_{S=0}^r g_{\k_S}(j_S k\L_{0_S})$. Since they are both zero, we see that $U(k\L_0,\l_i)$ is finite dimensional, hence the decomposition is finite.
\end{proof}
\begin{rem}
Computing explicitly   equal rank pairs  $(\g,\k)$ having  levels satisfying  condition \eqref{centralzero} and   condition \eqref{eecc}, we obtain   just the following three cases
$$
\begin{tabular}{ c | l | l  }
 $k$ & $\g$ & $\k$ \\ \hline
  $-l+3/2$ & $B_l$ & $D_l$ \\
  $-5/3$ & $G_2$ & $A_2$ \\
  $-5/2$ & $F_4$ & $B_4$ \\
\end{tabular}
$$
\centerline{\text{\rm \small Table 1}}
The finiteness of the decomposition in these cases has been proved in \cite{A}
for  $(B_l,D_l),\,(G_2,A_2)$  and in  \cite{Perse} for $(F_4,B_4)$.
 Note that our methods do not cover the cases  when one of the levels involved is a negative integer (see e.g., \cite[Table 2]{A}).\end{rem}

\section{Representations of vertex algebras}\label{rva}
\subsection{Intertwining operators and simple currents}\label{simplecurrentext}
Let $V$ be a vertex algebra. A $V$-module   is  a vector  superspace $M$  endowed with a  parity preserving map $Y^M$ from
$V$ to the superspace of $End(M)$-valued fields
$$ a\mapsto Y^M(a,z)=\sum_{n\in\ganz}a^M_{(n)}z^{-n-1}
$$ such that
\begin{enumerate}
\item $Y^M(|0\rangle,z)=I_M$,
\item for $a,b \in V$, $m,n,k \in \ganz$,
\begin{align*}& \sum_{j\in\nat}\binom{m}{j}(a_{(n+j)}b)^M_{(m+k-j)}\\
&=\sum_{j\in\nat}(-1)^j\binom{n}{j}(a^M_{(m+n-j)}b^M_{(k+j)}-p(a,b)(-1)^nb^M_{(k+n-j)}a^M_{(m+j)}).
\end{align*}
\end{enumerate}

Given three $V$-modules $M_1$, $M_2$, $M_3$, an {\it intertwining operator of type $\left[\begin{matrix}M_3\\M_1\ M_2\end{matrix}\right]$} (cf. \cite{FZ}) is a   map $I:a\mapsto I(a,z)=\sum_{n\in \ganz}a^I_{(n)}z^{-n-1}$ from
$M_1$ to the space of $End(M_2,M_3)$-valued fields such that
for $a \in V$, $b \in M_1$, $m,n \in \ganz$,
\begin{align*}& \sum_{j\in\nat}\binom{m}{j}(a^{M_1}_{(n+j)}b)^I_{(m+k-j)}\\
&=\sum_{j\in\nat}(-1)^j\binom{n}{j}(a^{M_3}_{(m+n-j)}b^I_{(k+j)}-p(a,b)(-1)^nb^I_{(k+n-j)}a^{M_2}_{(m+j)}).
\end{align*}
\begin{example}\label{ex} Let $W$ be a vertex algebra and $V$ a simple subalgebra of $W$. Let $M_i, i=1,2,3,$ be $V$-submodules (with respect to the standard representation of $W$ to itself, restricted to $V$), and let $\pi: W\to M_3$ be a linear map, commuting with the action of $V$ on $W$. Let $a\in M_1$. Then
$$I(a,z)=\pi Y(a,z)_{|M_2},$$
is an intertwining operator of type $\left[\begin{matrix}M_3\\M_1\ M_2\end{matrix}\right]$.
\end{example}
\vskip5pt
We let $N^{M_3}_{M_1,M_2}$ be the dimension of the space of intertwining operators of type $\left[\begin{matrix}M_3\\M_1\ M_2\end{matrix}\right]$. When $N^{M_3}_{M_1,M_2}$ is finite, it is  usually called a {\it fusion coefficient}.

\begin{defi}\ Let $V$ be a vertex algebra. \par{\bf 1.} A simple current for $V$  is a $V$-module $S$ such that, for any irreducible $V$-module $M$, there is a unique irreducible $V$-module $M_S$ such that   $N_{S,M}^{M_S}$ is nonzero and also $N_{S,M}^{M_S}=1$.
\par {\bf 2.} Let $V$ be a vertex algebra equipped with a Virasoro vector.
A simple current extension of $V$ is a simple vertex algebra $W$ such that $V\hookrightarrow W$ is a conformal embedding and there is grading $W=\sum_{a\in D} W^a$ of $W$ by  an abelian group  $D$ such that $W^0=V,\,W^a\cdot W^b\subset W^{a+b}$ and $W^a$ is a simple current for $V$ for any $a\in D$.
\end{defi}

\begin{rem}  Simple currents are usually defined in terms of a tensor product between $V$-modules (see, for example, \cite{DLM} and references therein). They are defined by the condition that $S\otimes M$ is irreducible whenever $M$ is. We won't need this general theory.
\end{rem}



We are mainly interested in the vertex algebra
 $V_k(\g)$ when $\g$ is  a simple Lie algebra and $k\in\nat$.

 Choose a Cartan subalgebra $\h$ of $\g$ and  a set of simple roots $\{\a_1,\ldots,\a_n\}$ for the root system of $(\g,\h)$. Let  $\b$ be the Borel subalgebra corresponding to these choices and let $\theta$ be the highest root.

 If $\l\in \h^*$, we denote by $L^\g(\l)$ the irreducible representation of $V^{k}(\g)$ having a vector $v_\l$ such that
\begin{align}
&x_{(n)}v_\l=0\text{ for }x\in\g\text{ and }n>0,\label{highestone}\\
&x_{(0)}v_\l=\l(x)\text{ for }x\in\b.\label{highestwo}
\end{align}
Here $\l$ is extended trivially on $[\b,\b]$. If the action of $V^{k}(\g)$ pushes down to an action of $V_{k}(\g)$, we keep denoting by $L^\g(\l)$ the corresponding $V_{k}(\g)$ module. The module $L^\g(\l)$ is called the irreducible module of highest weight $\l$ and the vector $v_\l$ is called a highest weight vector.  Let $P$ denote the weight  lattice  of $\g$, and let $\{\omega_1,\ldots,\omega_n\}$ be the fundamental weights. Let $P_+=\sum_i \nat\omega_i$ be the set of dominant integral weights. Set $P^k_+=\{\l\in P_+\mid \l(\theta^\vee)\le k\}$.  The irreducible  $V_k(\g)$-modules are precisely  the irreducible  highest weight  modules $L^\g(\l)$ with $\l\in P^k_+$.
In particular $V_k(\g)$ has a finite number of irreducible representations. Exercise 13.35 of \cite{Kac} gives an explicit  formula for the fusion coefficents $N^{L^\g(\nu)}_{L^\g(\l),L^\g(\mu)}$.

 Write $\theta=\sum_{i} a_i \a_i$. It is well known (cf. \cite{DLM}) that if $a_i=1$, then $L^\g(k\omega_i)$ is a simple current for  $V_k(\g)$. Set $J=\{i\mid a_i=1\}$. To simplify notation, set
\begin{align}\label{uno}
 \mathbbm 1 &=V_k(\g),\\\label{oi}
 o_i&= L^\g(k\omega_i),\,\,(i\in J),
 \end{align}  and define
\begin{equation}\label{Omega}
\Omega_{\g}=\{ \mathbbm 1\}\cup\{o_i\mid  i\in J\}.\end{equation}  We call  the elements of $\Omega_{\g}$ the {\it special simple currents} for $V_k(\g)$.

 It turns out that the special simple currents  are related to the
 the center $\mathfrak z(G)$ of the connected simply connected Lie group $G$ corresponding to $\g$.
 let $P^\vee \subset \h$ denote the coweight  lattice  of $\g$, and let $\{\omega^\vee_1,\ldots,\omega_n^\vee\}$ be the fundamental coweights.

 Let $f:\Omega_{\g}\to P^\vee$ be defined by
 \begin{equation}\label{ff}
 f(o_i)=\omega_i^\vee,\quad f(\mathbbm 1)=0.
 \end{equation}
  Then we have the following bijection $\exp\circ f$:



\begin{equation*}
\begin{CD}\Omega_{\g} @> f >>P^\vee @>  \exp>>\mathfrak z(G).\end{CD}\end{equation*}
Thus $\Omega_\g$ acquires a group structure $(M,M')\mapsto MM'$ from this bijection; as notation suggests, $\mathbbm 1=V_k(\g)$ is the identity element. It is well known that, if $M_1,M_2\in \Omega_{\g}$ and $M_3$ is an irreducible $V_k(\g)$-module, then
\begin{equation}\label{grouplaw}
N^{M_3}_{M_1,M_2}=\delta_{M_3,M_1 M_2}.
\end{equation}

 This result easily extends to the case when $\g$ is reductive. First remark that
 if $\g$ is abelian, letting $\b=\g$, then \eqref{highestone} and \eqref{highestwo} define $L^\g(\l)$ for any $\l\in\g^*$. Moreover $L^\g(\l)$ is a simple current for $V_{k}(\g),\,k\in\C^*$. Indeed, in this case we have,
 $$
N^{L^\g(\nu)}_{L^\g(\l),L^\g(\mu)}=\d_{\nu,\l+\mu}.
 $$
We therefore set
 $$
 \Omega_{\g}=\{L^\g(\l)\mid \l\in\g^*\}
 $$
and call its elements {\it special simple currents}. Identifying $\Omega_{\g}$ and $\g^*$ gives a group structure to $\Omega_{\g}$ and \eqref{grouplaw} holds also in this case.

If, finally, $\g$ is reductive,
let $\g=\oplus_{i=0}^s\g_i$ be the decomposition into the center $\g_0$ and the  sum of  the simple ideals of $\g$. Fix ${\bf  k}\in\C^*\times\nat^s$.
Set, for $\l\in \h^*$, $L^\g(\l)=L^{\g_0}(\l_0)\otimes L^{\g_1}(\l_1)\otimes \dots \otimes L^{\g_s}(\l_s)$, where $\l_i=\l_{|\h\cap\g_i}$, and
\begin{equation}\label{Of}
\Omega_{\g}=\{\otimes_{i=0}^sM_i\mid M_i\in \Omega_{\g_i}\}.
\end{equation}
The elements of $\Omega_{\g}$ are simple currents that we call the {\it special simple currents} for $V_{\mathbf{k}}(\g)$.
The map $M_0\otimes M_1\otimes \dots \otimes M_s\mapsto (M_0,M_1,\dots, M_s)$ identifies $\Omega_{\g}$ and $\prod_{i=0}^s \Omega_{\g_i}$ thus giving a group structure to $\Omega_{\g}$ that makes \eqref{grouplaw} to hold also in this case.

 \subsection{Relative fusion ring} Let $W$ be a simple vertex algebra and $V$ a simple subalgebra of $W$. Assume that the standard representation of $W$, when restricted to $V$, decomposes
 into a direct sum of $T$-invariant irreducible modules. Write
 \begin{equation}\label{id}
 W=\bigoplus_{j\in \mathcal J} W_j,
 \end{equation}
 for the decomposition of $W$ as a $V$-module into isotypic components. Here $\mathcal J$ is the set of all $j$ in the set of isomorphism classes of irreducible representations of $V$ such that $W_j\ne\{0\}$. As in the previous section we let $\uno$ be the isomorphism class of $V$ as a $V$-module. 
 Define $c_{ij}^k$ to be $1$ if $W_i\cdot W_j$ (which is a $V$-submodule of $W$, due to the Borcherds identity) intersects non-trivially $W_k$ and $0$ otherwise.
\begin{defi} The relative fusion ring $FR(W,V)$ is the free  $\ganz$-module with basis $ \mathcal J$ and product
\begin{equation}\label{pfp} i\cdot j=\sum_{k\in \mathcal J} c_{ij}^k k.\end{equation}
\end{defi}

\begin{example}
The complete
reducibility Theorem 10.7 from \cite{Kac} implies that \eqref{id}
holds    in the case of  the conformal embedding of an affine vertex algebra
$V=V_{\mathbf k}(\k)\hookrightarrow V_1(\g)$ with $\g$ a simple Lie algebra, $\k$ a reductive algebra as in \eqref{reductive} and $\mathbf k\in C^*\times \nat^s$.  Thus, we can define the relative fusion ring $F(V_1(\g),V_{\mathbf k}(\k))$.

Formula \eqref{id} holds also in the case of the conformal embedding $V_{\bf j}(\k)\hookrightarrow F(\bar A)$ for some vector space  $A$ (see Section 3.1). Indeed, by Proposition \ref{symmthm} above, \eqref{id} is given by the results of \cite{CKMP} (see Theorems \ref{decoeabeliani} and \ref{finalehs} below). Thus, we have the relative fusion ring $FR(F(\bar A),V_{\bf j}(\k))$.
\end{example}

By Lemma \ref{nonvanishing} (a), (c), $FR(W,V)$ is a unital commutative associative ring (with identity element $\uno$), such that $i\cdot j\ne 0 $ for all $i,j\in \mathcal J$.

\begin{defi} Let $\k$ be a reductive complex Lie algebra, $\mathfrak{t}$ a Cartan subalgebra of $\k$.
For $\l\in \mathfrak{t}^*$, set
$$\l^*=-w_0(\l),$$ where $w_0$ is the longest element of the  Weyl group of $\k$.
\end{defi}

\begin{lemma}\label{nonz}Let $W$ be a simple vertex algebras and $V$ a vertex subalgebra such that $\eqref{id}$ holds.

\begin{enumerate}
\item  For each $j\in \mathcal J$ there exists $j^*\in \mathcal J$ such that $c_{j,j^*}^\uno\ne 0$.
\item Let $V=V_{\mathbf k}(\k)$ with $\k$ as in \eqref{reductive} and $\mathbf{k}\in \C^*\times \nat^s$. If $j=L^\k(\l)$ then then there exists a unique $j^*$ as in (1), and $j^*=L^\k(\l^*)$.
\end{enumerate}
\end{lemma}
\begin{proof} To prove (1) assume that for $j\in \mathcal J$ the contrary case holds. Then subspace $W\cdot M_j$ is a non-trivial proper ideal of $W$  if $j\ne \uno$. The case $j=\uno$ is obvious.

To prove (2) we observe that, as a consequence of Verlinde's formula (see e. g. \cite[(5.4)]{Wakimoto}), $L^\k(\l^*)$ is the unique irreducible $V$-module such that
$N_{L^\k(\l), L^\k(\l^*)}^{L^\k(0) }=1$ and \break  $N_{L^\k(\l), M}^{L^\k(0) }=0$ for any other irreducible $V$-module $M$. Applying Example \ref{ex}, with $M_1=j=L^\k(\l)\subset W_j$, $M_2=j^*=L^\k(\mu)\subset W_{j^*}$ and $\pi$ the projection $W\to V$, we see that $N_{L^\k(\l), L^\k(\mu)}^{L^\k(0) }\ne0$ so $\mu=\l^*$.
\end{proof}



Assume $V=V_{\mathbf k}(\k)$ and define a $\ganz$-valued pairing on $FR(W,V)$ by:
\begin{equation}\label{pairing}(i,j)=\d_{i,j^*}.\end{equation}
Lemma \ref{nonz} implies the following proposition.
\begin{prop} \label{bilinearfr}(a). The bilinear form induced by \eqref{pairing} on $FR(W,V)$ is associative, i.e.  $(i\cdot j,k)=(i,j\cdot k)$.\par
(b). The map $i\to i^*$ is an involution of the ring $FR(W,V)$, leaving the bilinear form invariant.
\end{prop}

\begin{defi}\label{SW} Set $S_W^\k= \mathcal J\cap \Omega_\k$ and let
$SR(W,V)$ be the free $\ganz$-submodule of \break $FR(W,V)$ with basis $S^\k_W$.
\end{defi}
\begin{lemma}\label{argfeng}
$SR(W,V)$ is a subring of $FR(W,V)$. Moreover  $S^\k_W$ is a subgroup of $\Omega_\k$ and $SR(W,V)$ is isomorphic as a ring with basis to the group ring $\ganz[S^\k_W]$ of $S^\k_W$.
\end{lemma}
\begin{proof}
Let $i,j\in S_W^\k$.  Let $W_k$ be such that $W_k\cap (W_i\cdot W_j)\ne\{0\}$. As shown in Example \ref{ex}, $N^{k}
_{i,j}\ne 0$.
It follows from \eqref{grouplaw} that $k=ij\in\Omega_\k$ and $i\cdot j=k$. Thus $i\cdot j=ij\in S_W^\k$.

To prove the second claim first observe that we already proved above that $S_W^\k$ is a submonoid of $\Omega_\k$. It is clear that, if $L^\k(\l)\in \Omega_\k$, then $L^\k(\l^*)\in\Omega_\k$. Thus,  by Proposition  \ref{bilinearfr},  if $j\in S_W^\k$, then $j^*\in S_W^\k$. Since $\uno=j\cdot j^*=jj^*$, $S_W^\k$ is a subgroup of $\Omega_\k$.
\end{proof}


\begin{cor}\label{sumsimpleisext}If $W$ is a sum of special simple currents for  $V_{\mathbf k}(\k)$, then it is a simple current extension of $W_\uno$.
\end{cor}
\begin{proof}
$W$ is graded by the abelian group $S^\k_W$.
\end{proof}

\begin{cor}\label{subgroup}Set $V=V_{\mathbf k}(\k)$ and assume that $W_\uno=V$.
\begin{enumerate}
\item
 $U=\sum_{j\in S^\k_W}W_j$ is a simple current extension of $V$ and $U=\sum_{j\in S^\k_W}W_j$ is a multiplicity free decomposition.
\item
The intermediate simple vertex algebras $V\subset U'\subset U$ are in one to one correspondence with the subgroups of $S^\k_W$.
\end{enumerate}
\end{cor}
\begin{proof}
Since $U\cdot U\subset U$, $U$ is a vertex subalgebra. If $I$ is an ideal of $U$ and $W_j\cap I\ne \{0\}$, then, since $W$ is simple, $W_{j^*}\cdot (W_j\cap I)\ne \{0\}$, hence $V\subset I$. It follows that $I=U$. This proves that $U$ is simple.
 By Corollary \ref{sumsimpleisext}, it is a simple current extension.
 The character group $\widehat S^\k_W$ of $S^\k_W$ acts on $U$ by $\chi\cdot v=\chi(j)v$ for $v\in W_j,\chi\in \widehat S^\k_W$. By Theorem 3 of \cite{DM}
 (or Theorem 1.1 of \cite{KacR}), $W_j$ is an irreducible $V$-module. This proves (1).

If $U'$ is a simple intermediate vertex subalgebra, then $S^\k_{U'}$ is a subgroup of $S^\k_W$, hence we have a map from simple intermediate vertex subalgebras to the set of subgroups of $S^\k_W$. This map is surjective: if $H\subset S^\k_W$ is a subgroup, then $U'=\sum_{j\in H}W_j$ is an intermediate vertex subalgebra since $U'\cdot U'\subset U'$. The same argument employed for $U$ shows that $U'$ is simple. The map is obviously injective by the multiplicity free property. This proves (2).
\end{proof}

We will compute explicitly $S^\k_{V_1(\g)}$ and $SR(V_1(\g),V_{\mathbf{j}}(\k))$ in  Section 8  for any maximal conformal embedding $\k\hookrightarrow  \g$ with $\g$ of classical type.


\section{Symmetric pairs}\label{symmpairs}
We will consider conformal embeddings $V_{\mathbf{k}}(\k)\subset V_{1}(\g)$ closely related to symmetric pairs. In this section we review the basic features of such pairs.

Let  $(\fa,\r)$ be a symmetric pair and write $\aa=\r\oplus\p$ for the corresponding eigenvalue decomposition.
   As in Section \ref{confemb}, we let $\fa=\oplus \fa_s$ be the decomposition of $\fa$ into $\s$-indecomposable ideals and $(\cdot,\cdot)$ the Killing form of $\fa$. Let $\h_0$ be a
 Cartan subalgebra of $\r$ and $\h$ the centralizer of $\h_0$ in $\aa$.
    Write $\h=\h_0\oplus\h_\p$ for the orthogonal decomposition of $\h$. Denote
by $\D(\r)$ the set of $\h_0$-roots of $\r$ and by $\D(\p)$ the set
of $\h_0$-weights of $\p$. A choice of an element in $\h_0$ which is regular for $\aa$ defines a set of positive roots $\Dp$ for  the set of $\h$-roots of
$\aa$. Clearly the same element defines  a set of positive roots $\Dp(\r)\subset\D(\r)$ such that  $\Dp(\r)\subset\Dp_{|\h_0}$. Let
$\b_0$ be the corresponding Borel subalgebra and set $\Dp(\p)=\D(\p)\cap\Dp_{|\h_0}\setminus\{0\}$. \par
Clearly $\D=\cup_s \D(\aa_s)$, where $\D(\aa_s)$ is the set of roots for $\aa_s$ with respect to $\h\cap\aa_s$. If $\a\in \D$, then $\a\in\D(\aa_s)$ for some $s$ and we say that it is long if its length is largest among all root lengths in $\D(\aa_s)$. Likewise, if $\l\in\D(\r)\cup\D(\p)$, then $\l \in (\D(\r)\cup\D(\p))\cap (\D(\aa_s)_{|\h_0})$ for some $s$. We say that $\l$ is a long weight if there is a long root of $\a\in\D(\aa_s)$ such that $\Vert \l\Vert=\Vert\a\Vert$.

We say that a nonzero weight $\l\in\D(\r)\cup\D(\p)$ is a complex weight if $\l\in\D(\r)\cap\D(\p)$. If $\l$ is not complex, then we say that it is a compact weight if $\l\in\D(\r)$,  and a noncompact one if $\l\in\D(\p)$.

Fix a maximal isotropic subspace $\h_\p^+$ of $\h_\p$. Let moreover $\h_\p^-$ be
 a complementary isotropic subspace (if $\p$ is odd dimensional) also choose a unit vector $h$ such that $\h_\p=\C h\oplus\h^+_\p\oplus\h^-_\p$.
 Fix a basis of $\p$ by extending a basis of $\h_\p$ with weight
vectors $x_\a$,   $x_{-\a}$ for $\a\in\Dp(\p)$ chosen in such a way that $(x_\a,x_{-\a})=1$.
Order the basis of $\p$ as follows: $h$ (if $\p$ is odd-dimensional), a basis $\{h_i^+\}$ of $\h_\p^+$ , $x_\a,\,\a\in\Dp(\p),$ a basis $\{h_i^-\}$ of  $\h_\p^-$ dual to the chosen basis of $\h_\p^+$,  $x_{-\a},\,\a\in\Dp(\p).$ \par

Let $\p_\a$ denote the weight space associate to $\a$. Introduce the following notation.
\begin{align*}
&\Sigma = \text{set of $\b_0$-stable abelian subspaces of $\p$},\\
&\Sigma^{even}= \text{subset of even dimensional subspaces in $\Sigma$,}\\
&\Phi_{\ll}=\{\a\in\D(\p)\mid \p_\a\subset \ll\}, (\ll\in\Sigma),\\
&\Phi_\ll^\pm=\Phi_\ll\cap \D^{\pm}(\p).
\end{align*}

We will need an explicit description of the weights in $\D(\p)$. Since $\D(\p)=\cup_s \D(\p)\cap\D(\aa_s)$, we can assume that $\s$ is indecomposable. Recall that a symmetric pair
$(\fa,\r)$ is called irreducible if the corresponding involution $\s$ is indecomposable.

With this assumption the weights of $\p$ can be described  using the relationship between irreducible symmetric pairs and affine diagrams, that we now recall (see \cite{MMJ} and \cite{MMJA} for details).  One can associate to an indecomposable automorphism $\s$ an affine Kac-Moody Lie algebra $\widehat L(\aa,\s)$. Recall that $\widehat L(\aa,\s)=L(\aa,\s)\oplus \C d\oplus \C K$ where
$$
L(\aa,\s)= \left(\C[t,t^{-1}]\otimes\r \right)\oplus \left(t^{1/2}\C[t,t^{-1}]\otimes \p\right),
$$
$K$ is a central element, and $d$ acts as $td/dt$. Let $\ha=\h_0\oplus \C K\oplus \C d$ be the Cartan subalgebra of $\widehat L(\aa,\s)$ and  $\Da$ be the corresponding set of roots.

Fix the following set of positive roots in $\Da$:
$$
\Dap=\Dp(\r)\cup\{\a\in\Da\mid \a(d)>0\}.
$$

Let
 $\Pia$ be the corresponding set of simple roots.   In what follows we will denote by the same letter both diagrams and the corresponding sets of nodes. In particular $\Pia$ will also denote the Dynkin diagram of $\widehat L(\aa,\s)$.

If $\gamma\in \Da$, we write $\gamma=\sum_{\a \in \Pia} c_\a(\gamma)\a$. Let also $\a\mapsto \bar\a$ be the restriction map from $\ha$ to $\h_0$. Let $\d'\in\ha^*$ be defined by $\d'(d)=1$ and $\d'(K)=\d'(\h_0)=0$. Then $\Dap$ can be explicitly described as
$$
\Dp(\r)\cup\{s\d'+\eta\mid s\in\nat\backslash\{0\},\eta\in\D(\r)\}\cup\{(\tfrac{1}{2}+s)\d'+\eta\mid s\in \nat, \eta\in\D(\p)\}.
$$

 If $\a$ is a node of $\Pia$, let $a_\a$ be the label  associated to it in  Tables Aff $r$ of \cite{Kac} ($r=1,2$).     If $\aa$ is simple, according to Kac's classification of automorphisms (\cite[Ch. 8]{Kac}), one can encode $\s$ by a pair $(\eta;\mathbf{s'})$ where $\eta$ is an automorphism of the diagram of $\aa$ and $\mathbf{s'}=(s'_\a)_{\a\in \Pia}$ is a set of labels $s'_\a\in \ganz^+$.
Since $\s$ is an involution, $\eta$ and $\mathbf s'$ must satisfy the following constraints:  $s'_\a\in\{0,1\}$ for all $\a\in \Pia$, and $r\sum_{\a\in \Pia}s'_\a a_\a=2$. Define
\begin{equation}\label{salfa1}s_\a=\frac{s'_\a}{2}.\end{equation}
If $\aa$ is not simple then $\aa=\mathfrak  s\oplus \mathfrak s$, then $\Pia=\{\half \d'-\theta,\a_1,\ldots \a_l\}$ where $\{\a_1,\ldots, \a_l\}, \theta$ are a set of simple roots and the highest root of $\mathfrak s$, respectively. In this case we set
\begin{equation}\label{salfa2}
s_\a=\begin{cases}\half\quad&\text{if $\a=\half\d'-\theta$}\\0\quad&\text{otherwise}\end{cases}\end{equation}.\par
Let $P=\{\a\in \Pia\mid s_\a\ne 0\}$. Note that $P$ has at most two elements.
Set $$\Da^1=\{\gamma\in \Da\mid \sum_{\a\in P}c_\a(\gamma)=1\}.$$ By the explicit description of the set of simple roots corresponding to $\Dap$ given in \cite[Proposition 3.2]{MMJ}, one obtains  that
$$
\Da^1=\{\half\d'+\eta\mid \eta\in \D(\p)\},
$$
thus the restriction map $\gamma \mapsto \bar \gamma$ establishes a bijection between $\Da^1$ and $\D(\p)$.

If $P=\{p\}$ and $\gamma \in\Da^1$, the bijection is given by
\begin{equation}\label{weights}
\gamma\mapsto \bar \gamma=\sum_{\a\ne p}(c_\a(\gamma)-\frac{a_\a}{a_p}) \a.
\end{equation}
Recall that $P$ has only one element if and only if   $\r$ is semisimple. In such cases $\p$ is irreducible as a $\r$-module. We let  $\theta_{\p}$ be the highest weight of $\p$ with respect to $\Dp(\r)$. Since $(\cdot,\cdot)$ is nondegenerate when restricted to $\p$, it is clear that $\p$ is a self dual $\r$-module. It follows that $-\theta_\p$ is the lowest weight of $\p$. In the above correspondence $-\theta_\p$ clearly corresponds to $p$ so
\begin{equation}\label{highestp}
\theta_\p=\sum_{\a\ne p}\frac{a_\a}{a_p} \a.
\end{equation}

In the cases when $P$ has two elements, the explicit description of the map $\gamma\mapsto\bar \gamma$ is as follows. Let $\Pi_{\aa}$ be the set of simple roots for $\aa$ corresponding to $\Dp$. If $P=\{p,q\}$, there is a simple root $\a_{\p}\in\Pi_{\aa}$   such that
  $\Pi=\Pi_{\aa}\backslash\{\a_{\p}\}$. Moreover $\aa$ is simple and $\h\cap\aa=\h_0\cap\aa$. Let $\theta$ be the highest root of $\aa$ with respect to $\Dp$. One can always set up the bijection in such a way that $\bar p=-\theta$, $\bar q=\a_{\p}$, hence, if $\gamma\in\Da^1$,
  \begin{equation}\label{weights2}
\gamma\mapsto \bar \gamma=\sum_{\a\not\in  P}c_\a(\gamma) \a - c_p(\gamma)\theta + c_q(\gamma)\a_{\p}.
\end{equation}

 Let $\Wa$ be the Weyl group of $\widehat L(\aa,\s)$.
If $w\in\Wa$, set
$$
N(w)=\{\a\in\Dap\mid w^{-1}(\a)\in-\Dap\}.
$$

Set $\Wab=\{w\in \Wa\mid N(w)\subset \Da^1\}$. In \cite{IMRN} Êit has been shown that there is a bijection $\ll\mapsto w_\ll$ from $\Sigma$ to $\Wab$ such that, if $\ll$ is a $\b_0$-stable abelian subspace of $\p$, then $\Phi_\ll=\{-\bar\be\mid\be\in N(w_\ll)\}$.\par
Let $\Pi$ be the set of simple roots for $\r$ corresponding to our choice of $\Dp(\r)$. It turns out that $\Pi=\Pia\setminus P$. In particular, to each simple ideal  $\r_S$ corresponds
a connected component $\Pi_S$ of $\Pi$.

If $\a\in \Pi_S$, set $\omega_\a^\vee$ to be the unique element in $\h_0\cap \r_S$, such that $\be(\omega^\vee_\a)=\delta_{\a,\be}$ for all $\be\in \Pi$.

We say that a pair $(\aa,\r)$ is hermitian symmetric if it is  irreducible and  $\r$ is reductive but not semisimple. In such a case,   we let $\varpi$  be the element of $\h=\h_0$ defined by $\a(\varpi)=\d_{\a,\a_\p}$ for all $\a\in\Pi_\aa$. Then $\C\varpi$ is the center of $\r$. For notational convenience, we set $\varpi=0$ in the semisimple case.
 \begin{lemma}\label{sumcoweights} Assume that $(\aa,\r)$ is irreducible.
Let $\ll\in\Sigma$ and assume that there is $h\in \h_0,\,h\ne 0,$ such that  $\gamma(h)=1$ for $\gamma\in\Phi_\ll$  and $\gamma(h)=0$ for $\gamma\in \D(\p)\setminus(\Phi_\ll\cup(-\Phi_\ll))$.

Then
 there are   the following possibilities:
 \begin{enumerate}
 \item \label{sumofone} there exists $\a\in \Pi$ such that $h=\omega_{\a}^\vee+(\epsilon-1-\a_\p(\omega_{\a}^\vee))\varpi,\,\epsilon=0,1$;
 \item \label{sumoftwo}there exist
 $\a,\a'\in \Pi$ such that  $h=\omega_{\a}^\vee+\omega_{\a'}^\vee+(-1-\a_\p(\omega_{\a}^\vee+\omega_{\a'}^\vee))\varpi$;
 \item $h=\pm\varpi$.
 \end{enumerate}
 \end{lemma}
\begin{proof} Since $\ll$ is $\b_0$-stable, there exists  an irreducible component $\p_0$ of $\p$ whose highest weight $\l$ is such that $\l(h)=1$. Let $\mu$ be the corresponding lowest weight. We can find a string $\l_0,\l_1,\dots, \l_N$ of weights of $\p$ such that $\l_0=\mu$, $\l_N=\l$ and, if $t>0$, $\l_t=\l_{t-1}+\a$ for some $\a\in \Pi$. If $\p$ is irreducible, then $\l=\theta_\p,\,\mu=-\theta_\p$ and
by \eqref{highestp}, $\l_N-\l_0=2\sum_{\a\ne p }\frac{a_\a}{a_p} \a$. Otherwise, either  $\l=\theta,\,\mu=\a_\p$ or $\l=-\a_\p,\,\mu=-\theta$, so that
$\l_N-\l_0=\sum_{\a\ne \a_\p }a_\a\a$.
In both cases, for any $\a\in \Pi$, there is $t$ such that $\l_{t+1}=\l_t+\a$.

 Let $t_0$ be the first index such that $\l_{t_0}(h)=1$. Since $\ll$ is $\b_0$-stable, we have that $\l_t(h)=1$ for all $t>t_0$. Assume first
 that $\mu(h)=-1$.
 Let  $s_0$ be the largest index such that $\l_{s_0}(h)=-1$. This implies that $-\l_{s_0}\in\Phi_\ll$,
  hence the fact that $\ll$ is $\b_0$-stable implies as above that $\l_t(h)=-1$ for $t<s_0$. Since $\l_t(h)\in\{-1,0,1\}$,
we have two cases. Either $t_0=s_0+1$ or $\l_t(h)=0$ for $s_0<t<t_0$.
In the first case let $\a\in \Pi$ be such that $\l_{t_0}=\l_{s_0}+\a$. We obtain that $\a(h)=2$ and $\be(h)=0$ for $\be \in \Pi,\,\be\ne \a$.
 Thus $h=2\omega_{\a}^\vee+(-1-2\a_\p(\omega^\vee_{\a}))\varpi$. In the second case let $\a,\a'$ be such that $\l_{s_0+1}=\l_{s_0}+\a$ and $\l_{t_0}=\l_{t_0-1}+\a'$. Then $\a(h)=\a'(h)=1$ and $\be(h)=0$ for $\be\in\Pi,\,\be\ne \a,\a'$. If $\a=\a'$,
 we have $h=\omega_\a^\vee+(-1-\a_\p(\omega_\a^\vee))\varpi$. If $\a\ne \a'$, we have $h=\omega_{\a}^\vee+\omega_{\a'}^\vee+(-1-\a_\p(\omega_{\a}^\vee+\omega_{\a'}^\vee))\varpi$ as wished.

 Assume now that $\mu(h)=0$. In this case, if $\l_{t_0}=\l_{t_0-1}+\a$, we have $h=\omega^\vee_{\a}-\a_\p(\omega_\a^\vee)\varpi$ if $\be=\a_\p$ and $h=\omega^\vee_{\a}+(-1-\a_\p(\omega_{\a}^\vee))\varpi$ otherwise.

 Finally, if $\mu(h)=1$, then $h=\pm\varpi$.
\end{proof}

 If $\ll$ and $h$ are  as in Lemma \ref{sumcoweights}, then we write that $\ll=\ll(h)$.
  Let $\theta_S$ be the highest root of $\r_S$ with respect to $\Dp(\r)$. Write $\theta_S=\sum_{\a\in \Pi_S} a^S_\a\a$.

 \begin{lemma} \label{fundcoweights} \
 \begin{enumerate}
 \item If $\ll=\ll(h)$  with $h$ as in Lemma \ref{sumcoweights} \eqref{sumofone}, let $\Pi_S$ be  the component of $\Pi$ containing $\a$. Then either $a_\a=a_p$ for some $p\in P$ or $a_\a=\sum_{p\in P}a_p$. Moreover $a_\a^S=1$. \label{fund1}
 \item\label{sommadue} If $\ll=\ll(h)$  with $h$ as in Lemma \ref{sumcoweights} \eqref{sumoftwo}, let $\Pi_S,\Pi_{S'}$ be the components of $\Pi$ such that $\a\in \Pi_S$, $\a'\in \Pi_{S'}$. Then $a_\a=a_{\a'}=a_\a^S=a_{\a'}^{S'}=1$ and $S\ne S'$.\label{fund2}
  \end{enumerate}
   \end{lemma}
  \begin{proof}
 Let us prove \eqref{fund1}. If $\p$ is irreducible, then $\theta_\p(h)=1$ hence, by \eqref{highestp}, $a_p=a_\a$. If $\p$ is reducible then $\theta(h)\in\{0,1\}$. If $\theta(h)=0$  then $a_\a+\a_\p(h)=0$ so $a_\a=1=a_p$ for some $p\in P$. If $\theta(h)=1$ then $1=a_\a+\a_\p(h)$ so $a_\a=2=\sum_{p\in P}a_p$.

 Let now $\be$ be the lowest weight of an irreducible component of $\p$ such that $\be(h)=-1$. We have that $\be=\bar p$ for some $p\in P$.  We use the known fact that $p+\theta_S\in \Da^1$. It follows that $\be+\theta_S\in\D(\p)$. Observe that $(\be+\theta_S)(h)=a_\a^S-1$. If $a_\a^S>1$, then $ a_\a^S-1= 1$, thus $\be+\theta_S\in\Phi_\ll$. On the other hand $-\be\in \Phi_\ll$ so $-\be+(\be+\theta_S)$ is a root of $\r$. This contradicts the fact that $\ll$ is abelian.

  Let us prove \eqref{fund2}. Let $\l,\mu$ be respectively the highest and lowest weight  of an irreducible component of $\p$ such that $\l(h)=1$. In this case we have $\mu(h)=-1$. Then, if $\p$ is irreducible,
  $1=\l(h)=\sum_{\gamma\ne p }\frac{a_\gamma}{a_p} \gamma(h)=\frac{a_\a+a_{\a'}}{a_p}$ hence $a_\a+a_{\a'}=a_p$. It follows that $a_p=2$ and $a_\a=a_{\a'}=1$. If $\p$ is reducible, under our hypothesis, we have that $\theta(h)=1$ hence $a_\a+a_{\a'}-1=1$. We obtain $a_\a=a_{\a'}=1$ as well. Choose $p\in P$  so that $\bar p=\mu$. If $S\ne S'$, it is easy to check that $p+\theta_S+\theta_{S'}\in \Da^1$, so $\mu+\theta_S+\theta_{S'}\in\D(\p)$. Since $(\mu+\theta_S+\theta_{S'})(h)=a_\a^S+a_{\a'}^{S'}-1$, we find that $a_\a^S+a_{\a'}^{S'}=2$, so $a_\a^S=a_{\a'}^{S'}=1$.

   It remains to check that $S\ne S'$. If, on the contrary, $S=S'$, use the known fact that $p+\theta_S\in \Da^1$. Hence $\mu+\theta_S\in \D(\p)$. Since $(\mu+\theta_S)(h)=a_\a^S+a_{\a'}^{S}-1=1$, we have $\mu+\theta_S\in\ll$. But also $-\mu\in\Phi_\ll$ and $-\mu+(\mu+\theta_S)$ is a root of $\r$. This contradicts the fact that $\ll$ is abelian.
  \end{proof}

 \vskip10pt
 Let $\nu:\h\to\h^*$ be the  identification via the Killing form of $\aa$. If $\a\in\Pi_S$, let $\omega_\a$ be the corresponding fundamental weight of $\r_S$. We need one more general notation, which will be used often in the rest of the paper.
If $M$ is a set of weights of  a Lie algebra, we let
$$\langle M \rangle=\sum_{\mu\in M}\mu.$$

 \begin{prop}\label{aresimplecurrents} Assume that  $\ll\in \Sigma$ is such that $\ll=\ll(h)$ for some $h\in\h_0$. Then
 \begin{enumerate}
 \item  $\langle \ll(h) \rangle =j_S\omega_\a+\frac{|\Phi^+_\ll|-|\Phi_\ll^-|}{\dim(\p)}\nu(\varpi)$ if $h$ is as in Lemma \ref{sumcoweights}, (\ref{sumofone}).\label{unosolo}
  \item $\langle \ll(h) \rangle =j_S\omega_\a+j_{S'}\omega_{\a'}^{}+ \frac{|\Phi^+_\ll|-|\Phi_\ll^-|}{\dim(\p)} \nu(\varpi)$ if $h$ is as in Lemma \ref{sumcoweights}, (\ref{sumoftwo}).\label{due}
 \end{enumerate}
  \end{prop}
 \begin{proof} Let us prove \eqref{unosolo}. Let $\r^\a$ be the Levi component of the parabolic of $\r$ defined by $\omega_\a^\vee$.
Observe that if  $\be\in\Phi_\ll$ and $\gamma\in \Pi\cap\D(\r^\a)$, then clearly $\be+\gamma\in\Phi_\ll$ and also  $\be-\gamma\in\Phi_\ll$, since $\gamma(h)=0$. It follows that there
 is a 1-dimensional representation  of $\r^\a$ whose weight is $\langle \Phi_\ll\rangle$, hence  we have that
$\langle \Phi_\ll\rangle=x\omega_\a + z\nu(\varpi)$.  We now prove that if $\varpi\ne 0$, then $z=\frac{|\Phi^+_\ll|-|\Phi_\ll^-|}{\dim(\p)}$.  Indeed, on one hand
$\langle \Phi_\ll\rangle(\varpi)=|\Phi^+_\ll|-|\Phi_\ll^-|$. On the other hand, since $\omega_\a(\varpi)=0$,
$$\langle \Phi_\ll\rangle(\varpi)=z \nu(\varpi)(\varpi)=z tr(ad_\p(\varpi)^2)=z\dim \p.
$$
Recall that $h=\omega^\vee_\a+y\varpi$ with $y$ explicitly given in Lemma \ref{sumcoweights}. We now show that $y=2z$.
Indeed,
$$\sum_{\gamma\in\Dp(\p)}\gamma(h)=\sum_{\gamma\in\Dp(\p)}\gamma(\omega_\a^\vee)+y\sum_{\gamma\in\Dp(\p)}\gamma(\varpi).$$
On one hand, the left hand side equals $|\Phi^+_\ll|-|\Phi_\ll^-|$. On the other hand, $\sum_{\gamma\in\Dp(\p)}\gamma$ is a multiple of $\nu(\varpi)$, hence vanishes on $\omega_\a^\vee$, whereas
$\sum_{\gamma\in\Dp(\p)}\gamma(\varpi)=\dim(\p)/2$.\par

Note that $\half tr(ad_\p(h)^2)=\dim \ll$, hence, by \eqref{jS},
$$
\dim \ll=j_S(\omega_\a^\vee,\omega_\a^\vee)_S+2z^2\nu(\varpi)(\varpi).
$$
Note also that $\langle \Phi_\ll\rangle(h)=\dim\ll$.
Since $a_\a^S=1$, we have that $(\a,\a)_S=2$ so $(\omega_\a^\vee,\omega_\a^\vee)_S=(\omega_\a,\omega_\a)_S=\omega_\a(\omega_\a^\vee)$.
We therefore obtain that $$x(\omega_\a,\omega_\a)_S+zy\nu(\varpi)(\varpi)=\dim \ll=j_S(\omega_\a,\omega_\a)_S+2z^2\nu(\varpi)(\varpi),$$ hence $x=j_S$ as claimed.

Let us prove \eqref{due} in the semisimple case. As above we note that $\langle \Phi_\ll\rangle$ is the weight of   a one-dimensional representation of $\r^\a\cap\r^{\a'}$, so that  $\langle \Phi_\ll\rangle =x\omega_\a+y\omega_{\a'}$.
By \eqref{highestp} and Lemma \ref{fundcoweights}, \eqref{sommadue}, we know that $\theta_\p(\omega_\a^\vee)=\theta_\p(\omega_{\a'}^\vee)=\half$. If $\be,\gamma\in\D(\p)$, let us write $\be\le \gamma$ if $\gamma-\beta=\sum_{\gamma\in \Pi}n_\gamma\gamma$ with $n_\gamma\in\nat$. Since $\be\le \theta_\p$, then $\be(\omega_\a^\vee)\le \half$ and $\be(\omega_{\a'}^\vee)\le\half$. Moreover $\be(\omega_{\a}^\vee)\in  \half+\ganz$, $\be(\omega_{\a'}^\vee)\in  \half+\ganz$. Since $\be(\omega_\a^\vee+\omega_{\a'}^\vee)=1$ for $\be\in\Phi_\ll$, then $\be(\omega_\a^\vee)=\be(\omega^\vee_{\a'})=\half$. Thus
$$
\langle\Phi_\ll \rangle(\omega_{\a}^\vee)=\langle \Phi_\ll \rangle(\omega_{\a'}^\vee)=\frac{\dim \ll}{2}.
$$
Since $S\ne S'$,
$$\half tr(ad_\p(\omega_\a^\vee+\omega^\vee_{\a'})^2)=\dim\ll=j_S(\omega_\a^\vee,\omega_\a^\vee)_S+j_{S'}(\omega^\vee_{\a'},\omega^\vee_{\a'})_{S'}.
$$
Since $a_\a^S=a_{\a'}^{S'}=1$, we have $(\omega^\vee_{\a},\omega^\vee_{\a})_S=(\omega_{\a},\omega_{\a})_S$, $(\omega^\vee_{\a'},\omega^\vee_{\a'})_{S'}=(\omega_{\a'},\omega_{\a'})_{S'}$, and
$$
\langle \Phi_\ll \rangle(\omega_{\a'}^\vee)=x(\omega_{\a},\omega_{\a})_S,\quad \langle \Phi_\ll \rangle(\omega_{\a'}^\vee)=y(\omega_{\a'},\omega_{\a'})_{S'}.
$$
To conclude we need only to check that $j_S(\omega_\a,\omega_\a)=j_{S'}(\omega_{\a'},\omega_{\a'})_{S'}$. Let $\be\in \D(\p)$. Since $-\theta_\p\le\be\le
 \theta_\p$, we have $-\half=-\theta_\p(\omega^\vee_{\a})\le\be(\omega_\a^\vee)\le\theta_\p(\omega_\a^\vee)=\half$ and\break  likewise for $\omega_{\a'}^\vee$.  So $$j_S(\omega_\a,\omega_\a)_S=\half tr(ad_\p(\omega_\a^\vee)^2)=\frac{1}{8}\dim\p=\half tr(ad_\p(\omega_{\a'}^\vee)^2)=j_{S'}(\omega_{\a'},\omega_{\a'})_{S'},
$$
as wished.\par
It remains to deal with the hermitian symmetric case. If $R$ is a subset of $A_{\a\a'}:=\{\a,\a'\}$, let  $\Dp(\p,R)$ be the subset of $\Dp(\p)$ consisting of the weights $\l$ such that\break  $supp(\l)\cap A_{\a\a'}=R$. Note that $\Phi_\ll^+=\Dp(\p,A_{\a\a'})$ and $\Phi_\ll^-=-\Dp(\p,\emptyset)$. Let $\eta\in\Dp(\p,\emptyset)\cup\Dp(\p,\a')$ and $\gamma\in\Pi\backslash \{\a\}$. Since, by Lemma \ref{fundcoweights}, $a_\a=1$, we see that, if $\eta+\gamma$ is a root, then $\eta+\gamma\in\Dp(\p,\emptyset)\cup\Dp(\p,\a')$. Also, if $\eta-\gamma$ is a root,  then $\eta-\gamma \in\Dp(\p,\emptyset)\cup\Dp(\p,\a')$. This implies as above that
$\langle\Dp(\p,\emptyset)\cup\Dp(\p,\a')\rangle=x\omega_\a+z\nu(\varpi)$. Arguing as for the proof of \eqref{unosolo} above we obtain that
$$-\langle\Dp(\p,\emptyset)\rangle-\langle\Dp(\p,\a')\rangle=j_S\omega_\a-\frac{|\Dp(\p,\emptyset)|+|\Dp(\p,\a')|}{\dim(\p)}\nu(\varpi).
$$
Likewise
$$-\langle\Dp(\p,\emptyset)\rangle-\langle\Dp(\p,\a)\rangle=j_{S'}\omega_{\a'}-\frac{|\Dp(\p,\emptyset)|+|\Dp(\p,\a)|}{\dim(\p)}\nu(\varpi),
$$
Thus, noting that $\langle\Dp(\p)\rangle=\tfrac{1}{2}\nu(\varpi)$, we obtain
\begin{align*}\langle \Phi_{\ll(h)}\rangle&=-2\langle\Dp(\p,\emptyset)\rangle-\langle\Dp(\p,\a')\rangle-\langle\Dp(\p,\a)\rangle+\langle\Dp(\p)\rangle\\
&=j_S\omega_\a + j_{S'}\omega_{\a'}+\frac{|\Phi_\ll^+|-|\Phi_\ll^-|}{\dim(\p)}\nu(\varpi)
\end{align*}
as wished.
\end{proof}

Assume now that $P=\{p\}$ and that $\theta_\p$ is a long non compact weight.  If $\r_S$ is a simple ideal of $\r$,  let $ \a_S$ be the unique element  of  $\Pi_S$ not orthogonal  to $p$. Then it is shown in \cite{IMRN} that $a_{\a_S}^S=1$.
 Denote by  $\omega_S$ the fundamental weight of $\r_S$  corresponding to $\a_S$.
 Set
\begin{equation}\label{specialm}
 \mathfrak m=\{\mu\in\D(\p): \mu=\theta_\p-\be,\,\be\in\Dp(\k)\cup\{0\}\}.
 \end{equation}
 \begin{lemma}\label{special}
 $$\langle \mathfrak m \rangle + \theta_\p=\sum_S j_S \omega_S.$$
 \end{lemma}
 \begin{proof}    Let $w_\s\in\Wa$ be as defined in  \cite[Proposition 3.8] {CKMP}. According to \cite[(3.14)]{CKMP},
 $$N(w_\s)=\{\tfrac{1}{2}\d'-\a\mid     \a\in\mathfrak{m}\}\cup\{\tfrac{3}{2}\d'-\theta_\p\}.
 $$It can be deduced from  \cite[Lemma 5.7]{IMRN} that $\theta_\p=\sum_Sn_S\omega_S$. Set $\Psi_S=\{\a\in\Pi_S\mid \a(\omega_S)>0\}$. We have
\begin{equation}\label{previous}\langle \mathfrak m \rangle + \theta_\p=(|\mathfrak m|+1)\theta_\p-\sum_S \langle \Psi_S\rangle=\sum_S(n_S\ell(w_\s)-h^\vee_S)\omega_S.\end{equation}
 It has been proved in
 \cite[Lemma 7.5 (4)]{CMPP} that  $\ell(w_\s)$ equals the dual Coxeter number of $\aa$, hence \eqref{previous} reads
 $\langle \mathfrak m \rangle +\theta_\p=\sum_S j_S \omega_S$, as desired.
 \end{proof}

\section{Decomposition formulas for conformal embeddings in the classical cases}\label{6}
 Let $\g$ be a simple complex Lie algebra of classical type and let $\k$ be a proper  reductive subalgebra conformally embedded in $\g$. We also assume that the embedding is maximal among conformal ones.

In order to study the intermediate vertex   algebras between $V_{\mathbf{k}}(\k)$ and $V_{1}(\g)$,  we need an explicit description of the decomposition of $V_{1}(\g)$ as a representation of $V_{\mathbf{k}}(\k)$.
We shall introduce a symmetric pair $(\aa,\r),\,\aa=\r\oplus\p,$ associated to $(\g,\k)$ and we shall decompose
 $V_{1}(\g)$ into irreducible $V_{\mathbf{k}}(\k)$-modules indexed by suitable Borel stable abelian subspaces of $\p$.

Let $\ft$ be a
 Cartan subalgebra of $\k$. Fix a Borel subalgebra $\b_\k$ containing $\ft$. Recall from Section \ref{simplecurrentext} that $L^\k(\l)$ denotes the irreducible highest weight module for $V_{\mathbf{k}}(\k)$ of highest weight $\l$.

\subsection{The case of $so(n,\C)$}\label{caseson}
We first discuss the case when $\g=so(n,\C)$.
As observed in Remark \ref{confinson} above, there is a symmetric pair
 $(\aa,\r)$ such that  $\g=so(\p), \k=\r, \ft=\h_0,\b_\k=\b_0$.  Recall that we have denoted by $\Sigma$ the set of abelian $\b_0$-stable subspaces of $\p$.
If  $\ll\in\Sigma$, $k\in\ganz$, we set (see notation from Section 3)
\begin{equation} \label{hwvhs}v_{k,\ll}= \begin{cases}
:\prod\limits_{\eta\in\Phi^+_{\ll}}T^{{\bf k!}}\bar x_{\eta}\prod\limits_{\eta\in-\Phi^-_{\ll}}T^{{\bf (k-2)!}}\bar x_{\eta}
\prod\limits_{\eta\in\Dp(\p)\setminus (\Phi^+_\ll\cup-\Phi^-_\ll)}T^{{\bf (k-1)!}}\bar x_{\eta}:&\text{ if $k>0$},\\
:\prod\limits_{\eta\in\Phi^-_{\ll}}T^{{\bf (-k)!}}\bar x_{\eta}\prod\limits_{\eta\in-\Phi^+_{\ll}}T^{{\bf (-k-2)!}}\bar x_{\eta}
\prod\limits_{\eta\in-\Dp(\p)\setminus (-\Phi^+_\ll\cup\Phi^-_\ll)}T^{{\bf (-k-1)!}}\bar x_{\eta}:&\text{ if $k<0$},\\
:\prod\limits_{\eta\in \Phi_\ll}\bar x_\eta: &\text{ if $k=0$}.
\end{cases}
\end{equation}

Define $\mathbf{j}$ as in \eqref{mathbfj}.
The following result provides the decomposition of $F(\bar \p)$ as a $V_{\mathbf{j}}(\k)$-module when $\s$ is indecomposable and $\k$ is semisimple.
In \cite{CKMP} we found explicit $\widehat {\mathfrak r}$-decompositions of the Clifford module $\text{Cliff}(L(\bar \p))/(\text{Cliff}(L(\bar \p))L(\bar \p)^+)$. Remark \ref{Cliffordmodule} allows us to use these results. We now state them in the present setting.

\begin{theorem}\label{decoeabeliani}\cite[Theorem 3.9]{CKMP} Assume $\s$ indecomposable and  $\k$  semisimple. Then
\begin{enumerate}
\item If $\theta_\p$ is not long noncompact then
\begin{equation}\label{corta}
F(\bar \p)=\bigoplus_{\ll\in \Sigma}L^\k\left(\langle \Phi_\ll\rangle\right).\end{equation}
 \item If $\theta_\p$ is
long  noncompact then
\begin{equation}\label{anomalo}
F(\bar \p)=\bigoplus_{\ll\in \Sigma}L^\k\left(\langle
\Phi_\ll\rangle\right)\bigoplus L^\k\left(\langle \mathfrak m\rangle+\theta_\p\right)\end{equation} where $\mathfrak m$ is as in \eqref{specialm}.

\end{enumerate}
Moreover, in
both cases  the highest weight vector of $L^\k\left(\langle
\Phi_\ll\rangle\right)$ is, up to a
constant factor,
$v_{0,\ll}$ (cf. \eqref{hwvhs}).
 A highest weight vector of  the rightmost   component
of \eqref{anomalo}  is
\begin{equation}\label{maximal}v_\mathfrak m=:T(\bar x_{\theta_\p})\bar x_{\eta_1}\dots \bar x_{\eta_s}:
\end{equation}
if $\mathfrak m=\{\eta_1,\ldots,\eta_s\}$.
\end{theorem}

We now discuss the hermitian symmetric case. For $C\in\ganz$, denote by $F(\bar\p)[C]$ the eigenspace of eigenvalue $C$ of $\varpi$ (recall that $\varpi$ is the generator of the
center of $\r$ such that $\a_\p(\varpi)=1$).  Also set
\begin{equation}\label{sigmaprimo}
\Sigma'=\{\ll\in\Sigma\mid \aa_{-\a_\p}\subset \ll\}.
\end{equation}

\begin{theorem}\label{finalehs} \cite[Theorem 5.4]{CKMP} In the hermitian symmetric case,
\begin{equation}\label{hsf}F(\bar\p)[C]=\sum_{\substack{\ll\in\Sigma'\\
|\Phi _\ll^+|-|\Phi _\ll^-|\equiv
C\,mod\,\frac{\dim(\p)}{2}}}
L^\k(\langle
\Phi_\ll\rangle+ {k_\ll(C)}\nu(\varpi)
),\end{equation}
where $k_\ll(C)=\frac{(C-|\Phi _\ll^+|+|\Phi _\ll^-|)}{\dim(\p)}$.
\end{theorem}

\begin{rem} If $\s$ is not indecomposable then, clearly, $F(\bar\p)=\otimes_s F(\ov{\p\cap\aa_s})$, so combining Theorems \ref{decoeabeliani} and \ref{finalehs} one has the decomposition of $F(\bar\p)$ as a $V_{\mathbf{j}}(\k)$-module.\end{rem}
\vskip5pt
In order to write down explicit formulas for  highest weight vectors  for the irreducibles appearing in the r.h.s. of \eqref{hsf}, we need the following refined analysis.
Identify $\h_0$ and $\h_0^*$ via $\nu$ and set $(\h_0)_\R$ to be the real span $\D_{|\h_0}$. Set $\ha_\R=(\h_0)_\R\oplus\R K\oplus \R d$. Define
\begin{equation}\label{caf}
\mathcal A=\{h\in(\h_0)_\R\mid \bar\a(h)> -s_\a,\,\a\in \Pia\}
\end{equation}
(the $s_\a$ were defined in \eqref{salfa1},\eqref{salfa2}). Clearly $\Wa$ acts naturally on $d+(\h_0)_\R$ viewed  as a subspace $\ha_\R/\R K$. The same argument used  in \cite[Chapter 6]{Kac} shows that  the closure of $d+\mathcal A$ is a fundamental domain for this action.

Let $L^\k(\langle
\Phi_\ll\rangle+ {k}\nu(\varpi)
)$ occur  in $F(\bar\p)[C]$ with $k=k_\ll(C)$. Let $w_k$ be the unique element of $\Wa$ such that
\begin{equation}\label{wk}
w_k(d+\mathcal A)=t_{2k\varpi}w_\ll(d+\mathcal A)
\end{equation}
(here $t_h:
 \h_0\to\h_0$ is the translation by $h$).
Write explicitly
$N(w_k)=\{\be_1,\dots, \be_s\}$. If $x\in \mathcal A$, it is clear that
\begin{equation}\label{Nwk}
N(w_k)=\{\a\in\Dap\mid \a(t_{2k\varpi}w_\ll(d+x))<0\}.
\end{equation}
If $\a=h\d'+\bar \a\in\Dp(\k)\cup\{s\d'+\eta\mid s>0,\eta\in\D(\k)\}$, then
$$
\a(w_\ll(d+x)+2k\varpi)=h+\a(w_\ll(d+x))\ge 0,
$$
thus, for $i=1,\dots,s$,
$$
\be_i=(n_i+\half)\d'+\bar \be_i.
$$
By Lemma 5.1 of \cite{CKMP} and Theorem 1.1 of \cite{MMJ}, we have that
$$
F(\bar\p)=\sum_{k\in \ganz} L^\k(w_k^{-1}(\rho)-\rho)
$$
with $w_k$ as in \eqref{wk}. In the proof of Theorem 1.1 of \cite{MMJ} an explicit expression for the highest weight vector for $L^\k(w_k^{-1}(\rho)-\rho)$ is given. Since $w_k^{-1}(\rho)-\rho\equiv\langle\Phi_
\ll\rangle+ {k}\nu(\varpi)\mod \C\d'$ (see again Section 5 of \cite{CKMP}), the highest weight vector of $L^\k(\langle
\Phi_\ll\rangle+ {k}\nu(\varpi)
)$ is, under the isomorphism described in Remark \ref{Cliffordmodule},
\begin{equation}\label{hwwk}
v=:T^{n_1}(\bar x_{-\bar\be_i})\cdots T^{n_s}(\bar x_{-\bar\be_s}):.
\end{equation}

We need to express $v$ more explicitly in terms of the roots in $\Phi_\ll$, thus we compute $N(w_k)$. By \eqref{Nwk}, we have to solve the inequality
$\a(w_\ll(d+x)+2k\varpi)<0$, with $\a=h\d'+(\half\d'+\bar \a)$, $\bar\a\in\D(\p)$.\par
 If $\bar\a=0$, then
$\a(w_\ll(d+x)+2k\varpi)=h+\half>0$.
If $\bar\a\ne 0$, we first observe that
\begin{equation}\label{minusone}
(\half\d'+\bar\a)(w_\ll(d+x))>-1
\end{equation}
Indeed, if $(\half\d'+\bar\a)(w_\ll(d+x))\leq -1$, then $(\d'+\half\d'+\bar\a)(w_\ll(d+x))<0$, contrary to the fact that $w_\ll$ is $\Wab$.

If $\bar\a\in\Phi_\ll$, then
\begin{equation}\label{plusinl}
1<(\half\d'+\bar\a)(w_\ll(d+x))<2
\end{equation}
Indeed, by \eqref{minusone}, we have $-1<(\half\d'-\bar\a)(w_\ll(d+x))<0$, so $-2<(-\half\d'-\bar\a)(w_\ll(d+x))<-1$.

If $\bar\a\in-\Phi_\ll$, then, by \eqref{minusone}
\begin{equation}\label{minusinl}
-1<(\half\d'+\bar\a)(w_\ll(d+x))<0
\end{equation}

Finally, if $\bar\a\in\D(\p)\backslash (\Phi_\ll\cup(-\Phi_\ll))$, then
\begin{equation}\label{plusone}
0<(\half\d'+\bar\a)(w_\ll(d+x))<1
\end{equation}
Indeed, since $\bar\a\not\in-\Phi_\ll$, we have $(\half\d'+\bar\a)(w_\ll(d+x))>0$. Since $\bar\a\not\in\Phi_\ll$,  we have $(\half\d'-\bar\a)(w_\ll(d+x))>0$, so  $(-\half\d'-\bar\a)(w_\ll(d+x))>-1$.
Observe also that, if $\eta\in\D(\p)$ and $\eta\ne0$, then $\eta(\varpi)=1$ if and only if $\eta\in\Dp(\p)$.

Assume $k>0$. If $\bar\a\in\Dp(\p)$, then
$\a(w_\ll(d+x)+2k\varpi)=h+(\half\d'+\bar \a)(w_\ll(d+x))+2k$, hence, by \eqref{plusinl}, \eqref{minusinl}, and \eqref{plusone}, we deduce that $\a(w_k(d+x))>0$.
 If $\bar\a\in-\Dp(\p)$, then
$\a(w_\ll(d+x)+2k\varpi)=h+(\half\d'+\bar \a)(w_\ll(d+x))-2k$, hence, by \eqref{plusinl}, \eqref{minusinl}, and \eqref{plusone}, we get
$$
((h+\half)\d'+\bar\a)(w_k(d+x))<0 \Leftrightarrow\begin{cases}h\le 2k-2&\text{if $\bar\a\in\Phi_\ll$}\\
h\le 2k-1&\text{if $\bar\a\not\in(\Phi_\ll\cup(-\Phi_\ll))$}\\
h\le 2k&\text{if $\bar\a\in-\Phi_\ll$}
\end{cases}.
$$

Summarizing,  we obtain that, if $k>0$,
\begin{align*}
N(w_k)&=\{(h+\half)\d'+\eta\mid \eta\in -\Phi_\ll^+, h\le 2k\}\cup\{(h+\half)\d'+\eta\mid \eta\in \Phi_\ll^-, h\le 2k-2\}\\&\cup\{(h+\half)\d'+\eta\mid \eta\in -\Dp(\p)\backslash(-\Phi_\ll^+\cup\Phi_\ll^-), h\le 2k-1\}.
\end{align*}

If instead $k<0$ and $\bar\a\in-\Dp(\p)$, then
$\a(w_\ll(d+x)+2k\varpi)=h+(\half\d'+\bar \a)(w_\ll(d+x))-2k$, hence, by \eqref{plusinl}, \eqref{minusinl}, and \eqref{plusone}, $\a(w_k(d+x))>0$.
 If $\bar\a\in\Dp(\p)$, then
$\a(w_\ll(d+x)+2k\varpi)=h+(\half\d'+\bar \a)(w_\ll(d+x))+2k$, hence, by \eqref{plusinl}, \eqref{minusinl}, and \eqref{plusone},
$$
((h+\half)\d'+\bar\a)(w_k(d+x))<0 \Leftrightarrow\begin{cases}h\le -2k-2&\text{if $\bar\a\in\Phi_\ll$}\\
h\le -2k-1&\text{if $\bar\a\not\in(\Phi_\ll\cup(-\Phi_\ll))$}\\
h\le -2k&\text{if $\bar\a\in-\Phi_\ll$}
\end{cases}.
$$

Hence we obtain that, if $k<0$
\begin{align*}
&N(w_k)=\\
&\{(h+\half)\d'+\eta\mid \eta\in -\Phi_\ll^-, h\le 2|k|\}\cup\{(h+\half)\d'+\eta\mid \eta\in \Phi_\ll^+, h\le 2|k|-2\}\\&\cup\{(h+\half)\d'+\eta\mid \eta\in \Dp(\p)\backslash(-\Phi_\ll^-\cup\Phi_\ll^+), h\le 2|k|-1\}.
\end{align*}


Combining our description of $N(w_k)$ with \eqref{hwwk} we have proven the following proposition:
\begin{prop}\label{nw} If $(\aa,\k)$ is an hermitian symmetric  pair and $\ll\in\Sigma'$, then a highest weight vector of $L^\k(\langle \Phi_\ll\rangle +k\nu(\varpi))$ in $F(\bar\p)$ is $v_{2k,\ll}$
(cf. \eqref{hwvhs}).
\end{prop}
\vskip10pt
\subsection{Two special cases}\label{specialcases} In order to obtain decomposition formulas for conformal embeddings in a simple Lie algebra of classical type other than $so(n,\C)$, we need an application of the results of Section \ref{caseson} to two special cases.
  Let $\g_n$ be either $sp(n,\C)$ or $sl(n+1,\C)$. Set $\u=sl(2,\C)$ in the former case, $\u=\C$ in the latter. It  turns out  that $(\g_{n+1},\u\times\g_n)$ is a symmetric pair, and let  $\g_{n+1}=(\u\times\g_n)\oplus \q$ be the corresponding eigenspace decomposition. In these special cases $\mathbf{j}=(j_\u,j_{\g_n})$ and it turns out that $j_{\g_n}=1$.

Choose  Cartan subalgebras $\h_\u,\,\h_n$ of $\u,\,\g_n,$ respectively.  If $M$ is an irreducible $V_{\mathbf{j}}(\u\times \g_n)$-module, then
$M=L^\u(\l)\otimes L^{\g_n}(\mu)$ with $\l\in \h_\u^*$, $\mu\in \h_n^*$, and $L^\u(\l)$ (resp. $L^{\g_n}(\mu)$ ) an irreducible highest weight module  for $V_{j_\u}(\u)$ (resp. $V_{1}(\g_n)$). Given a  $V_{\mathbf{j}}(\u\times \g_n)$-module $M$,  denote by $\L^\k(M)$ the set of irreducible modules appearing in the decomposition of $M$.

 \begin{lemma} \label{isirrep}
 If $L^{\u}(0)\otimes L^{\g_n}(\mu)\in \L^\k(F(\bar \q))$,  then $\mu=0$ and $L^\u(0)\otimes L^{\g_n}(0)= V_{\mathbf{j}}(\u\times \g_n)$ occurs with multiplicity one.
\end{lemma}
 \begin{proof} Suppose $\g_n$ of type $C_n$, so that  $\g_{n+1}$ is of type $C_{n+1}$. The Lie algebra $\widehat L^\k(\g_{n+1},\s_n)$ corresponding to the symmetric pair $(\g_{n+1},\g_n\times \u)$ is of type $C_{n+1}^{(1)}$. If we use the numbering of Dynkin diagrams given in \cite{Kac}, then the $A_1$-factor corresponds to label $0$ and the $C_n$ factor to labels $2,\ldots,n+1$. Since $\theta_\q$ is short, formula \eqref{corta} gives
 $$
F(\bar \q)=\sum_{\ll\in\Sigma}L^\k(\langle \Phi_\ll \rangle).
 $$

The set of elements of $\Wa$ encoding $\Sigma$  is given by
 $\{w_j\mid 0\leq j \leq n-1\}$ where
 \begin{equation}\label{w}
 w_j=s_1s_2\cdots s_{j}.
 \end{equation} It follows that, if $w_\ll=w_j$, then $\langle \Phi_\ll \rangle =-j\bar \a_1-(j-1)\a_2-\ldots-\a_j$ and
 $\bar\a_1=-\frac{1}{2}\a_0-\a_2-\ldots-\a_n-\frac{1}{2}\a_{n+1}$, hence the coefficient of $\a_0$ is zero only if $j=0$, i.e. $w_0=Id$, so $\ll=\{0\}$.

 Assume now $\g_n$ of type $A_{n}$. The Lie algebra $\widehat L(\g_{n+1},\s_n)$ corresponding to the symmetric pair $(\g_{n+1},\g_n\times \u)$ is of type $A_{n+1}^{(1)}$.  We can assume that $\Pi=\{\a_2,\dots,\a_{n+1}\}$. In this setting the set of elements of $\Wa$ encoding the abelian subspaces   in $\Sigma'$  is $\{w_j\mid 1\leq j \leq n+1\}$ with $w_j$ as in \eqref{w}. If $\ll\in\Sigma'$, then $\Phi_\ll^\pm=\{\be\in\Phi_\ll\mid \be(\varpi)=\pm 1\}$. It follows that $\Phi _\ll^+=\emptyset$. By Theorem \ref{finalehs},
  $L^\k(\langle \Phi_\ll \rangle+ k_\ll(0)\nu(\varpi))$ occurs in $F(\bar\q)[0]$  if and only if $|\Phi_\ll|\equiv 0\mod n+1$. Only $w_{n+1}$ has this property, proving our claim.
\end{proof}

\subsection{A general decomposition theorem}\label{generaldec} Let now $\k\hookrightarrow  \g$ be a maximal conformal embedding with $\g$ simple of classical type. If $\g=so(n,\C)$, let $(\aa,\r)$ be the symmetric pair associated to the conformal embedding $\k\hookrightarrow  \g$ by Proposition \ref{symmthm}. If $\g=sp(n,\C)$ or $sl(n+1,\C)$, let $\u$ and $\q$ be as in Section \ref{specialcases}. Remark that $\u\times \k$ is conformal in $\u\times\g$. Thus $\u\times \k$ is conformal in $so(\q)$, hence, by the Symmetric Space Theorem,  there is a symmetric pair $(\aa,\r)$ with
 $\r=\u\times \k$.

 Recall the eigenspace decomposition $\aa=\r\oplus\p$. Let $\h_0$, $\b_0$, $\Sigma$ be as defined in Section \ref{symmpairs}. Recall that $\ft$ is a Cartan subalgebra of $\k$. Then $\h_0= \ft'\times\ft$ with $\ft'=\{0\}$ when $\g=so(n,\C)$ and $\ft'=\h_\u$ in the other cases.     Let $\mathfrak z(\k)$ be the center of $\k$.  Browsing through the maximal conformal embedding described in \cite{AGO}, one checks that  $\dim\mathfrak{z}(\k)= 1$, hence we can choose $\varkappa\in\ft$ so that $\mathfrak{z}(\k)=\C\varkappa$. Let $\kappa
\in \ft^*$ be defined by setting $\kappa(\varkappa)=1$ and $\kappa(\ft\cap[\k,\k])=0$. Write $\aa=\sum_{i\in I}\aa_{i}$  for the decomposition of $\aa$ into indecomposable ideals. Divide the index set $I$ in the two subsets $I_0$ and $I'$ according to whether  $\mathfrak z(\r)\cap \aa_{i}= \{0\}$ or not.

 Set $\aa_0=\sum_{i\in I_0}\aa_i$, $\aa'=\sum_{i\in I'}\aa_i$, and
let, for $\ll\in\Sigma$,
 \begin{equation}\label{llprime}\ll_0=\ll\cap\aa_0\quad \ll'=\ll\cap\aa'.
 \end{equation}
For $i\in I'$, choose $\varkappa_i\in \mathfrak z(\r)\cap \aa_i$ normalized by setting $\a_{\p\cap\aa_i}(\varkappa_i)=1$.
Set
$$\Sigma'=\{\ll\in\Sigma\mid -\a_{\p\cap\aa_i}\in\Phi_\ll\text{ for all $i\in I'$}\}.$$ Clearly this definition of $\Sigma'$ generalizes the one given in \eqref{sigmaprimo}.
Set
$$
\Sigma_{0}=\{\ll\in\Sigma\mid \langle\Phi_\ll\rangle_{|\ft'}=0\},\quad\Sigma'_{0}=\Sigma'\cap\Sigma_{0},\quad \Sigma^{even}_{0}=\Sigma_{0}\cap\Sigma^{even}.
$$

Set $\chi$ to  be the truth function that is $1$ if and only if $\g=so(n,\C)$, $\k$ semisimple, $\s$ indecomposable, and $\theta_\p$ long noncompact. Recall from Theorem \ref{decoeabeliani} that in such cases we defined a set of roots $\mathfrak{m}$. Set   $\e_\mathfrak{m}=1$ if $\mathfrak{m}$  has odd cardinality and $\e_\mathfrak{m}=0$ otherwise.
  \begin{theorem} \label{decoclassical}Let $\k\hookrightarrow  \g$ be a maximal conformal embedding with $\g$ simple of classical type. Then its decomposition into  irreducible $V_{\bf j}(\k)$-modules
  is given by
\begin{equation}\label{gn}V_{1}(\g)=\bigoplus_{\ll\in\Sigma^{even}_0} L^\k({\langle \Phi_\ll\rangle}_{|\ft})\oplus \chi \e_\mathfrak{m}L^\k(\langle \mathfrak m\rangle+\theta_\p)\end{equation}
if $\k$ is semisimple.

 If $\k$ is not semisimple and  $\mathfrak z({\r})=\mathfrak z({\k})$, so that $I'=\{i\}$ and we can set $\varkappa=\varkappa_i$, then
\begin{equation}\label{gnhs}V_{1}(\g)=\bigoplus_{\substack{\ll\in\Sigma'_0,k\in\frac{1}{2}\ganz\\
\dim\ll+k\dim(\p\cap\aa_i)\in 2\ganz}} L^\k({\langle \Phi_\ll\rangle_{|\ft}}+k\dim(\p\cap\aa_i)\kappa)
\end{equation}

In particular $v_{0,\ll}$, $v_\mathfrak m$ give
  highest weight vectors  in $F(\bar\p)$  for the irreducibles appearing in \eqref{gn}. The highest weight vector for the irreducible occurring in \eqref{gnhs} is $:v_{2k,\ll'}v_{0,\ll_0}:$ (cf. \eqref{hwvhs}, \eqref{maximal}).
 \end{theorem}
 \begin{rem}\label{zetarzetak}
 Actually the hypothesis of Theorem \ref{decoclassical} that $\mathfrak z(\k)=\mathfrak z(\r)$ in the non semisimple case rules out only the conformal embedding $sl(p)\times sl(q)\times \C\hookrightarrow sl(p+q)$. This latter case requires a special discussion which will be done in Proposition \ref{slpq}.
 \end{rem}
\begin{proof}Assume first $\k$ semisimple.
Since $V_{1}(so(\p))=F(\bar\p)^0$, formula \eqref{gn} follows immediately from Theorem \ref{decoeabeliani} in the case when $\g=so(\p)$ and $\s$ indecomposable, because, in this case, $\Sigma_0^{even}=\Sigma^{even}$.
The only case with $\g$ of type $so(n,\C)$ and $\s$ decomposable occurs when $\k=so(p,\C)\times so(q,\C)$ and $\aa=so(p+1,\C)\times so(q+1,\C)$. Set $\Sigma^N$ to denote the set of abelian subspaces corresponding to the pair $(so(N+1,\C),so(N,\C))$. Then Theorem \ref{decoeabeliani} gives in this case that
$$
F(\bar\p)=\bigoplus_{\ll\in\Sigma^p,\i\in\Sigma^q} L^{so(p,\C)}(\langle \Phi_\ll\rangle)\otimes L^{so(q,\C)}(\langle \Phi_\i\rangle)=\bigoplus_{\ll\in\Sigma} L^\k(\langle\Phi_\ll\rangle)
$$
hence \eqref{gn} follows also in this case from the observation that $V_{1}(so(\p))=F(\bar\p)^0$.

It remains to check formula \eqref{gn} in the cases $\g=sp(n,\C)$ and $\g=sl(n+1,\C)$. Recall that we are assuming $\k$ semisimple. Looking at the possible cases listed in \cite{AGO} one checks that $\s$ is indecomposable. By Lemma \ref{isirrep}, $L^\u(0)\otimes L^{\g}(0)=L^\u(0)\otimes V_{1}(\g)$ is the unique factor in $V_{1}(so(\p))$ of type $L^\u(0)\otimes L^{\g}(\mu)$. It follows that $L^\k(\l)$ occurs in $V_{1}(\g)$ if and only if $L^\u(0)\otimes L^\k(\l)$ occurs in $F(\bar\p)$.
By  \cite[(5.13)]{IMRN}, we have Ê that $({\langle \mathfrak m\rangle}+\theta_\p)_{|\ft'}\ne 0$ (see also \eqref{previous}).
By Theorem  \ref{decoeabeliani}, we get
$$
L^\u(0)\otimes V_{1}(\g)=\bigoplus_{\ll\in\Sigma_0}L^\u(0)\otimes L^\k(\langle \Phi_\ll\rangle_{|\ft}),
$$
In particular, since  $L^\u(0)\otimes V_{1}(\g)\subset F(\bar\p)^0$, we get $\Sigma_0=\Sigma_0^{even}$ and \eqref{gn} follows in this case.
The same argument applied for $\g=sl(n+1,\C)$ gives, by Theorem \ref{finalehs},
\begin{equation}\label{parz}
V_{1}(\g)=\sum_{\substack{\ll\in\Sigma'\\
|\Phi_\ll^+|\equiv|\Phi_\ll^-|\,mod\,\frac{\dim(\p)}{2}}}L^\k({\langle
\Phi_\ll\rangle_{|\ft}}
).\end{equation}
Following the notation preceding Theorem \ref{finalehs}, define $\varpi$ to be the element in $\h_0$ such that $\a(\varpi)=\d_{\a,\a_\p}$ for all $\a\in\Pi_\aa$.
Observe that   ${\langle\Phi_\ll\rangle}_{|\ft'}=-k_\ll(0)\nu(\varpi)$, hence $\ll\in \Sigma'_0$ if and only if
  $|\Phi_\ll^+|=|\Phi_\ll^-|$. Split the sum appearing in the r.h.s. of \eqref{parz} into two pieces as
 \begin{equation}\label{parz2}
 V_{1}(\g)=\sum_{\ll\in\Sigma'_0}L^\k({\langle
\Phi_\ll\rangle_{|\ft}})
+\sum_{\substack{\ll\in\Sigma', |\Phi_\ll^+|\ne|\Phi_\ll^-|\\
|\Phi_\ll^+|\equiv|\Phi_\ll^-|\,mod\,\frac{\dim(\p)}{2}}}L^\k({\langle
\Phi_\ll\rangle_{|\ft}})
\end{equation}
A result of Panyushev \cite{Pan} guarantees that, in the hermitian symmetric case, $\frac{\dim(\p)}{2}$ is the maximal dimension of an abelian subspace.
Hence the conditions
 $|\Phi_\ll^+|\ne|\Phi_\ll^-|, |\Phi_\ll^+|\equiv|\Phi_\ll^-|\,mod\,\frac{\dim(\p)}{2}
$ imply that either $\Phi_\ll^+$ or $\Phi_\ll^-$ is empty. By the definition of $\Sigma'$, we have $\Phi_\ll^-\ne\emptyset$, hence $\Phi_\ll=\Phi_\ll^-$ and
$|\Phi_\ll|=\frac{\dim(\p)}{2}$. In turn, $\ll$ is the nilradical of the parabolic defined by $-\varpi$ and $\langle\Phi_\ll\rangle=-\frac{1}{2}\nu(\varpi)$, so $\langle\Phi_\ll\rangle_{|\ft}=0$.
We conclude that  \eqref{parz2} can be written as
$$
V_{1}(\g)=\sum_{\ll\in\Sigma'_0\cup\{0\}}L^\k({\langle
\Phi_\ll\rangle_{|\ft}}).
$$
To  finish the proof of  \eqref{gn} observe that $\sum_{\ll\in\Sigma_0}L^\k({\langle
\Phi_\ll\rangle_{|\ft}})$ must occur in $V_{1}(\g)$, so $\Sigma'_0\cup\{0\}=\Sigma_0$. Since $L^\u(0)\otimes V_{1}(\g)\subset F(\bar\p)^0$, we see that also in this case $\Sigma_0=\Sigma_0^{even}$ and \eqref{gn} follows in this case.

 Assume now that $\k$ is reductive but not semisimple and $\mathfrak z(\r)=\mathfrak z(\k)$. As observed in Remark \ref{zetarzetak}, this hypothesis covers all cases when $\s$ is indecomposable and the case $\r=\k=so(p,\C)\times so(2,\C)\hookrightarrow so(p+2,\C)$.

 We now discuss the cases when $\s$ is indecomposable. Since $\mathfrak z(\r)=\mathfrak z(\k)$, we have that $\u\subset [\r,\r]$ for otherwise $\u\subset \mathfrak z(\r)$. Also recall that we normalize $\varkappa$ assuming that $\a_\p(\varkappa)=1$.
Set $\g'=\g$ when $\g=so(n,\C)$ and $\g'=\u\times \g$ in the other cases. According to Theorem \ref{finalehs},
 \begin{equation}V_{1}( \g')=\sum_{\substack{k\in\frac{1}{2}\ganz,\ll\in\Sigma'\\
k\dim(\p)+|\Phi^+_\ll|-|\Phi^-_\ll|\in2\ganz}}
L^\r(\langle
\Phi_\ll\rangle+k\nu(\varkappa)
),\end{equation}
so,
applying Lemma \ref{isirrep} if $\g\ne so(n,\C)$, we obtain in all cases that
 \begin{equation}V_{1}( \g)=\sum_{\substack{k\in\frac{1}{2}\ganz,\ll\in\Sigma'\\
k\dim(\p)+|\Phi^+_\ll|-|\Phi^-_\ll|\in2\ganz\\
(\langle
\Phi_\ll\rangle+k\nu(\varkappa))(\u)=0}}
L^\k((\langle
\Phi_\ll\rangle+k\nu(\varkappa))_{|\ft}).
\end{equation}
Since  $\u\subset [\r,\r]$ and $\varkappa\in\mathfrak z(\r)$, we have $\nu(\varkappa)(\u)=0$.   Thus $(\langle
\Phi_\ll\rangle+k\nu(\varkappa))(\u)=0$ if and only if $\langle
\Phi_\ll\rangle(\u)=0$. This implies that $\ll\in\Sigma'_0$. Observe that $\nu(\varkappa)(\varkappa)=tr(ad_\p(\varkappa)^2)=\dim \p$, so $\nu(\varkappa)_{|\ft}=(\dim \p )\kappa$.
It follows that
 \begin{equation}V_{1}( \g)=\sum_{\substack{k\in\frac{1}{2}\ganz,\ll\in\Sigma_0'\\
k{\dim(\p)}+|\Phi^+_\ll|-|\Phi^-_\ll|\in2\ganz}}
L^\k(\langle
\Phi_\ll\rangle_{|\ft}+k\,{\dim\p}\,\kappa).
\end{equation}
which is \eqref{gnhs} in these cases.

It remains to check the case when $\r=\k=so(p,\C)\times so(2,\C)$ with $p\ge 3$.  We observe that  $\aa'=so(3,\C)$ and $\aa_0=so(p+1,\C)$. Since $(\aa_0,\r\cap\aa_0)$ is irreducible with $\r\cap\aa_0$ semisimple, we can apply Theorem \ref{decoeabeliani}. Since $\theta_{\p\cap\aa_0}$ is short or complex, \eqref{corta} holds.  We apply Theorem \ref{finalehs} to $\aa'$. We get
\begin{equation}
F(\bar\p)=\sum_{\ll_0\in\Sigma,k\in\frac{1}{2}\ganz,\ll'\in\Sigma'\\
}
L^\k(\langle\Phi_{\ll_0}\rangle+\langle
\Phi_{\ll'}\rangle+ k\nu'(\varkappa)
),
\end{equation}
where $\nu'$ is the identification of $\h_0\cap\aa'$ with its dual via the Killing form of $\aa'$.
Note that in this case $\Sigma'=\Sigma'_0$ and that $\ll\in\Sigma'$ if and only if $\ll_0\in\Sigma$ and $\ll'\in\Sigma'$.
Since $V_{1}(\g)=F(\bar \p)^0$ we see that $L^\k(\langle\Phi_{\ll_0}\rangle+\langle
\Phi_{\ll'}\rangle+ k\nu'(\varkappa)
)$ occurs in $V_{1}(\g)$ if and only if its highest weight vector $:v_{k,\ll'}v_{0,\ll_0}:$ is in $F(\bar\p)^0$. It is clear from its explicit description given in \eqref{hwvhs} that this happens if and only if $\dim\ll_0+\dim\ll'+k\dim{\p\cap\aa'}=\dim\ll+ k\dim{\p\cap\aa'}\in2\ganz$. This observation proves \eqref{gnhs} in this case.
\end{proof}

We now discuss the missing case $sl(p)\times sl(q)\times \C\hookrightarrow sl(p+q)$.  Here  $\aa=\aa'=sl(p+1,\C)\times sl(q+1,\C)$. Write $\aa_1=sl(p+1,\C)$ and $\aa_2=sl(q+1,\C)$. Set $m=GCD(p+1,q+1)$ and $M=GCD(p(p+1),q(q+1))$. Realizing explicitly the embedding $\k\hookrightarrow  \g$, we see that we can choose $\varkappa=\frac{Mq}{(q+1)(p+q)}\varkappa_1 -\frac{Mp}{(p+1)(p+q)}  \varkappa_2$ and $\u=\C((p+1)\varkappa_1+(q+1)\varkappa_2)$.

We will later need the following elementary result.
\begin{lemma}\label{elementary}Consider the map $\varphi:\ganz\times\ganz\to \ganz/M\ganz$ defined by setting $$\varphi(i,j)=(p+1)i+(q+1)j+M\ganz.$$
Choose $(x,y)\in\ganz^2$ such that
\begin{equation}\label{xy}xp(p+1)+yq(q+1)=M.
\end{equation}
Then
\begin{enumerate}
\item If $(i,j)\in Ker\varphi$, then $-\frac{m}{M}(yqi-xpj)\in\ganz$.
\item The map $\varphi$ pushes down to define a map  on $\ganz/p\ganz\times\ganz/q\ganz$.
\item The map $\psi:Ker\varphi\to \ganz/\frac{mpq}{M}\ganz$ defined by setting $$\psi(i+p\ganz,j+q\ganz)=-\frac{m}{M}(yqi-xpj)+\frac{mpq}{M}\ganz$$ is a group isomorphism.
\end{enumerate}
\end{lemma}

Recall from the proof of Lemma \ref{isirrep} that given $i=1,\dots,p$ and $j=1,\dots,q$ there is a unique $\ll(i,j)\in\Sigma'$ such that $\dim\ll(i,j)\cap\aa_1=i$ and $\dim\ll(i,j)\cap\aa_2=j$. Moreover $\Sigma'=\{\ll(i,j)\mid 1\le i \le p,\,1\le j\le q\}$.
\begin{prop}\label{slpq} If $\g=sl(p+q)$ and $\k=sl(p)\times sl(q)\times \C\varkappa$, then
\begin{equation}\label{decoslpslq}
V_{\mathbf{1}}(\g)=\sum_{\substack{i=1,\dots,p,\,j=1,\dots,q\\(i,j)\in Ker\varphi,\,t\in\psi(i+p\ganz,j+q\ganz) }}
 L^\k(\langle\Phi_{\ll(i,j)}\rangle_{|\ft\cap[\k,\k]}+ t\kappa).
\end{equation}
\end{prop}
\begin{proof}
 Lemma \ref{isirrep} and Theorem \ref{finalehs} give that
 \begin{equation}\label{sudecomp}
V_{\mathbf{1}}(\u\times\g)=\sum_{\substack{\ll\in\Sigma',h,k\in\frac{1}{2}\ganz\\(\langle\Phi_{\ll}\rangle+ h\nu(\varkappa_1)+k\nu(\varkappa_2))_{|\u}=0} }
 L^{\u\times\k}(\langle\Phi_{\ll}\rangle+ h\nu(\varkappa_1)+k\nu(\varkappa_2)).
\end{equation}

 As in Theorem \ref{decoclassical}, the decomposition of $V_{1}(\g)$ is obtained by dropping the factor $L^\u(0)$ from \eqref{sudecomp}:
  \begin{equation}\label{slpslqdec}
V_{\mathbf{1}}(\g)=\sum_{\substack{\ll\in\Sigma',h,k\in\frac{1}{2}\ganz\\(\langle\Phi_{\ll}\rangle+ h\nu(\varkappa_1)+k\nu(\varkappa_2))_{|\u}=0} }
 L^\k((\langle\Phi_{\ll}\rangle+ h\nu(\varkappa_1)+k\nu(\varkappa_2))_{|\ft}).
\end{equation}
 We now make explicit the condition $(\langle\Phi_{\ll}\rangle+ h\nu(\varkappa_1)+k\nu(\varkappa_2))_{|\u}=0$. Recalling that $\u=\C((p+1)\varkappa_1 + (q+1)\varkappa_2)$,  the condition becomes, if $\ll=\ll(i,j)$,
 $$
 2h(p+1)p+2kq(q+1)=(p+1)i+(q+1)j,
 $$
 which has solution if and only if $(i,j)\in Ker\varphi$. Set $r=\frac{(p+1)p}{M}$ and $s=\frac{(q+1)q}{M}$ and choose $x,y$ such that \eqref{xy} holds. The solutions are given by
 $$
 \begin{pmatrix}
 2h\\2k
 \end{pmatrix}= \begin{pmatrix}
 -s\\r
 \end{pmatrix}z+ \frac{(p+1)i+(q+1)j}{M}\begin{pmatrix}
 x\\y
 \end{pmatrix}
 $$
as $z$ varies in $\ganz$. Substituting in \eqref{slpslqdec}, observing that $$\langle \Phi_{\ll(i,j)}\rangle_{|\ft}=\langle \Phi_{\ll(i,j)}\rangle_{|\ft\cap[\k,\k]}-(i\frac{Mq}{(q+1)(p+q)}-j\frac{Mp}{(p+1)(p+q)})\kappa$$ and that $h\nu(\varkappa_1)+k\nu(\varkappa_2))_{|\ft}=(h\nu(\varkappa_1)+k\nu(\varkappa_2))(\varkappa)\kappa$, we get, after some elementary computations, that
$$
V_{\mathbf{1}}(\g)=\sum_{\substack{i=1,\dots,p,\,j=1,\dots,q\\(i,j)\in Ker\varphi,z\in\ganz }}
 L^\k\left(\langle\Phi_{\ll(i,j)}\rangle_{|\ft\cap[\k,\k]}+ \left(-\frac{mpq}{M}z-\frac{m}{M}(yqi-xpj)\right)\kappa\right)
$$
which is \eqref{decoslpslq}.
   \end{proof}
\section{Intermediate vertex subalgebras of $(V_{1}(\g), V_{\mathbf{j}}(\k))$}\label{7}
Recall that in \eqref{hwvhs} we introduced the elements $v_{k,\ll}\in F(\bar\p)$.
Define
\begin{equation} \label{hwvhsneg}v^*_{k,\ll}= \begin{cases}
:\prod\limits_{\eta\in\Phi^+_{\ll}}T^{{\bf (k)!}}\bar x_{-\eta}\prod\limits_{\eta\in-\Phi^-_{\ll}}T^{{\bf (k-2)!}}\bar x_{-\eta}
\prod\limits_{\eta\in\Dp(\p)\setminus (\Phi^+_\ll\cup-\Phi^-_\ll)}T^{{\bf (k-1)!}}\bar x_{-\eta}:\!\!\!\!&\text{ if $k>0$},\\
:\prod\limits_{\eta\in\Phi^-_{\ll}}T^{{\bf (-k)!}}\bar x_{-\eta}\prod\limits_{\eta\in-\Phi^+_{\ll}}T^{{\bf (-k-2)!}}\bar x_{-\eta}
\prod\limits_{\eta\in-\Dp(\p)\setminus (-\Phi^+_\ll\cup\Phi^-_\ll)}T^{{\bf (-k-1)!}}\bar x_{-\eta}:\!\!\!\!&\text{ if $k<0$},\\
:\prod\limits_{\eta\in \Phi_\ll}\bar x_{-\eta}:\!\!\!\!&\text{ if $k=0$}.
\end{cases}
\end{equation}

By Proposition \ref{f}, we see that
\begin{align}
(v_{k,\ll})_{(N-1)}(v^*_{k,\ll})&=C_{k,\ll}((k+1)\sum_{\eta\in\Phi^+_{\ll}}:\bar x_{\eta}\bar x_{-\eta}:+(k-1)\sum_{\eta\in-\Phi^-_{\ll}}:\bar x_{\eta}\bar x_{-\eta}:\label{ach}
\\&+k\sum_{\eta\in\Dp(\p)\backslash(\Phi_\ll\cup(-\Phi_{\ll}))}:\bar x_{\eta}\bar x_{-\eta}:)\notag
\end{align}
The multiplicative constant $C_{k,\ll}$ can be calculated explicitly using Proposition \ref{f}. In particular it is nonzero.

\begin{lemma}\label{nonsense}
Set $\l_{k,\ll}={\langle \Phi_\ll \rangle_{|\ft}}+k\nu(\varkappa)$. Then $v^*_{2k,\ll}\in L^\k(\l_{k,\ll}^*)\subset F(\bar\p)$.
\end{lemma}
\begin{proof}Recall that $\l^*=-w_0(\l)$, with $w_0$ the longest element of the Weyl group of $\k$.
Recall that $v_{2k,\ll}$ is constructed as in \eqref{hwwk} from the roots in $N(w_k)$. If $N(w_k)=\{\be_1,\dots,\be_s\}$ with $\be_i=(n_i+\tfrac{1}{2})\d'+\bar\be_i$, set $\gamma_i=(n_i+\tfrac{1}{2})\d'-w_0(\bar\be_i)$ and observe that $\{\gamma_1,\dots,\gamma_s\}$ is biconvex with respect to $\Dap$. Thus there is $u_k\in\Wa$ such that $N(u_k)=\{\gamma_1,\dots,\gamma_s\}$.

If we use $v(\Dap)$ as a set of positive roots with $v$ an element of the Weyl group of $\k$ then, letting $L_{v(\b)}(\mu)$ denote the irreducible highest weight module defined using $v(\b)$ instead of $\b$, it is clear that, since $\l$ is dominant integral,  $L^\k(\l)=L_{v(\b)}(v(\l))$. On the other hand, by Theorem 1.1 of \cite{MMJ}, if $w\in \Wa$ encodes via \eqref{hwwk} an highest weight vector for $L^\k(\l)$, then the highest weight vector for $L_{v(\b)}(v(\l))$ is encoded by $vwv^{-1}$, thus it is constructed, via \eqref{hwwk}, from $N_{v(\Dap)}(vwv^{-1})=v(N(w))$. Applying this discussion to $v=w_0$ and $w=u_k$ we find that the highest weight vector for $L_{w_0(\b)}(-\l)=L^\k(\l^*)$ is precisely $v^*_{2k,\ll}$.
\end{proof}
Recall that $\Sigma$ is the set of $\b_0$-stable abelian subspaces of $\p$ for the symmetric pair $(\aa,\r)$ we associated to the conformal embedding $\k\hookrightarrow  \g$. If $\ll\in\Sigma$, we say that $\ll$ occurs in $V_{1}(\g)$ if $\ll\in\Sigma^{even}_0$ when $\k$ is semisimple, $\ll\in\Sigma'_0$ if $\mathfrak{z}(\k)=\mathfrak z(\r)\ne\{0\}$, and $\l\in\Sigma'$ if $\mathfrak{z}(\r)\ne\mathfrak(\k)\ne\{0\}$. For notational convenience we set $\varkappa=0$ when $\k$ is semisimple.

\begin{prop}\label{hink}
Consider a conformal embedding $\k\hookrightarrow  \g$ where $\g$ is a Lie algebra of classical type. Assume that the embedding is maximal among conformal ones.

Let $W$ be a  simple vertex subalgebra of $V_{1}(\g)$ such that
 there is  $\ll$ that occurs in $V_{1}(\g)$ and $k_i\in\half\ganz$, $i\in I'$ such that   $V_{\bf j}(\k)\oplus L^\k(({\langle \Phi_\ll \rangle}+\sum_{i\in I'}k_i\nu(\varkappa_i))_{|\ft}) \subset W$.

Then either $W=V_{1}(\g)$ or there is $h\in \h_0$  such that $\be(h)=1$ for $\be\in\Phi_\ll$  and $\be(h)=0$ for $\be\in \D(\p)\setminus(\Phi_\ll\cup(-\Phi_\ll))$.
\end{prop}

For the proof of Proposition \ref{hink}, we need the following technical result.

\begin{lemma}\label{levi}
Let $\k\subset \g'\subsetneq \g$ be finite-dimensional Lie algebras with $\k$ reductive and $\g$  semisimple. Assume $\k$  maximal in $\g$ among reductive subalgebras. Let $\h_0$ be a Cartan subalgebra of $\k$, $\h$ a Cartan subalgebra of $\g$ containing $\h_0$. Then
$\h \cap \g'=\h_0$.
\end{lemma}
\begin{proof} First remark that, if $\g'$ is semisimple, it coincides with its algebraic closure $\bar\g'$, and if it is not, $\bar\g'$ is not semisimple, hence $\bar\g'\subsetneq\g$. We may therefore assume that $\g'$ is algebraic
(replacing it with $\bar\g'$). Let $\k_1$ be a maximal reductive subalgebra
of $\g'$, containing $\h\cap\g'$. But all maximal  reductive subalgebras in $\g'$ are conjugate; in particular, $\k$ and $\k_1$ have the same rank,
hence $\h\cap\g'$  cannot be larger than $\h_0$.
\end{proof}

\begin{proof}[Proof of Proposition \ref{hink}]  Set, for $k\in\ganz$ and $\ll\in\Sigma$,
\begin{align*}
h_{k,\ll}=(k+1)\sum_{\eta\in\Phi^+_{\ll}}:\bar x_{\eta}\bar x_{-\eta}:+(k-1)\sum_{\eta\in-\Phi^-_{\ll}}:\bar x_{\eta}\bar x_{-\eta}:+k\sum_{\eta\in\Dp(\p)\backslash(\Phi_\ll\cup(-\Phi_{\ll}))}:\bar x_{\eta}\bar x_{-\eta}:,
\end{align*}
and $h_\ll=h_{0,\ll}$.

Set $\l={\langle \Phi_\ll \rangle_{|\ft}}+\sum_{i\in I'}k_i\nu(\varkappa_i)$.
Recall that $W=\oplus_{n\in\C}W_n$ is the eigenspace decomposition of $W$ with respect to $\omega_{(1)}$.

We now prove that $\sum_{i\in I'}h_{2k_i,\ll'\cap\aa_i}+h_{\ll_0}\in W_1$.
With the notation of \eqref{llprime},  we have that $\ll=\ll'\oplus\ll_0$.
By Theorem \ref{decoclassical}, $:\prod_{i\in I'}v_{2k_i,\ll\cap\aa_i'}v_{0,\ll_0}:$ is a highest weight vector for $L^\k(\l)$. Since $W$ is simple, by Proposition \ref{bilinearfr}, $L^\k(\l^*)$ occurs in $W$. By Lemma \ref{nonsense}, we have that $:\prod_{i\in I'}v^*_{2k_i,\ll\cap\aa_i'}v^*_{0,\ll_0}:\in W$, hence
$$\left(:\prod_{i\in I'}v_{2k_i,\ll\cap\aa_i'}v_{0,\ll_0}\right)_{(N-1)}\left(:\prod_{i\in I'}v^*_{2k_i,\ll\cap\aa_i'}v^*_{0,\ll_0}:\right)\in W.
$$
 By Proposition \ref{f}, we have that $\sum_{i\in I'}h_{2k_i,\ll'\cap\aa_i}+h_{\ll_0}\in W$.
By \eqref{ach}, since $\omega$ is the Virasoro vector of $F(\bar\p)$, it follows that
$\sum_{i\in I'}h_{2k_i,\ll'\cap\aa_i}+h_{\ll_0}\in W_1$.

We can now prove the statement.
Given $\a,\be\in\D(\p)$, let $E_{\a,\be}$ be the linear map on $\p$ given by $E_{\a,\be}(x_\be)=x_\a$. By \eqref{theta} we have
$$\Theta(E_{\beta,\beta}-E_{-\beta,-\beta})=:\bar x_\be \bar x_{-\be}:,$$
hence $\sum_{i\in I'}h_{2k_i,\ll'\cap\aa_i}+h_{\ll_0}$ belongs to the diagonal Cartan of $so(\p)$. By the explicit description of the embedding $\Theta$ of $\g$ in $F(\bar\p)$, we see that a Cartan subalgebra of $\g$ containing $\ft$ is contained in the diagonal Cartan of $so(\p)$. Since $\sum_{i\in I'}h_{2k_i,\ll'\cap\aa_i}+h_{\ll_0}\in W_1\subset \g$, it follows that $\sum_{i\in I'}h_{2k_i,\ll'\cap\aa_i}+h_{\ll_0}$ belongs to a Cartan subalgebra of $\g$ containing $\h_0$. Then, applying Lemma \ref{levi} with $h=\sum_{i\in I'}h_{2k_i,\ll'\cap\aa_i}+h_{\ll_0}$ and $\g'=W_1$, we get that either $W_1=\Theta(\g)$ or $\sum_{i\in I'}h_{2k_i,\ll'\cap\aa_i}+h_{\ll_0}\in\Theta(\ft)$.

In the first case $V_{1}(\g)= W$. In the second case, since $\Theta(\varkappa_i)=\sum_{\a\in\Dp(\p\cap\aa_i)}:\bar x_\a\bar x_{-\a}:\in\Theta(\h_0)$, we have that $ h_\ll=\sum_{i\in I'}h_{2k_i,\ll'\cap\aa_i}-\sum_{i\in I'}2k_i\Theta(\varkappa_i)+h_{\ll_0}= \Theta(h)$ for some $h\in\h_0$.  Clearly
$$\Theta(h)=\half\sum_{\be\in\D(\p)}\be(h):\bar x_{\be}\bar x_{-\be}:\ =\sum_{\be\in\Dp(\p)}\be(h):\bar x_{\be}\bar x_{-\be}:.$$
Thus
$$\sum_{\be\in \Phi_\ll}:\bar x_\be \bar x_{-\be}:=\sum_{\be\in\Dp(\p)}\be(h):\bar x_{\be}\bar x_{-\be}:
$$
for some $h\in\h_0$.
Notice that $\{:\bar x_\be\bar x_{-\be}:\mid \be\in\Dp(\p)\}$ is a linearly independent set. Moreover, since $\ll$ is abelian, if $\be\in\Phi_\ll$, then $-\be\not\in\Phi_\ll$. Thus, since
$$
h_\ll=\sum_{\be\in \Phi_\ll\cap\Dp(\p)}:\bar x_\be \bar x_{-\be}: -\sum_{\be\in (-\Phi_\ll)\cap\Dp(\p)}:\bar x_\be \bar x_{-\be}:=\sum_{\be\in\Dp(\p)}\be(h):\bar x_{\be}\bar x_{-\be}:,
$$
the statement follows.
\end{proof}

\begin{theorem}\label{p}Let $\k\hookrightarrow  \g$ be a maximal conformal embedding  and $\g$ a simple classical Lie algebra. Assume that $W$ is a simple vertex subalgebra of $V_{1}(\g)$ such that
$V_{\mathbf{j}}(\k)\subset W$.
Then either $W=V_{1}(\g)$ or  $W$ is a simple current extension  of  $V_{\bf j}(\k)$.
\end{theorem}
\begin{proof} Assume $W\ne V_{1}(\g)$. If $L^\k(({\langle \Phi_\ll \rangle}+\sum_{i\in I'}k_i\nu(\varkappa_i))_{|\ft})$ occurs in $W$, then, by Proposition \ref{hink}, $\ll=\ll(h)$ for some $h\in\h_0$. But then we can apply Proposition \ref{aresimplecurrents} to each irreducible component of $(\aa,\r)$ to deduce that $L^\k(({\langle \Phi_\ll \rangle}+\sum_{i\in I'}k_i\nu(\varkappa_i))_{|\ft})$ is a simple current.

If $\g=so(n,\C)$ and $L^\k(\langle \mathfrak m \rangle + \theta_\p)\subset W$, then, by Lemma \ref{special}, we have that $L^\k(\langle \mathfrak m \rangle + \theta_\p)$ is also a simple current. We have shown that
$W$ is a sum of special simple currents. The result now follows from Corollary \ref{sumsimpleisext}.
\end{proof}

\begin{cor}\label{subgroupvertex}If $\k\hookrightarrow  \g$ is a maximal conformal embedding with $\g$ of classical type then
the simple intermediate vertex algebras $V_{\mathbf{j}}(\k)\subset W \subset V_1(\g)$ are in one to one correspondence with the subgroups of $S^\k_{V_1(\g)}$.
\end{cor}
\begin{proof}
By Theorem \ref{p}, $W$ is a simple current extension. The result now follows from Corollary \ref{subgroup}.
\end{proof}

\begin{rem}\label{exc} Theorem \ref{p} holds also when $\g$ is of exceptional  type and $\k$ is a maximal conformal regular subalgebra.
In this case there is indeed nothing to prove: browsing through decomposition formulas given in \cite{KS} Êone finds that in the decomposition of $V_{1}(\g)$ either the summands are simple currents or the decomposition has
exactly two summands. In both cases  any intermediate simple vertex algebra is a simple current extension.
\end{rem}
\begin{rem}\label{conje}  We want to present some argument towards conjecture 1.1. By Remark \ref{exc}, we are left to deal with embeddings
$\k\hookrightarrow\g$ where $\k$ is not a regular subalgebra of $\g$. Again one looks at decomposition formulas given in \cite{KS}; it turns out that there are six  cases in which there are more than two summands in the decomposition and not all of them  Êare simple currents:
\begin{enumerate}
\item $G_2\times A_1 \hookrightarrow F_4,$
\item $A_2\hookrightarrow  E_6,$
\item $A_2\hookrightarrow  E_7,$
\item $C_2 \hookrightarrow  E_8,$
\item $A_1\times A_2\hookrightarrow  E_8,$
\item $A_1\times G_2\hookrightarrow  E_7$.
\end{enumerate}
Cases 1-4 can be settled under the additional hypothesis that a proper simple intermediate algebra $V_{\bf j}(\k)\subset W\subset V_1(\g)$ is rational.
Case (1) is dealt with the following argument. The decomposition of $V_{1}(\g)$, in the notation of \cite{KS}, is:
$$\L_0=(\dot{\L}_0,8\ddot{\L}_0)+(\dot{\L}_2,4\ddot{\L}_0+4\ddot{\L}_1)+ (\dot{\L}_0,8\ddot{\L}_1).$$
\end{rem}
The middle summand in the right hand side is not a simple current.  Assume by contradiction that the sum of the first two factors is  a proper algebra $W$. Then $(\dot{\L}_0,8\ddot{\L}_1)$ is a $W$-module: looking at the fusion rules one gets $(\dot{\L}_2,4\ddot{\L}_0+4\ddot{\L}_1)\cdot (\dot{\L}_0,8\ddot{\L}_1)=0$, against the fact that $V_1(\g)$ is simple.\par
For cases 2,3,4, we can invoke Gannon's classification
of $A_2$ and $C_2$ modular invariants to get our claim.
Indeed, according to \cite{Gan1}, \cite{Gan2}, in all these cases there are only three physical modular invariants that can arise from conformal embeddings $V_{\mathbf{j}}(\k)\subset W$, with $W$ a rational vertex algebra. Since one arises from the conformal embedding of $V_{\mathbf{j}}(\k)$ in itself, one from the embedding of $V_{\mathbf{j}}(\k)$ in the sum of its simple currents and the third from the embedding of  $V_{\mathbf{j}}(\k)$ in $V_1(\g)$, there are no other possibilities.
Cases 5,6  are unclear.
\section{Computing $SR(V_1(\g),V_{\mathbf{j}}(\k))$ for maximal conformal \\embeddings}\label{8}
Motivated by Corollary \ref{subgroupvertex}, in this section we compute explicitly the groups $S^\k_{V_1(\g)}$ and the rings $SR(V_1(\g),V_{\mathbf{j}}(\k))$ for maximal conformal embeddings $\k \hookrightarrow  \g$ with $\g$ of classical type.

Let $(\aa,\mathfrak r)$ be an irreducible symmetric pair.  Lemma \ref{fundcoweights} singles out the following elements of $\Omega_{\mathfrak r}$ (cf. \eqref{Of}, \eqref{oi} for notation; also recall
from the discussion before Lemma    \ref{special} the definition of $\omega_S$ and set $o_S=L^\k(j_S\omega_S)$):
\begin{enumerate}
\item $x_0:=\prod_S o_{S}$,
\item $o_s, s\in S: a_s=a_p$ for some $p\in P$ or $\sum_{p\in P}a_p=a_s$,
\item $o_so_{s'},\,s\in S,\,s'\in S', S\ne S', \sum_{p\in P}a_p=2, a_{s}=a_{s'}=1$.
\end{enumerate}
If $\r$ is semisimple, let $H_\r$  be the set of representation of type (2) and (3) with $\uno$ added and, if  $\theta_\p$ is long noncompact, also $x_0$ is added. If $\r$ is not semisimple, let $\varpi\in\mathfrak z(\r)$ be the element such that $\a_\p(\varpi)=1$. Define $\omega\in\mathfrak z(\r)^*$ by setting $\omega(\varpi)=1$.   If $x\in \Omega_{[\r,\r]}$, define $\Phi_x=\{\a\in\D(\p)\mid \a(f(x)+(-1-\a_\p(f(x))\varpi)=1\}$. In this section we identify $\Omega_{\mathfrak z(\r)}$ with $\C$ by identifying $L^{\mathfrak z(\r)}(c\omega)$ with $c\in\C$. Note finally that $\Omega_\r=\Omega_{\mathfrak z(\r)}\times \Omega_{[\r,\r]}$.

Set
$$H_\r=\{(n,x)\in \ganz\times \Omega_{[\r,\r]}\subset \Omega_\r\mid x\text{ of type (2) or (3), } \langle\Phi_x\rangle(\varkappa)\equiv n\mod \tfrac{\dim\p}{2}\}.
$$

Finally, if $(\aa,\r)$ is any symmetric pair and $(\aa_i,\r\cap\aa_i)$ ($i\in I$) are the irreducible components, then we set
$H_\r=\prod_{i\in I}H_{\r\cap\aa_i}$.
\begin{prop}\label{simplecurrentsinFp} If $(\aa,\r)$ is a symmetric pair then   $H_\r=S^{\r}_{F(\bar\p)}$. Moreover, $H_\r$ is a subgroup of $\Omega_\r$ and
\begin{equation}\label{SR}SR(F(\bar\p),V_{\mathbf{j}}(\r))=\ganz[H_\r].
\end{equation}
\end{prop}
\begin{proof}  Since $F(\bar\p)=\otimes_{i\in I}F(\ov{\p\cap\aa_i})$, we can clearly assume that $(\aa,\r)$ is irreducible.

Assume first $\r$ semisimple. Consider $x\in H_\r\backslash\{x_0\}$. Assume $x$ of type (2).
Let $\l\in\D(\p)$ and let $\gamma\in\Da$ corresponding to it under the map described in \eqref{weights}. Observe that $\l(\omega_s^\vee)=c_s(\gamma)-1$. Since $\frac{2}{a_p}\delta-\gamma$ is positive root, we have $c_s(\gamma)\le \frac{2a_s}{a_p}=2$. It follows that $\l(\omega_s^\vee)\in\{1,0,-1\}$. Set $\ll=\sum_{\l(\omega^\vee_s)=1}\C x_\l$. This is obviously an abelian $\b_0$-stable subspace of $\p$ and, by Proposition \ref{aresimplecurrents}, $x=L^\k(\langle\Phi_\ll\rangle)$. Thus, by Theorem \ref{decoeabeliani}, $x$ occurs in $F(\bar\p)$.
Assume now $x$ of type (3).
Let $\l\in\D(\p)$ and let $\gamma\in\Da$ corresponding to it under the map described in \eqref{weights}. Observe that $\l(\omega_s^\vee+\omega_{s'}^\vee)=c_s(\gamma)+c_{s'}(\gamma)-1$. Since $a_p=2$, $\delta-\gamma$ is a positive root so $c_s(\gamma)+c_{s'}(\gamma) \le a_s+a_{s'}=2$. It follows that $\l(\omega_s^\vee+\omega_{s'}^\vee)\in\{-1,0,1\}$ and we can conclude as in the previous case.
If finally $x=x_0$ then, by Lemma \ref{special} and Theorem \ref{decoeabeliani}, $x$ occurs in $F(\bar\p)$.

On the other hand, if $x\in S^\r_{F(\bar\p)}$, then either   $\theta_\p$ is long and $x=x_0$ or $x=L^\k(\langle \Phi_\ll \rangle)$.  In the first case we have $x\in H_\r$. In the second case,
since $F(\bar \p)$ is simple, the proof of Proposition \ref{hink} gives that $\ll=\ll(h)$ for some $h\in\h_0$ and Lemma \ref{fundcoweights} gives that $x\in H_\r$.

Assume now $\r$ not semisimple. Consider $(n,x)\in H_\r$. We now show that  $\Phi_x=\Phi_\ll$ for some $\ll\in\Sigma'$. Assume that $x$ is of type (2). Let $\l\in \D(\p)$ and let $\gamma\in \Dap$ be the root corresponding to it under the map described in \eqref{weights2}. Let $h=\omega^\vee_s+(-1-\a_\p(\omega_s^\vee))\varkappa$. Observe that $\l(h)=(1-a_s)c_p(\gamma)+c_s(\gamma)-c_q(\gamma)$. Since $\delta-\gamma$ is a positive root, we have $c_s(\gamma)\le a_s\le 2$. It follows that $\l(h)\in\{1,0,-1\}$. Set $\ll=\sum_{\l(h)=1}\C x_\l$. This is obviously an abelian $\b_0$-stable subspace of $\p$ and $\ll\in \Sigma'$. By Proposition \ref{aresimplecurrents}, $L^\k(\langle\Phi_\ll\rangle_{|\mathfrak z(\r)})\otimes x=L^\k(\langle\Phi_\ll\rangle)$. Thus, by Theorem \ref{finalehs}, $(n,x)$ occurs in $F(\bar\p)$.
Assume now $x$ of type (3).
Let $\l\in\D(\p)$ and let $\gamma\in\Da$ corresponding to it under the map described in \eqref{weights2}. Let $h=\omega_s^\vee+\omega_{s'}^\vee+(-1-\a_\p(\omega_s^\vee+\omega_{s'}^\vee))\varkappa$. Observe that $\l(h)=c_s(\gamma)+c_{s'}(\gamma)-1$. Since  $\delta-\gamma$ is a positive root, $c_s(\gamma)+c_{s'}(\gamma) \le a_s+a_{s'}=2$. It follows that $\l(h)\in\{-1,0,1\}$ and we can conclude as in the previous case.

On the other hand, if $(n,x)\in S^\r_{F(\bar\p)}$, then
 $(n,x)=L^\k(\langle \Phi_\ll \rangle+k\nu(\varkappa))$ with $\ll\in \Sigma'$ and $n=k\dim\p+\langle \Phi_\ll\rangle(\varkappa)$.  Since $F(\bar \p)$ is simple and $L^\k(\langle \Phi_\ll \rangle)$ occurs in $F(\bar\p)$, by Lemma \ref{nonz}, we have that  $L^\k(\langle \Phi_\ll \rangle^*)$ occurs in $F(\bar\p)$. The proof of Proposition \ref{hink} then gives that $\ll=\ll(h)$ for some $h\in\h_0$ and Lemma \ref{fundcoweights} gives that $(\langle \Phi_\ll\rangle(\varkappa),x)\in H_\r$, thus $(n,x)\in H_\r$.

Since $F(\bar \p)$ is simple, formula \eqref{SR} follows from Lemma \ref{argfeng}.
\end{proof}

Let $(\g,\k)$ be a conformal pair with $\g$ of classical type.
Recall from Section 5.3  the symmetric pair $(\aa,\r)$ associated to the pair $(\g,\k)$. Note that $\Omega_\k\cong L^\u(0)\otimes \Omega_\k\subset \Omega_\r$.  We set
$$
H_\k=H_\r\cap\Omega_\k.
$$
\begin{prop}\label{centro}\
\begin{enumerate}
\item Assume $\k$ semisimple.
If $\g=so(n,\C)$ then
$$S^{\k}_{V_{1}(\g)}=\{x\in H_\k\mid x=L^\k(\langle\Phi_\ll\rangle) \text{ for some }\ll\in\Sigma^{even}\}\cup\{x_0\},$$
where the rightmost element appears only if $\dim\mathfrak m$ is odd.
If $\g=sp(n, \C)$ or $\g=sl(n+1, \C)$, then
$S^{\k}_{V_{1}(\g)}=H_\k$.
\item Assume $\k$ not semisimple and $\mathfrak z(\k)=\mathfrak z(\r)$. Then
\begin{align*}S^{\k}_{V_{1}(\g)}=\{(n,x)\in H_\k\mid (&n,x)=L^\k(\langle\Phi_\ll\rangle_{|\ft}+k\dim(\p\cap\aa')\kappa),k \in \tfrac{1}{2}\ganz,\\&\ll\in\Sigma'_0,\,\dim\ll+k\dim(\p\cap\aa')\in2\ganz \}.\end{align*}
\item If $\k=\C\times sl(p,\C)\times sl(q,\C)$ and $\g=sl(p+q,\C)$ then there is an isomorphism $\Omega_\k\cong \C\times \ganz/p\ganz\times \ganz/q\ganz$ such that, with the notation of Lemma \ref{elementary},
$$
S^\k_{V_1(\g)}=\{(t,x,y)\mid (x,y)\in Ker\varphi,\, t\in \psi(x,y)\}.
$$
\end{enumerate}
In particular $S^\k_{V_1(\g)}\cong\ganz$.
\end{prop}
\begin{proof}We first prove (1), so that $\k$ is assumed to be semisimple. Assume first that $(\aa,\r)$ is irreducible.

If $\g=so(n,\C)$, it suffices to remark $x\in H_\r$ actually belongs to $S^{\k}_{V_{1}(\g)}$ if and only if $x$ occurs in $F(\bar\p)^0$ (cf. \eqref{F}). Assume $x\ne x_0$. Let $\ll\in\Sigma$ be such that  $x= L^\k(\langle\Phi_\ll\rangle)$. Then
$v_\ll\in F(\bar\p)^0$ if and only if  $\dim\ll$ is even. Now assume $x=x_0$. Then the highest weight vector of $x_0$ is, by \eqref{anomalo}, given by $:T(\bar x_{\theta_\p})v_{0,\mathfrak m}:$, hence
it belongs to $F(\bar\p)^0$ if and only if $\dim\mathfrak m$ is odd.\par
Now let $\g=sp(n, \C),sl(n+1,\C)$. If $x\in S^{\k}_{V_{1}(\g)}$, then  $L^\u(0)\otimes x$ occurs in $F(\bar \p)$, so, by Proposition \ref{simplecurrentsinFp}, $L^\u(0)\otimes x\in H_\r$, hence $x\in H_\k$.

On the other hand, if $x\in H_\k$, by Proposition \ref{simplecurrentsinFp},  $L^\u(0)\otimes x=L^\k(\langle \Phi_\ll \rangle)$ for some $\ll\in \Sigma$. Thus  $\langle \Phi_\ll \rangle_{|\mathfrak{t}'}=0$. It follows that $\ll\in\Sigma_0$, hence, by Theorem \ref{decoclassical}, $x$ occurs in $V_{1}(\g)$. If $(\aa,\r)$ is reducible, then the result follows by applying the above argument to each irreducible component.

 Part (2) follows from \eqref{gnhs}.

 We now prove (3). Enumerate the simple roots of type $A_n$ by $1,\dots,n$ from left to right.  We note that $\langle\ll(i,j)\rangle_{|\ft\cap[\k,\k]}=(\omega_i,\omega_j)$. In particular $L^{[\k,\k]}(\langle\ll(i,j)\rangle_{|\ft\cap[\k,\k]})$ is a special simple current. We define the isomorphism between $\Omega_\k$ and $\C\times \ganz/p\ganz\times \ganz/q\ganz$  by
 $L(c\kappa)\otimes L^{[\k,\k]}(\langle\ll(i,j)\rangle_{|\ft\cap[\k,\k]})\mapsto (c,i+p\ganz,j+q\ganz)$. Then the result follows from Proposition \ref{slpq} and Lemma \ref{elementary}.
\end{proof}

Using Proposition \ref{centro}, we  can describe explicitly  the structure of $S^{\k}_{V_{1}(\g)}$ when $\g$ is of classical type.
In tables below we use the list
of conformal embeddings from \cite{AGO}. Our results in the adjoint case are
given in Table 2. For all other conformal embeddings in $so(n,\C)$ our
results are given in Table 3, and for those in  $sp(n,\C)$ and $sl(n+1,\C)$ in Tables 4 and 5, respectively. We number the simple roots of $\k$ by seeing  its diagram as a sub-diagram of the affine diagram corresponding to the pair $(\aa,\r)$.
\vskip10pt
\centerline{
\begin{tabular}{ l | c | c | l}
Type of $\k$  & $S^{\k}_{F(\bar\k)}$  &$S^{\k}_{V_{1}(so(\k))}$&generators for $S^{\k}_{V_{1}(so(\k))}$\\  \hline
$A_n$, $n$ odd&$\ganz/(n+1)\ganz$&$\ganz/\tfrac{n+1}{2}\ganz$&$o_2$\\\hline
$A_n$, $n$ even&$\ganz/(n+1)\ganz$&$\ganz/(n+1)\ganz$&$o_1$\\\hline
$B_n$&$\ganz/2\ganz$&$\{\uno\}$&\\\hline
$C_n$, $n\equiv0,3\mod 4$&$\ganz/2\ganz$&$\ganz/2\ganz$&$o_n$\\\hline
$C_n$, $n\equiv1,2\mod 4$&$\ganz/2\ganz$&$\{\uno\}$&\\\hline
$D_n$, $n\equiv0\mod 4$&$\ganz/2\ganz\times \ganz/2\ganz$&$\ganz/2\ganz\times \ganz/2\ganz$&$o_{n-1},o_n$\\\hline
$D_n$, $n\equiv1\mod 4$&$\ganz/4\ganz$&$\ganz/4\ganz$&$o_n$\\\hline
$D_n$, $n\equiv2\mod 4$&$\ganz/2\ganz\times \ganz/2\ganz$&$ \ganz/2\ganz$&$o_1$\\\hline
$D_n$, $n\equiv3\mod 4$&$\ganz/4\ganz$&$\ganz/2\ganz$&$o_1$\\\hline
$E_6$&$\ganz/3\ganz$&$\ganz/3\ganz$&$o_1$\\\hline
$E_7$&$\ganz/2\ganz$&$\{\uno\}$&\\\hline
$E_8$&$\{\uno\}$&$\{\uno\}$&\\\hline
$F_4$&$\{\uno\}$&$\{\uno\}$&\\\hline
$G_2$&$\{\uno\}$&$\{\uno\}$&\\\hline
 \end{tabular}}
 \vskip5pt
\centerline{\small Table 2: the adjoint case.}

\small{
\centerline{
\begin{tabular}{ l | c | c | l}
conformal embedding  & $S^{\k}_{F(\bar\p)}$  &$S^{\k}_{V_{1}(so(\p))}$&generators\\ & & & for $S^{\k}_{V_{1}(so(\p))}$\\  \hline
$so(m,\C) \hookrightarrow so((m+2)(m-1)/2)$, & $\ganz/2\ganz$ & $\{\uno\}$ &\\$m\equiv 1,3\mod 8$  & & \\\hline
$so(m,\C) \hookrightarrow so((m+2)(m-1)/2)$, & $\ganz/2\ganz$ & $\ganz/2\ganz$ &$o_{(m-3)/2}$\\ $m\equiv 5,7\mod 8$  & &  \\\hline
$so(m,\C) \hookrightarrow so((m+2)(m-1)/2)$, & $\ganz/2\ganz\times \ganz/2\ganz$ & $\ganz/2\ganz\times \ganz/2\ganz$ &$o_{0}$, $o_{1}$\\ $m\equiv 0\mod 8$  & & \\\hline
$so(m,\C) \hookrightarrow so((m+2)(m-1)/2)$,  & $\ganz/4\ganz$ & $\ganz/4\ganz$ &$o_{0}$\\ $m\equiv 6\mod 8$  & & \\\hline
$so(m,\C) \hookrightarrow so((m+2)(m-1)/2)$,  & $\ganz/4\ganz$ & $\ganz/2\ganz$ &$o_{(m-2)/2}$\\ $m\equiv 2\mod 8$ & & \\\hline
$so(m,\C) \hookrightarrow so((m+2)(m-1)/2)$,  & $\ganz/2\ganz\times \ganz/2\ganz$ & $\ganz/2\ganz$ &$o_{(m-2)/2}$\\$m\equiv 4\mod 8$ & & \\\hline
$A_1\hookrightarrow B_2$  & $\ganz/2\ganz$  & $\{\uno\}$&\\\hline
$C_m\hookrightarrow  so((2m+1)(m-1))$, &$\ganz/2\ganz$ &$\ganz/2\ganz$&$o_m$\\  $m\equiv 0,1\mod 4$  & & \\\hline
$C_m\hookrightarrow  so((2m+1)(m-1))$, &$\ganz/2\ganz$ &$\{\uno\}$&\\ $m\equiv 2,3\mod 4$    & & \\\hline
$\C\times A_{m-1}\hookrightarrow D_m$ &$\ganz$ &$2\ganz$&$(-2,o_3)$\\&&&$(-2,\uno)$ if $m=1$\\\hline
$sp(m,\C)\times sp(n,\C)\hookrightarrow so(4mn,\C)$ &$\ganz/2\ganz$  &$\ganz/2\ganz$ & $o_0o_{m+n}$\\
$mn$ odd & & \\
\hline
$sp(m,\C)\times sp(n,\C)\hookrightarrow so(4mn,\C)$ &$\ganz/2\ganz$  &$\{\uno\}$ & \\
$mn$ even & & \\
\hline
$so(m,\C)\times so(n,\C)\hookrightarrow so(mn,\C)$&$(\ganz/2\ganz)^3$ &$(\ganz/2\ganz)^3$&$o_{(m-2)/2},o_{(m+2)/2},$\\$m\equiv n\equiv 0\mod 4$&&&$o_0o_{(m+n)/2}$\\\hline
$so(m,\C)\times so(n,\C)\hookrightarrow so(mn,\C)$&$\ganz/2\ganz\times \ganz/4\ganz$ &$\ganz/2\ganz\times \ganz/4\ganz$&$o_{(m-2)/2},$\\$m\equiv0, n\equiv 2\mod 4$&&&$o_0o_{(m+n)/2}$\\\hline
$so(m,\C)\times so(n,\C)\hookrightarrow so(mn,\C)$&$\ganz/2\ganz\times \ganz/4\ganz$ &$\ganz/2\ganz\times \ganz/2\ganz$&$o_{(m-2)/2},o_{(m+2)/2}$\\$m\equiv n\equiv 2\mod 4$&&&\\\hline
$so(m,\C)\times so(n,\C)\hookrightarrow so(mn,\C)$&$\ganz/2\ganz\times \ganz/2\ganz$ &$\ganz/2\ganz$&$o_{(m+2)/2}$\\$m$ even, $n$ odd&&&\\\hline
$so(m,\C)\times so(n,\C)\hookrightarrow so(mn,\C)$&$\ganz/2\ganz\times \ganz/2\ganz$ &$\ganz/2\ganz$&$o_{(m-3)/2}o_{(m+1)/2}$\\$m,n$ odd&&&\\\hline
$C_4\hookrightarrow D_{21}$ &$\ganz/2\ganz$ &$\ganz/2\ganz$&$o_3$\\\hline
$F_4\hookrightarrow D_{13}$ &$\{\uno\}$ &$\{\uno\}$&\\\hline
$A_7\hookrightarrow D_{35}$ & $\ganz/4\ganz$&$\ganz/2\ganz$&$o_3$\\\hline
$D_8\hookrightarrow D_{64}$ &$\ganz/2\ganz$ &$\ganz/2\ganz$&$o_6$\\\hline
$B_4\hookrightarrow D_{8}$ &$\{\uno\}$ &$\{\uno\}$&\\\hline
$so(m,\C)\times so(n,\C)\hookrightarrow so(m+n,\C)$ & $\ganz/2\ganz\times\ganz/2\ganz$&$\ganz/2\ganz$&$o_{\lfloor
\frac{m}{2}\rfloor-1}o'_{\lfloor \frac{n}{2}\rfloor-1}$\\
$m\ge 3,\,n\ge3$&&&\\
\hline
$so(m,\C)\times so(2,\C)\hookrightarrow so(m+2,\C)$ & $\ganz/2\ganz\times\ganz$&$\ganz$&$(1,o_{m-1})$\\
$m\ge 3$&&&\\
\hline
 \end{tabular}}}
\vskip10pt
\centerline{\small Table 3:  $S^{\k}_{V_{1}(so(\p))}$ with $(\aa,\k)$ symmetric pair not of adjoint type.}
\centerline{
\begin{tabular}{ l | c | l }
conformal embedding  & $S^{\k}_{V_{1}(sp(n,\C))}$ & generators\\
  \hline
  $\C\times A_{m-1}\hookrightarrow C_m$& $\ganz$ & $(-2,o_3)$\\\hline
$so(m,\C)\times A_1\hookrightarrow C_m$, $m$ odd & $\ganz/2\ganz$ & $o_3$\\\hline
$so(m,\C)\times A_1\hookrightarrow C_m$, $m\equiv_4 0$  & $\ganz/2\ganz\times  \ganz/2\ganz$ & $o_{\frac{m}{2}+2}o_1,\,o_{\frac{m}{2}+1}o_1$\\\hline
$so(m,\C)\times A_1\hookrightarrow C_m$, $m\equiv_4 2$  & $\ganz/4\ganz$& $o_{\frac{m}{2}+1}o_1$\\\hline
$A_{5}\hookrightarrow C_{10}$ & $\ganz/3\ganz$ & $o_2$\\\hline
$D_6\hookrightarrow C_{16}$ & $\ganz/2\ganz$ &$o_1$\\\hline
$E_7\hookrightarrow C_{28}$ & $\{\uno\}$\\\hline
$C_3\hookrightarrow C_{7}$ & $\{\uno\}$\\\hline
$A_1\hookrightarrow C_2$ & $\{\uno\}$\\\hline
 \end{tabular}}
\vskip10pt
\centerline{\small Table 4:  $S^{\k}_{V_{1}(sp(n,\C))}$.}
\vskip10pt
\centerline{
\begin{tabular}{ l | c | l }
conformal embedding  & $S^{\k}_{V_{1}(sl(n+1,\C))}$ & generators\\
  \hline
$A_{p-1}\times A_{q-1}\hookrightarrow A_{pq-1}$,  & $\ganz/u\ganz$,  $u=GCD(p,q)$& $o_{r}o_{p+s}$,\\&&$r=p/u,\,s=q/u$\\\hline
$so(m,\C)\hookrightarrow A_{m-1}$  & $\ganz/2\ganz$ & $o_2$\\\hline
$A_{m-1}\hookrightarrow A_{\binom{m}{2}-1}$, $m$ odd &$\{\uno\}$\\\hline
$A_{m-1}\hookrightarrow A_{\binom{m-1}{2}-1}$,   $m$ odd &$\{\uno\}$\\\hline
$A_{m-1}\hookrightarrow A_{\binom{m}{2}-1}$, $m$ even &$\ganz/2\ganz$ & $o_{m/2}$\\\hline
$A_{m-1}\hookrightarrow A_{\binom{m-1}{2}-1}$,  $m$ even &$\ganz/2\ganz$ & $o_{m/2}$\\\hline
$E_6\hookrightarrow A_{26}$ & $\{\uno\}$\\\hline
$D_5\hookrightarrow A_{15}$ & $\{\uno\}$\\  \hline
$\C\times A_{p-1}\times A_{q-1}\hookrightarrow A_{p+q-1}$ & $\ganz$&  $(1,o_{i}o'_{j})$,\,$i-1\equiv-\frac{q+1}{m}\mod p$,\\
&& $j-1\equiv\frac{p+1}{m}\mod q$\\\hline
 \end{tabular}}
\vskip10pt
\centerline{\small Table 5:  $S^{\k}_{V_{1}(sl(n+1,\C))}$.}
\medskip
The last line of Table 5 is a restatement of Proposition \ref{centro} (3). We used the notation of Proposition \ref{slpq}.
The other instances in
Tables 4 and 5 have been derived using the following procedure. We first find the subspaces $\ll\in\Sigma_0$ such that $\ll=\ll(h)$ for some $h\in\h_0$. To accomplish this, we simply list all the weights of $\p^+$
and, for  any $h\in\h_0$ of the type described in Lemmas \ref{sumcoweights}, \ref{fundcoweights}, we compute the number of weights $\l$ for which $\l(h)=1$, the number of weights for which $\l(h)=-1$
and we pick the $h$ for which these two numbers coincide.  As an example, we work out the most  difficult case, that of  the conformal embedding   $A_{p-1}\times A_{q-1}\hookrightarrow A_{pq-1}$ in Table 5.
\par
Recall that in this case the corresponding symmetric pair is $(A_{p+q-1},A_{p-1}\times A_{q-1}\times\C\varkappa)$.
Order the roots of $A_{p+q-1}$  from left to right so that the first $p-1$ (resp. the last $q-1$) correspond to the roots of the  $A_{p-1}$-component (resp. $A_{q-1}$-component) of $\k$. We want to prove that if
$r=p/M,\,s=q/M$, then  $S^{\k}_{V_{1}(\g)}$ is cyclic of order $M$ generated by  $o_{r}o_{p+s}$.
\par
Denote by $\a_1,\ldots,\a_{p+q-1}$ the simple roots of $A_{p+q-1}$ and set $\a_{ij}=\a_i+\ldots+\a_j$. Then
\begin{align*}\Dp(\p)=&\{\a_{ij}\mid 1 \leq i\leq p\leq j\leq p+q-1\}.\end{align*}
Consider $x\in S^{\k}_{V_{1}(\g)}$. If $x=o_s$, we observe that $\a_{ij}(\omega_s^\vee+(-\epsilon-\a_\p(\omega_{s}^\vee))\varpi)\leq 0$ if $\epsilon=1$ and
$\a_{ij}(\omega_s^\vee+(-\epsilon-\a_p(\omega_{s}^\vee))\varpi)\geq 0$ if $\epsilon=0$ (cf.     Lemma \ref{sumcoweights}, \eqref{sumofone}), hence $\langle\Phi_\ll\rangle_{|\ft'}\ne 0$.
  So we may assume that $x=o_u o_{p+v},\,1\leq u<p,\,1\leq v <q$. Let $\omega_{uv}=\omega_{u}^\vee+\omega_{p+v}^\vee+(-1-\a_p(\omega_{u}^\vee+\omega_{p+v}^\vee))\varpi$
. Then $\a_{ij}(\omega)=1$ exactly when $i \leq u,\,p+v\leq j$; these indices are $u(q-v)$ in number.
Similarly, $\a_{ij}(\omega)=-1$ exactly when  $i > u,\,p+v>j$; these indices are $(p-u)v$ in number.
Therefore we are led to solve the  linear Diophantine equation
$
uq=pv,
$
which gives the desired result.

The above procedure is easily done in the exceptional cases; in the remaining classical cases it is performed in the same way  using the pictorial display of positive roots
given in \cite{CP}.
\begin{rem} If $\g=sp(n, \C)$ and $\k$ is semisimple, then $S^{\k}_{V_{1}(\g)}$ is a subgroup of index $2$ of $\Omega_\k$.
\end{rem}
\begin{rem} The structure of simple current extension in the framework of conformal nets has been studied, in special cases, in \cite{F} Ê(cf. Theorems 3.8 and 3.11).
The results obtained there agree with those displayed in Tables 2-5.
\end{rem}
\vskip10pt
We give now a more geometric characterization of the groups $S^\r_{F(\bar \p)[0]}$ for $(\aa,\r)$ a symmetric pair.  Recall from \eqref{ff}  the definition of $f:\Omega_{\g}\to P^\vee$.
\begin{prop}\label{integral}  If $\r\hookrightarrow so(n,\C)$ is a conformal embedding, and $(\aa,\r)$ is the associated symmetric pair, we have
\begin{equation}\label{H}
S^\r_{F(\bar \p)[0]}=\{(0,x)\in \Omega_{\mathfrak z({\r})} \times\Omega_{[\r,\r]}\mid f(x)(\D(\p))\subset \ganz\}.\end{equation}
In particular, if $\r$ is semisimple, $S^\r_{F(\bar \p)}$ has index $a_p$ in $\Omega_\r$.
 \end{prop}
\begin{proof}We can assume that the pair $(\aa,\r)$ is irreducible. Let $\tilde H$ denote the right hand side of \eqref{H}.

Assume first that $\r$ is semisimple. Then \eqref{H} turns into
\begin{equation*}\label{HH}
S^\r_{F(\bar \p)}=\{x\in \Omega_\r\mid f(x)(\D(\p))\subset \ganz\}.\end{equation*}

 By \eqref{weights}, we have that either $\l(\D(\mathfrak p))\subset \ganz$ or
 $\l(\D(\mathfrak p))\subset \frac{1}{2}+\ganz, (\l\in P^\vee_\r)$.  Also note that $\l(\D(\mathfrak p))\subset \ganz$ if and only if $\l(\theta_\p)\in \ganz$.  If $a_p=1$, then clearly $\tilde H=\Omega_\r$. If $a_p=2$, consider the map from $\Omega_\r=P^\vee_\r/Q^\vee_\r$ to $\frac{1}{2}\ganz/\ganz$ given by $x+Q^\vee_\r\mapsto \theta_\p(x)+\ganz$.
 Recall that in any affine Dynkin diagram there is a simple root  with label $1$, hence the map is onto. Since  $\tilde H$ is the kernel of the map, it  is a subgroup of $\Omega_\r$ of index $2$.

 By our description of $H_\r$, we have  that $H_{\r}\subset \tilde H$.
 This is clear from the proof of Proposition \ref{simplecurrentsinFp} for elements of types (2) and (3) and follows from \cite[Lemma 5.7]{IMRN} for $x_0$.
 Viceversa, assume $h\in\tilde H$. Then $h=\prod_Sh_S$, with $h_S=o_i$, $i\in S$ and $a_i^S=1$ or $h_S=\mathbbm 1$. Moreover $\sum_{h_S\ne \mathbbm 1}\frac{a_i}{a_p}\in \ganz$. If $a_i\in\{1,2\}$ for all $i$ and $a_p=2$, then $h$ is a product of elements of $H_{\r}$ hence, since $H_{\r}$ is a group, $h \in H_{\r}$. This rules out all the classical untwisted cases. The same argument works when $a_p=1$ and $a_i=a_i^S=1$ for all $i$. This rules out the adjoint case and $D^{(2)}_{l+1}$. The exceptional cases and the remaining twisted cases are dealt with by a direct inspection.
 \par
 Assume now $\r$ not semisimple.
For  $k\in \nat$ set $\Phi_k=\{\a\in\Dp(\p)\mid (\a-\a_\p)(f(x))=k\}$.  Set $T=a_s$ if $x=o_s$ and $T=a_s+a_{s'}$ if $x=o_so_{s'}$. Since $0=\frac{1}{2}\nu(\varkappa)(f(x))=\sum_{\a\in\Dp(\p)}\a(f(x))$ we have that $
 0=\frac{\dim\p}{2}\a_\p(f(x))+\sum_{k=0}^{T}k|\Phi_k|
 $ so
\begin{equation}\label{ap(h)}
 \a_\p(f(x))=-\frac{\sum_{k=0}^Tk|\Phi_k|}{\sum_{k=0}^T|\Phi_k|}.
\end{equation}

  Clearly
 $$
 S^\r_{F(\bar\p)[0]}=\{(n,x) \in H_{\r}\mid n=0\}.
 $$
 We now prove that $ S^\r_{F(\bar\p)[0]}\subset \tilde H$.
  If $x=\uno$ then $f(x)=0$ and there is nothing to prove. We can therefore assume that $x\ne0$.  By the definition of $H_\r$ we have that $\langle\Phi_x\rangle(\varkappa)\equiv 0\mod \tfrac{\dim\p}{2}$. As shown in Proposition \ref{simplecurrentsinFp}, there is $\ll\in\Sigma'$ such that $\Phi_x=\Phi_\ll$. As shown in the proof of  Theorem \ref{decoclassical}, $|\langle\Phi_\ll\rangle(\varkappa)|<\frac{\dim \p}{2}$, so $\langle\Phi_\ll\rangle(\varkappa)= 0$. It is clear from \eqref{weights2} that $h(\D(\mathfrak p))\subset \ganz$ if and only if $h(\a_\p)\in \ganz$. We now check that $\a_\p(f(x))=-1$.
Note that $\Phi_\ll^+=\Phi_2$ and $\Phi_\ll^-=-\Phi_0$, hence, since $\langle \Phi_\ll\rangle(\varkappa)=|\Phi^+|-|\Phi^-|=0$, we have that
$$
\a_\p(f(x))=-\frac{2|\Phi_\ll^+|+|\Phi_1|}{|\Phi_\ll^+|+|\Phi_1|+|\Phi_\ll^-|}=-1.
$$

Finally we prove that $\tilde H\subset S^\r_{F(\bar\p)[0]}$. Choose $(0,x)\in \tilde H$.
We are assuming that $\a_\p(f(x))\in \ganz$, hence, by \eqref{ap(h)},
$$
\sum_{k=0}^Tk|\Phi_k|\equiv 0 \mod \frac{\dim\p}{2}.
$$
Assume first $T=1$, so $x=o_s$ for some $s\in S$ and $a_s=1$. Then, by \eqref{ap(h)}, $\a_\p(\omega^\vee_s)>-1$ hence $\a_\p(f(x))=0$. This implies $\Phi_2=\Phi_1=\emptyset$, so $\Phi_0=\Dp(\p)$. In this case $\langle \Phi_x\rangle(\varkappa)=-|\Phi_0|=\frac{\dim\p}{2}$, hence $(0,x)\in H_\r$.
Now assume $T=2$. Then
$$
2|\Phi_2|+|\Phi_1|\equiv 0 \mod \frac{\dim\p}{2}.
$$
Since $2|\Phi_2|+|\Phi_1|=|\Phi_2|-|\Phi_0|+\frac{\dim\p}{2}$ and $\langle \Phi_x\rangle(\varkappa)=|\Phi_2|-|\Phi_0|$, we are done also in this case.
\end{proof}

\begin{cor}\label{cori} If $\k\hookrightarrow \g$ is a conformal embedding of a semisimple Lie algebra $\k$ in $\g=sp(n, \C)$ or $\g=sl(n+1,\C)$ and $(\aa,\r)$ is the associated symmetric pair, then
$$S^{\k}_{V_{1}(\g)}=\{x\in \Omega_\k\mid f(x)(\D(\p)_{|\ft})\subset \ganz\}.$$
\end{cor}
\begin{proof}
If $x\in S^{\k}_{V_{1}(\g)}$ then, by Proposition \ref{centro}, $x\in H_{\r}$. Recall that $\r=\u\times \k$. Since $\k$ is semisimple, then $\mathfrak z(\r)\subset \u$, hence $L^\u(0)\otimes x$ occurs in $F(\bar\p)[0]$. By Proposition \ref{integral}, $f(x)(\D(\p)_{|\ft})=f(x)(\D(\p))\subset \ganz$.

Viceversa, if $x\in \Omega_\k$ is such that $f(x)(\D(\p)_{|\ft})\subset \ganz$, then, by Proposition \ref{integral}, $(0,x)\in H_{\r}$, hence, by Proposition \ref{centro}, $x\in S^\k_{V_1(\g)}$.
\end{proof}
\begin{rem} Observe that in all cases $\D(\p)_{|\ft}$ is the set of weights of the standard representation of $\g$ when restricted to $\k$.
\end{rem}
\subsubsection*{Acknowledgments}

We would like to thank Christian Krattenthaler for providing references about the evaluation of the determinant \eqref{Kra}, and the referee for useful suggestions. \par
Feng Xu is  partially
supported by NSF.
\providecommand{\bysame}{\leavevmode\hbox to3em{\hrulefill}\thinspace}
\providecommand{\href}[2]{#2}

\vskip10pt
\footnotesize{
\noindent{\bf V.K.}: Department of Mathematics, MIT, 77
Mass. Ave, Cambridge, MA 02139;
{\tt kac@math.mit.edu}

\noindent{\bf P.MF.}: Politecnico di Milano, Polo regionale di Como,
Via Valleggio 11, 22100 Como,
Italy; {\tt pierluigi.moseneder@polimi.it}

\noindent{\bf P.P.}: Dipartimento di Matematica, Sapienza Universit\`a di Roma, P.le A. Moro 2,
00185, Roma, Italy; {\tt papi@mat.uniroma1.it}

\noindent{\bf F.X.}: Department of Mathematics, University of California at Riverside, Riverside, CA 92521, United States;
 {\tt xufeng@math.ucr.edu}
}

 \end{document}